\documentclass[a4paper,7pt]{report}

\usepackage[english]{babel}
\usepackage{lmodern} %Type1-font for non-english texts and characters
\usepackage[utf8]{inputenc}
\usepackage[usenames]{color}
\usepackage{pifont}
\usepackage{palatino,url}
\usepackage[pagewise]{lineno}%\linenumbers
\usepackage[]{algorithm2e}
\usepackage{amsmath}
\usepackage{amssymb}
\usepackage{amsthm}
\usepackage{amsfonts}
\usepackage{graphicx}
\usepackage{epstopdf}
\usepackage{mathtools}
\usepackage{multirow}
\usepackage{paralist}
\usepackage[colorlinks=true]{hyperref}
\usepackage{sidecap}
\usepackage[mathscr]{euscript}
\hypersetup{urlcolor=blue, citecolor=red}
\DeclareMathOperator{\erf}{erf}		% error function
\DeclareMathOperator{\lcm}{lcm}		% least common multiple
\DeclareMathOperator{\sgn}{sgn}		% sign function
\DeclareMathOperator{\spann}{span}	% span
\newcommand{\tm}{\times}
\newcommand{\C}{{\mathbb C}}

\newcommand{\N}{{\mathbb N}}

\newcommand{\R}{{\mathbb R}}
\renewcommand{\S}{{\mathbb S}}
\newcommand{\Z}{{\mathbb Z}}
\newcommand{\cB}{{\mathcal B}}
\newcommand{\sF}{{\mathscr F}}
\newcommand{\sG}{{\mathscr G}}
\newcommand{\sK}{{\mathscr K}}
\newcommand{\sU}{{\mathscr U}}
\newcommand{\fA}{{\mathfrak A}}
\newcommand{\fC}{{\mathfrak C}}
\renewcommand{\d}{\,{\mathrm d}}
\newcommand{\vphi}{\varphi}
\newcommand{\eps}{\varepsilon}
\newcommand{\fall}{\quad\text{for all }}
\newcommand{\on}{\quad\text{on }}
\newcommand{\abs}[1]{\left|#1\right|}
\newcommand{\norm}[1]{\left\|#1\right\|}
\newcommand{\set}[1]{\left\{#1\right\}}
\newcommand{\intcc}[1]{\left[#1\right]}
\newcommand{\intoo}[1]{\left(#1\right)}
\newcommand{\iprod}[1]{\langle\!\!\langle#1\rangle\!\!\rangle}		% inner product
\newcommand{\sprod}[1]{\langle#1\rangle}		% inner product
\newcommand{\svector}[2]{\bigl[\begin{smallmatrix}#1\\ #2\end{smallmatrix}\bigr]}
\newtheorem{theorem}{Theorem}[section]
\newtheorem{corollary}[theorem]{Corollary}
\newtheorem{lemma}[theorem]{Lemma}
\newtheorem{prop}[theorem]{Proposition}

\theoremstyle{definition}

\theoremstyle{remark}
\newtheorem{example}{Example}[section]

\newtheorem{remark}{Remark}[section]
\newcommand{\cref}[1]{Cor.~\ref{#1}}

\newcommand{\eref}[1]{Ex.~\ref{#1}}
\newcommand{\fref}[1]{Fig.~\ref{#1}}
\newcommand{\lref}[1]{Lemma~\ref{#1}}
\newcommand{\pref}[1]{Prop.~\ref{#1}}
\newcommand{\rref}[1]{Rem.~\ref{#1}}
\newcommand{\sref}[1]{Sect.~\ref{#1}}
\newcommand{\tref}[1]{Thm.~\ref{#1}}
\numberwithin{equation}{section}
\renewcommand{\theequation}{\arabic{section}.\arabic{equation}}
\allowdisplaybreaks
\begin{document}
%\linenumbers
%
%\title[Bifurcations in periodic integrodifference equations I]{Bifurcations in periodic integrodifference equations in $C(\Omega)$ I:\\ Analytical results and applications}
\title{Bifurcations in periodic integrodifference equations in $C(\Omega)$ I:\\ Analytical results and applications}
%
%\author{ \\ }
\author{Christian Aarset \\ christian.aarset@aau.at 
   \and Christian P\"otzsche \\ christian.poetzsche@aau.at \\ [1ex]
   Institut f\"ur Mathematik, Universit\"at Klagenfurt \\
	 Universit\"atsstra{\ss}e 65--67, 9020 Klagenfurt, Austria}
%\address{Christian Aarset and Christian P\"otzsche, }
%\email{christian.aarset@aau.at \\ christian.poetzsche@aau.at}
%
%\author{}
%
%\keywords{Integrodifference equation, periodic difference equation, duality pairing, fold bifurcation, crossing curve bifurcation, flip bifurcation}
%\subjclass{37G15; 45G15; 39A30; 39A28; 39A23; 92D25}
%

%\begin{abstract}
%	We study local bifurcations of periodic solutions to time-periodic (systems of) integrodifference equations over compact habitats. Such infinite-dimensional discrete dynamical systems arise in theoretical ecology as models to describe the spatial dispersal of species having nonoverlapping generations. Our explicit criteria allow us to identify branchings of fold- and crossing curve-type, which include the classical transcritical-, pitchfork- and flip-scenario as special cases. Indeed, not only tools to detect qualitative changes in models from e.g.\ spatial ecology and related simulations are provided, but these critical transitions are also classified. In addition, the bifurcation behavior of various time-periodic integrodifference equations is investigated and illustrated. This requires a combination of analytical methods and numerical tools based on Nystr\"om discretization of the integral operators involved. 
%\end{abstract}

%
\maketitle
%
%\tableofcontents
%
%
%
\section{Introduction}
Integrodifference equations (abbreviated IDEs) are infinite-di\-mensional dynamical systems in discrete time. They act as iterates of integral operators of e.g.\ Urysohn- or Hammerstein-type on an ambient function space. Beyond early applications in population genetics \cite{slatkin:73}, the latter kind has recently enjoyed wide popularity in spatial ecology \cite{kot:schaeffer:86,lutscher:19}, where it is employed in order to describe the dispersal and growth of populations with nonoverlapping generations over some habitat $\Omega$. Here, the nonlinearity is commonly given by a growth function $g$ of e.g.\ Beverton-Holt, Ricker or logistic type (see \cite{brauer:castillo:01}). Dispersal is realized by means of an integral operator with a kernel $k$. If growth precedes dispersal, one obtains a recursion of Hammerstein form \cite{kot:schaeffer:86,hardin:takac:webb:88,hardin:takac:webb:90}
\begin{equation}
	u_{t+1}(x)=\int_\Omega k(x,y)g(u_t(y),\alpha)\d y\fall x\in\Omega,
	\label{ide1}
\end{equation}
while species first dispersing and then growing can be modeled via \cite{andersen:91,lutscher:petrovskii:08}
\begin{equation}
	u_{t+1}(x)=g\intoo{\int_\Omega k(x,y)u_t(y)\d y,\alpha}\fall x\in\Omega.
	\label{ide2}
\end{equation}
In both cases, the values $u_t(x)\geq 0$ describe the population density at a point $x$ in the habitat $\Omega$ for the $t$-th generation, resp.\ the density vector when various populations interact. 
In such applications modeling dispersal, an advantage of IDEs compared to e.g.\ reaction-diffusion systems is their flexibility. Specifying a particular kernel $k$ allows one to incorporate a variety of different dispersal strategies \cite{hardin:takac:webb:88}. 

Being models in the life sciences, IDEs require detailed information on the robustness and destabilization of solutions as well as the accompanying bifurcations when the environment is varied. Indeed, various papers point out the occurrence of bifurcations in terms of period doubling cascades for Ricker \cite{andersen:91} and logistic nonlinearities \cite{kot:schaeffer:86,kirk:lewis:97}, a fold bifurcation \cite{kirk:lewis:97} for Allee growth, but also the alternately transcritical and pitchfork bifurcations along the trivial branch for Beverton-Holt growth functions \cite{kirk:lewis:97}. While these observations are often based on simulations, for instance \cite[Lemma 4, 5]{lutscher:lewis:04} prove a criterion for transcritical bifurcations in matrix models; global bifurcations are addressed in \cite{alzoubi:10,robertson:cushing:12}. Finally, \cite{bramburger:lutscher:18} give an explicit analysis of a logistic model equipped with a degenerate kernel. Here, the IDEs reduce to polynomial difference equations in $\R^2$ such that classical criteria apply. 

The paper at hand provides a systematic approach to identify and verify (local) bifurcations of codimension $1$ in periodic IDEs on compact habitats $\Omega$. It extends the above contributions in various aspects: 
\begin{itemize}
	\item Seasonality is an important issue in applications from the life sciences and for earlier work on periodic bifurcations in ecology we refer to \cite{bacaer:09,bacaer:dats:12,cushing:ackleh,cushing:henson} or specifically to \cite{zhou:fagan:17} dealing with IDEs. The contribution at hand generalizes the more theoretically oriented references \cite{luis:elaydi:oliveira:12} (equations in $\R$ via period map) or \cite{poetzsche:12} (finite-dimensional systems as operator equations in the space of periodic sequences) to an infinite-dimensional set-up suitable for IDEs. 
	
	\item Our setting applies to models typically studied in theoretical ecology (scalar IDEs \cite{andersen:91,hardin:takac:webb:90,kot:schaeffer:86,reimer:bonsall:maini:16,kirk:lewis:97,zhou:fagan:17}, systems \cite{bramburger:lutscher:18,kot:schaeffer:86}, structured \cite{alzoubi:10,lutscher:lewis:04,robertson:cushing:12} etc.). The considered problem class includes both growth-dispersal equations \eqref{ide1} and dispersal-growth IDEs \eqref{ide2} as special cases although they are conjugated as follows: The right-hand sides are compositions of integral operators $\sK$ (with dispersal kernel $k$) and superposition operators $\sG$ (induced by the growth function $g$). Now for solutions $u_t$ of \eqref{ide1} the sequence $\sG\circ u_t$ solves \eqref{ide2} and conversely, given a solution $v_t$ to \eqref{ide2}, $\sK\circ v_t$ satisfies \eqref{ide1}. In this sense also bifurcating objects for \eqref{ide1} and \eqref{ide2} are conjugated and simply capture different consensus times \cite{lutscher:petrovskii:08}. Nonetheless, both kinds of equations are investigated in the literature and we aim for a flexible, generally applicable approach. 

	\item We work with IDEs involving a general measure-theoretical integral based on a finite measure $\mu$. This unifies the analysis for ``classical IDEs'' (where $\mu$ is the Lebesgue measure) with their Nystr{\"o}m discretizations (important in simulations), and also applies to so-called \emph{metapopulation models} for dispersal between finitely many patches \cite{kot:schaeffer:86} as well as finite-dimensional difference equations \cite{poetzsche:12} (we refer to \eref{exmeasure} for details). Extending this, general difference equations on the space of measures, as well as their applications in population dynamics, were recently studied in \cite{thieme:20}.
\end{itemize}
Finally, our analysis restricts to those solutions whose period is a multiple of the basic period of the difference equation. Other constellations of these two periods require a possibly large codimension as pointed out in \cite{beyn:huels:samtenschnieder:08}.
\subsection{Bifurcations in integrodifference equations}
First of all, when dealing with periodic equations the bifurcating objects are periodic solutions (denoted as \emph{cycles} in \cite{cushing:henson}) rather than fixed points, and the period reflects the seasonal driving. As a rule of thumb, IDEs inherit their bifurcation behavior from the scalar (autonomous) models
\begin{equation}
	u_{t+1}=g(u_t,\alpha)
	\label{ide3}
\end{equation}
given by their growth function $g$. However, the critical values for the parameter $\alpha$ are typically larger due to the loss of population when leaving the bounded habitat $\Omega$ (see \cite{reimer:bonsall:maini:16}). In order to list further similarities and discrepancies between \eqref{ide3} and \eqref{ide1},\eqref{ide2}, two kinds of behavior can be observed:
\begin{itemize}
	\item For monotone growth functions like in the Beverton-Holt case, \cite[Thms.~3.3 and 4.9]{thieme:79} or \cite{hardin:takac:webb:88,hardin:takac:webb:90,zhou:fagan:17} supplemented by strict subhomogeneity \cite{zhao:17} or concavity \cite{krause:15}, global attractivity results hold. Every solution in the cone of nonnegative functions converges to zero, or there exists a unique nonzero globally attractive periodic solution (see Figs.~\ref{figlaplace1} and \ref{figlaplace1s}). These are the only biologically meaningful solutions although, in \eqref{ide1}, \eqref{ide2}, there are countably many more bifurcations along the trivial solution, whose branches consist of functions having negative values and are biologically irrelevant (cf.~\sref{sec52}). 

	\item For non-monotone growth (logistic, Hassell, Ricker) the dynamics of nonnegative solutions is much more complicated \cite{day:junge:mischaikow:04} featuring e.g.\ a period doubling scenario (cf.~\sref{sec53}). Adding dispersal yields even richer dynamics since, different from scalar equations \eqref{ide3}, for instance along the primary nontrivial branch, more than just one period doubling bifurcation occurs in \eqref{ide1} or \eqref{ide2} (see \fref{figgauss2}). In addition, the period doublings can lead to positive solutions with an increasing spatial inhomogeneity \cite{andersen:91,kot:schaeffer:86,kirk:lewis:97}. 
\end{itemize}
\subsection{Contents}
In a quite general framework of periodic IDEs, the paper at hand provides sufficient conditions for local continuation and various branchings of periodic solutions. In practise these criteria must be verified numerically --- particularly when dealing with realistic examples. We accordingly formulate our assumptions such that they can be tackled effectively using numerical tools. 

Although our proofs are purely analytical, we always provide interpretations in terms of dynamical systems, formulate our conclusions accordingly and hint at ecological interpretations. The subsequent \sref{sec2} contains basic notions and ingredients to understand the local dynamics near periodic solutions of periodic IDEs. Although the methods of \cite{poetzsche:12} developed for systems of difference equations in $\R^d$ formally extend to IDEs, we choose a different approach: Rather than fixed points of the period map \cite{luis:elaydi:oliveira:12}, we directly compute periodic solutions. This has at least two advantages when it comes to numerics: Evaluating the period map of a $\theta$-periodic IDE over a domain $\Omega\subset\R^\kappa$ involves $\kappa\theta$-fold integrals which, in turn, require high-dimensional cubature rules. Instead of solving periodic eigenvalue problems, we employ cyclic block operators (see \pref{propspec}), whose evaluation avoids (numerical) stability issues (see e.g.\ \cite[pp.~291ff]{watkins:07}). Under a weak hyperbolicity condition, we provide a persistence (continuation) result for periodic solutions in \sref{sec3}. As illustration serves the effect of dispersal when added to the Allee equation. Afterwards, fold and \emph{crossing curve bifurcations} are addressed in \sref{sec4}. The aforesaid crossing curve bifurcations include classical transcritical and pitchfork bifurcation scenarios as special cases, but the existence of a known solution branch is not assumed. We also give sufficient conditions for stability changes along bifurcating solution branches in terms of their Morse index. By doubling the particular period length, flip bifurcations can be naturally embedded into our framework for pitchfork bifurcations. Finally, certain symmetry properties are discussed, demonstrating that for e.g.\ Hammerstein operators with symmetric kernel, the solution branches bifurcating off the trivial branch either consist of even functions, or appear as pairs of solution branches sharing the same total population.

The concrete applications covered in \sref{sec5} not only illustrate our theoretical results, but additionally demonstrate their applicability. Being scalar IDEs over a symmetric interval as habitat (i.e.\ $\kappa=1$), they are admittedly simple, but nevertheless exhibit essential features and allow numerical simulations of reliable accuracy. The required discretizations are based on Nystr\"om methods \cite{atkinson:92,atkinson:97,kress:14}, that is, one replaces integrals by quadrature rules. Since the integrals become finite weighted sums, this results in parametrized finite-dimensional difference equations approximating the integral operators in \eqref{ide1} and \eqref{ide2}. Due to \eref{exmeasure}(2), the branching behavior of these discretizations is also covered by our abstract framework. The examples are arranged such that their analysis requires an increasing numerical effort: \sref{sec51} discusses IDEs having a degenerate kernel, essentially reducing the problems to finite dimensions. This allows an explicit analysis of fold and $2$-periodic pitchfork bifurcations. While the related paper \cite{bramburger:lutscher:18} studies the (equivalent) finite-dimensional version of a spatial logistic growth model, we work directly with IDEs having polynomial growth functions without ecological motivation. For periodic spatial Beverton-Holt equations of both Hammerstein- and dispersal-growth type, bifurcations along the trivial solution are studied in \sref{sec52}. It is shown that the primary bifurcation is transcritical at a critical parameter value related to the basic reproduction number $R_0$ of the species. The countably many further (but biologically irrelevant) bifurcations along $0$ alternate between being of pitchfork and transcritical type, which confirms a numerical observation from \cite{kirk:lewis:97}. In particular, if we equip the considered IDEs with the Laplace kernel, then the computational effort of our analysis reduces to the numerical solution of a transcendental equation in the reals. Related results concerning branchings along the zero solution also hold for an autonomous Ricker IDE tackled in \sref{sec53}. Whence, our focus is a period doubling scenario along nontrivial solution branches. This eventually requires a fully numerical analysis based on path-following schemes to detect also unstable branches of periodic solutions, corresponding eigenvalue computations and an approximate evaluation of integrals. 

For the reader's convenience we added two appendices. First, App.~\ref{appA} contains the necessary abstract bifurcation results initiated in \cite{crandall:rabinowitz:71,crandall:rabinowitz:73}. We rely on a criterion for \emph{crossing curve bifurcations} due to \cite{shi:99,liu:shi:wang:07} in \tref{thmcrosbif}, which not only contains transcritical and pitchfork patterns as special cases, but also applies without assuming a given known solution branch. Supplementing \cite{shi:99,liu:shi:wang:07}, an ``exchange of stability principle'' for \tref{thmcrosbif} is provided, covering the behavior of the critical eigenvalue $1$ along the bifurcating branches. Second, one goal of this paper is to provide an analytical justification for bifurcations observed in computer simulations of IDEs. This still requires to approximate eigenvalues of integral operators or to compute solution branches by e.g.\ pseudo-arc length continuation. App.~\ref{appB} sketches the central algorithms for this purpose and surveys some of the related literature. Implementations in Matlab are available for interested readers. 
\subsection{Notation}
Let $\R_+:=[0,\infty)$ denote the nonnegative reals and $\S^1$ the unit circle in $\C$. On a real Banach space $X$, $B_r(x):=\set{y\in X:\,\norm{y-x}<r}$ is the open ball with center $x$ and radius $r>0$; $\bar B_r(x)$ is its closure. On the Cartesian product $X\tm Y$ with another real Banach space $Y$ we use the norm 
$$
	\norm{(x,y)}:=\max\set{\norm{x}_X,\norm{y}_Y}.
$$
Moreover, $L_l(X,Y)$, $l\in\N$, is the linear space of bounded $l$-linear operators from $X^l$ to $Y$; $L_0(X,Y):=Y$, $L(X,Y):=L_1(X,Y)$, $L(X):=L(X,X)$ and $I_X$ is the identity map on $X$. We write $\sigma(T)$ for the \emph{spectrum}, $\sigma_p(T)\subseteq\sigma(T)$ for the \emph{point spectrum} of (the complexification of) $T\in L(X)$, and refer to $(\lambda,x)\in\C\tm X$ as \emph{eigenpair} of $T$ if $x\neq 0$ and $Tx=\lambda x$ hold. The null space and range of $T$ are $N(T):=T^{-1}(0)$ resp.\ $R(T):=TX$. 

The spaces $X,Y$ equipped with a bilinear form $\iprod{\cdot,\cdot}:Y\tm X\to\R$ satisfying
\begin{align*}
	\forall x\in X\setminus\set{0}:\exists y\in Y:\,\iprod{y,x}\neq 0,\quad
	\forall y\in Y\setminus\set{0}:\exists x\in X:\,\iprod{y,x}\neq 0
\end{align*}
are called a \emph{duality pairing} $\iprod{Y,X}$ (cf.~\cite[pp.~45ff]{kress:14}). One speaks of a \emph{bounded} bilinear form if there exists a $C\geq 0$ such that $\abs{\iprod{y,x}}\leq C\norm{x}\norm{y}$ for all $x\in X$, $y\in Y$. The \emph{annihilator} of a subspace $X_0\subseteq X$ is denoted by
$$
	X_0^\perp:=\set{y\in Y:\,\iprod{y,x}=0\text{ for all }x\in X_0}. 
$$
Given an operator $T\in L(X)$, its \emph{dual operator} $T'\in L(Y)$ is uniquely determined by $\iprod{y,Tx}=\iprod{T'y,x}$ for all $x\in X$, $y\in Y$ \cite[p.~46, Thm.~4.6]{kress:14}. 

In case $U\subseteq X$ is an (open) subset, then $C^m(U,Y)$ consists of all mappings $f:U\to Y$ whose $m$-th Fr{\'e}chet derivative exists and is continuous, where $m\in\N_0$. 

Norms on finite-dimensional spaces are denoted by $\abs{\cdot}$. In particular, on $\R^d$ we exclusively use the norm induced by the inner product $\sprod{y,x}:=\sum_{j=1}^dx_jy_j$. The space $L(\R^m,\R^d)$ is canonically identified with the $d\tm m$-matrices $\R^{d\tm m}$, and the transpose of $K\in\R^{d\tm m}$ is denoted by $K^T\in\R^{m\tm d}$. 
\section{Periodic difference equations}
\label{sec2}
We first present the basics necessary to keep this paper self-contained, but also to carve out differences and extensions to the finite-dimensional situation \cite{poetzsche:12}. Above all, periodic solutions of difference equations are characterized as zeros of a cyclic operator, to which the abstract bifurcation results from App.~\ref{appA} apply. This is based on the subsequent preparations on differentiability. We moreover state some stability criteria and develop an ambient Fredholm theory. The latter relies on duality pairings (cf.~\cite[pp.~45ff]{kress:14} or \cite[pp.~303ff]{zeidler:95}). 

Abstractly, we are interested in (nonautonomous) difference equations
\begin{equation}
	\tag{$\Delta_\alpha$}
	\boxed{u_{t+1}=\sF_t(u_t,\alpha),}
	\label{deq}
\end{equation}
which depend on a parameter $\alpha\in A$ from an ambient set $A$ to be specified later. The right-hand sides $\sF_t:U_t\tm A\to X$ are defined on subsets $U_t\subseteq X$, $t\in\Z$. A sequence $\phi=(\phi_t)_{t\in\Z}$ satisfying $\phi_t\in U_t$ and $\phi_{t+1}=\sF_t(\phi_t,\alpha)$ for all $t\in\Z$ is called an \emph{entire solution} of \eqref{deq} and, given $\eps>0$, the set
$$
	\cB_\eps(\phi):=\set{(t,x)\in\Z\tm X:\,\norm{x-\phi_t}<\eps}
$$
is its $\eps$-\emph{neighborhood}. One typically identifies $\phi$ and $\set{(t,\phi_t)\in\Z\tm X:\,t\in\Z}$. Keeping $\alpha\in A$ fixed, the solution of \eqref{deq} starting at an initial time $\tau\in\Z$ in the initial state $u_\tau\in U_\tau$ is given by 
\begin{equation*}
	\vphi_\alpha(t;\tau,u_\tau)
	:=
	\begin{cases}
		\sF_{t-1}(\cdot,\alpha)\circ\ldots\circ \sF_\tau(\cdot,\alpha)(u_\tau),&\tau<t,\\
		u_\tau,&t=\tau,
	\end{cases}
%	\label{cocy}
\end{equation*}
as long as the compositions stay in $U_t$. We call $\vphi_\alpha$ the \emph{general solution} of \eqref{deq}. 

The paper focuses on periodic difference equations \eqref{deq}, i.e.\ the situation where there exists a \emph{basic period} $\theta_0\in\N$ such that
\begin{align*}
	\sF_{t+\theta_0}=\sF_t:U_t\tm A&\to X,&
	U_{t+\theta_0}=U_t\fall t\in\Z.
\end{align*}
In case $\theta_0=1$, the right-hand sides $\sF_t$ and the subsets $U_t$ are constant in $t$, and one obtains an \emph{autonomous} difference equation \eqref{deq}. 
\subsection{The cyclic operator $G$}
In order to characterize periodic solutions of \eqref{deq} the following notions are fundamental: Given $\theta\in\N$, the linear space of $\theta$-periodic sequences in $X$, 
$$
	\ell_\theta(X):=\set{\phi=(\phi_t)_{t\in\Z}:\,\phi_t\in X\text{ and }\phi_{t+\theta}=\phi_t\text{ for all }t\in\Z},
$$
is equipped with the norm $\norm{\phi}:=\max_{t=0}^{\theta-1}\norm{\phi_t}_X$. We identify $\ell_\theta(X)$ with the $\theta$-fold product $X^\theta$, as these spaces are (topologically) isomorphic by means of the mutually inverse maps $\phi\mapsto\hat\phi:=(\phi_0,\ldots,\phi_{\theta-1})$ and $(\phi_0,\ldots,\phi_{\theta-1})\mapsto(\ldots,\underline{\phi_0},\ldots,\phi_{\theta-1},\ldots)$, the underline indicating the element of index $0$ in a sequence $\phi\in\ell_\theta(X)$. 

We restrict to such solutions, whose period $\theta$ is a multiple of the period $\theta_0$ of the difference equation \eqref{deq}. With the product $\hat U:=U_0\tm\ldots\tm U_{\theta-1}\subseteq X^\theta$ and the cyclic mapping
\begin{align}
	G:\hat U\tm A&\to X^\theta,&
	G(\hat u,\alpha)
	&:=
	\begin{bmatrix}
		\sF_{\theta-1}(u_{\theta-1},\alpha) - u_0\\
		\sF_0(u_0,\alpha) - u_1\\
		\vdots\\
		\sF_{\theta-2}(u_{\theta-2},\alpha) - u_{\theta-1}
	\end{bmatrix}, 
	\label{Gdef}
\end{align}
we arrive at the elementary, yet crucial
\begin{theorem}\label{thmmain}
	Let $\alpha\in A$, $\theta\in\N$ be a multiple of $\theta_0$ and $\phi\in\ell_\theta(X)$. Then $\phi$ is a solution of \eqref{deq} if and only if $G(\hat\phi,\alpha)=0$.
\end{theorem}
\begin{proof}
	The immediate proof is left to the suspicious reader. 
\end{proof}

\begin{prop}[properties of $G$]\label{propG}
	Let $m\in\N$ and $A$ be an open subset of a Banach space $P$. If each mapping $\sF_t:U_t\tm A\to X$, $0\leq t<\theta_0$, is $m$-times continuously differentiable, then the following holds for all $\hat u\in\hat U$ and $\alpha\in A$:
	\begin{enumerate}
		\item $G:\hat U\tm A\to X^\theta$ is of class $C^m$ with the partial derivatives
		\begin{align}
			D_1G(\hat u,\alpha)\hat v
			&=
			\begin{bmatrix}
				D_1\sF_{\theta-1}(u_{\theta-1},\alpha)v_{\theta-1} - v_0\\
				D_1\sF_0(u_0,\alpha)v_0 - v_1\\
				\vdots\\
				D_1\sF_{\theta-2}(u_{\theta-2},\alpha)v_{\theta-2} - v_{\theta-1}
			\end{bmatrix},
			\label{propG1}\\
			\hspace*{-10mm}
			D_1^iD_2^jG(\hat u,\alpha)\hat v^1\ldots\hat v^ip^1\ldots p^j
			&=
			\begin{bmatrix}
				D_1^iD_2^j\sF_{\theta-1}(u_{\theta-1},\alpha)v^1_{\theta-1}\cdots v^i_{\theta-1} p^1\cdots p^j\\
				D_1^iD_2^j\sF_0(u_0,\alpha)v^1_0\cdots v^i_0p^1\cdots p^j\\
				\vdots\\
				D_1^iD_2^j\sF_{\theta-2}(u_{\theta-2},\alpha)v^1_{\theta-2}\cdots v^i_{\theta-2}p^1\cdots p^j
			\end{bmatrix}
			\label{propG2}
		\end{align}
	for $i,j\in\N_0$, $1\leq i+j\leq m$, $(i,j)\neq(1,0)$ and $\hat v$, $\hat v^k\in X^\theta$, $p^k\in P$, $1\leq k\leq m$. 

		\item If every $D_1\sF_t(u_t,\alpha)\in L(X)$, $0\leq t<\theta$, is compact, then $D_1G(\hat u,\alpha)\in L(X^\theta)$ is a Fredholm operator of index $0$. 
	\end{enumerate}
\end{prop}
Note that continuous differentiability of $\sF_t$ in the assumptions of \pref{propG} is to be understood in the sense of Fr{\'e}chet for open sets $U_t$, or as cone differentiability when $U_t$ is a cone in $X$ (cf.~\cite[pp.~225--226]{deimling:85}). 
\begin{proof}
	(a) Thanks to e.g.\ \cite[pp.~246--247, Prop.~4]{zeidler:95} and \eqref{Gdef}, the differentiability properties of $\sF_t$ and of the component functions transfer to $G$. 

	(b) The partial derivative $D_1G(\hat u,\alpha)\in L(X^\theta)$ is Fredholm of index $0$, because due to \eqref{propG1} it is a compact perturbation of the identity, and thus \cite[p.~300, Thm.~5.E]{zeidler:95} implies the claim. 
\end{proof}
\subsection{Periodic linear equations, stability and Fredholm theory}
Our first goal is capturing the local dynamical behavior of \eqref{deq} near given branches $\phi(\alpha)$ of $\theta$-periodic solutions via their linearization. Thereto, suppose the partial derivatives $D_1\sF_t:U_t\tm A\to L(X)$, $t\in\Z$, exist. Given $\alpha\in A$, the linear difference equation
\begin{equation}
	\tag{$V_\alpha$}
	\boxed{v_{t+1}=D_1\sF_t(\phi(\alpha)_t,\alpha)v_t}
	\label{var}
\end{equation}
in $X$ is called the \emph{variational equation} (associated to the $\theta$-periodic solution $\phi(\alpha)$ of \eqref{deq}). It has the \emph{transition operator}
$$
	\Phi_\alpha(t,\tau)
	:=
	\begin{cases}
		D_1\sF_{t-1}(\phi(\alpha)_{t-1},\alpha)\cdots D_1\sF_\tau(\phi(\alpha)_\tau,\alpha),&
		\tau<t,\\
		I_X,&\tau=t,
	\end{cases}
$$
the \emph{period operator} $\Xi_\theta(\alpha):=\Phi_{\alpha}(\theta,0)$ and the \emph{Floquet spectrum} $\sigma_\theta(\alpha):=\sigma(\Xi_\theta(\alpha))$. The elements of $\sigma_\theta(\alpha)$ are called \emph{Floquet multipliers}. 

In this setting, the stability properties of a solution $\phi(\alpha)\in\ell_\theta(X)$ to \eqref{deq} are as follows:
\begin{itemize}
	\item $\sigma_\theta(\alpha)\subseteq B_1(0)$ if and only if $\phi(\alpha)$ is \emph{exponentially stable}, 
	i.e.\ there exist reals $K\geq 1$, $\gamma\in(0,1)$ and $\rho>0$ such that
	$$
		\norm{\vphi_\alpha(t;\tau,u_\tau)-\phi(\alpha)_t}
		\leq
		K\gamma^{t-\tau}\norm{u_\tau-\phi(\alpha)_\tau}
		\fall\tau\leq t,\,u_\tau\in\bar B_\rho(\phi(\alpha)_\tau)
	$$
	(see \cite{gyori:pituk:01}, \cite[p.~2, Thm.~1]{iooss:79}, \cite[Thm.~2.1(a)]{poetzsche:russ:19}),

	\item if there exists a decomposition $\sigma_\theta(\alpha)=\sigma_0\cup\sigma_u$ with $\sup_{\lambda\in\sigma_0}\abs{\lambda}<\inf_{\lambda\in\sigma_u}\abs{\lambda}$, $1<\inf_{\lambda\in\sigma_u}\abs{\lambda}$ and $\sigma_u\neq\emptyset$, then $\phi(\alpha)$ is unstable (see \cite[p.~3, Thm.~2]{iooss:79} or \cite[Thm.~2.1(b)]{poetzsche:russ:19}). 
\end{itemize}

In what follows, let us keep a parameter $\alpha^\ast\in A$ and a solution $\phi^\ast\in\ell_\theta(X)$ to $(\Delta_{\alpha^\ast})$ fixed. A spectral decomposition $\sigma_\theta(\alpha^\ast)=\sigma_s\dot\cup\sigma_c\dot\cup\sigma_u$ into disjoint closed sets $\sigma_s\subseteq B_1(0)$, $\sigma_c\subseteq\S^1$ and $\sigma_u\subseteq\C\setminus\bar B_1(0)$ gives rise to a decomposition
$$
	X=X_s\oplus X_c\oplus X_u
$$
of the space $X$ into closed subspaces $X_s,X_c$ and $X_u$ (cf.~\cite[p.~178, Thm.~6.17]{kato:80}). If $X_c\oplus X_u$ is finite-dimensional, then $m_\ast(\phi^\ast):=\dim X_u$ is called the \emph{Morse index} and $m^\ast(\phi^\ast):=\dim X_c\oplus X_u$ is the \emph{upper Morse index} of a solution $\phi^\ast$. In case
$$
	\sigma_\theta(\alpha^\ast)\cap{\mathbb S}^1=\emptyset,
$$
that is, $\sigma_c=\emptyset$, or equivalently, $X_c=\set{0}$, one denotes the periodic solution $\phi^\ast$ as \emph{hyperbolic}. In this situation we have $m_\ast(\phi^\ast)=m^\ast(\phi^\ast)$, and $\phi^\ast$ is exponentially stable if and only if $m_\ast(\phi^\ast)=0$. Thus, the Morse index is a measure of instability. 

In order to simplify our subsequent analysis, we suppose for the remaining section that every $D_1\sF_t(\phi(\alpha)_t,\alpha)\in L(X)$, $0\leq t<\theta_0$, is compact. We now address the question of how the Floquet spectrum is related to the point spectrum of the derivative of $G$: 
\begin{prop}\label{propspec}
	If $\theta\in\N$ is a multiple of $\theta_0$, then
	$$
		\sigma_p(D_1G(\hat\phi^\ast,\alpha^\ast))
		=
		\bigl\{\lambda-1\in\C:\,\lambda^\theta\in\sigma_p(\Xi_\theta(\alpha^\ast))\bigr\}.
	$$
	Moreover, for any $\lambda\in\C,\hat\xi\in X^\theta$, the following statements are equivalent:
	\begin{enumerate}
		\item $(\lambda-1,\hat\xi)$ is an eigenpair of $D_1G(\hat\phi^\ast,\alpha^\ast)$, 

		\item $(\lambda^\theta,\xi_0)$ is an eigenpair of $\Xi_\theta(\alpha^\ast)$, and $\lambda^t\xi_t=\Phi_{\alpha^\ast}(t,0)\xi_0$ for all $0\leq t<\theta$. 
	\end{enumerate}
\end{prop}
\begin{proof}
	The proof is a straightforward generalization of a well-known result for cyclic matrices, found in e.g. \cite[pp.~293--295]{watkins:07}. The argument is based on the fact that the power $[I_{X^\theta}+D_1G(\hat\phi^\ast,\alpha^\ast)]^\theta$ is block diagonal. 
\end{proof}

\begin{prop}\label{proplin}
	If $\theta\in\N$ is a multiple of $\theta_0$ and $\xi^\ast=(\xi_t^\ast)_{t\in\Z}\in\ell_\theta(X)$, then the following statements are equivalent:
	\begin{enumerate}
		\item $\xi^\ast$ is a solution of $(V_{\alpha^\ast})$, 

		\item $\hat\xi^\ast\in N(D_1G(\hat\phi^\ast,\alpha^\ast))$, 

		\item $\xi_t^\ast=\Phi_{\alpha^\ast}(t,0)\xi_0^\ast$ holds for all $0\leq t<\theta$. If $\xi_0^\ast\neq 0$, then $(1,\xi_0^\ast)$ is an eigenpair of $\Xi_\theta(\alpha^\ast)$.
	\end{enumerate}
	In particular, we have $\dim N(D_1G(\hat\phi^\ast,\alpha^\ast))=\dim N(\Xi_\theta(\alpha^\ast)-I_X)$. 
\end{prop}
\begin{proof}
	We suppress the dependence on the fixed parameter $\alpha^\ast$. 
	
	$(a)\Rightarrow(b)$ In particular, $\xi_{t+1-\theta}^\ast=\xi_{t+1}^\ast=D\sF_t(\phi_t^\ast)\xi_t^\ast$ for all $0\leq t<\theta$. Referring to the representation \eqref{propG1}, we now have 
	\begin{align*}
		(DG(\hat\phi^\ast)\hat\xi^\ast)_{t+1} &= D\sF_t(\phi_t^\ast)\xi^\ast_t - \xi^\ast_{t+1} = 0\fall 0\leq t<\theta,\\
		(DG(\hat\phi^\ast)\hat\xi^\ast)_0 &= D\sF_{\theta-1}(\phi_{\theta-1}^\ast)\xi^\ast_{\theta-1} - \xi^\ast_0 = D\sF_{\theta-1}(\phi_{\theta-1}^\ast)\xi^\ast_{\theta-1}-\xi^\ast_{\theta} = 0
	\end{align*}
	yielding $DG(\hat\phi^\ast)\hat\xi^\ast = 0$, as desired.

	$(b)\Rightarrow(c)$ The result is evident for $\xi^\ast=0$. If $\xi^\ast\neq 0$, then the result follows as a special case of \pref{propspec}. 

	$(c)\Rightarrow(a)$ Assume that $(1,\xi^\ast_0)$ is an eigenpair of $\Xi_\theta$ and that $\xi^\ast_t=\Phi(t,0)\xi^\ast_0$ for all $0\leq t<\theta$. For any $t\in\Z$, there exists a uniquely determined $k\in\Z$ such that $\tilde{t}:=t+k\theta\in\set{0,\ldots,\theta-1}$, and by the $\theta$-periodicity of $(\xi^\ast_t)_{t\in\Z}$ we have $\xi^\ast_t=\xi^\ast_{\tilde{t}}$ for every $t\in\Z$. The $\theta$-periodicity of $(D\sF_t(\phi_t^\ast))_{t\in\Z}$ now implies that
$
	\xi^\ast_{t+1} = \xi^\ast_{\tilde{t}+1} = \Phi(\tilde{t}+1,\tilde{t})\xi^\ast_{\tilde{t}} = D\sF_{\tilde{t}}(\phi_{\tilde{t}}^\ast)\xi^\ast_{\tilde{t}} = D\sF_t(\phi_t^\ast)\xi^\ast_{t}
$
whenever $0\leq\tilde{t}\leq\theta-2$, and
$$
	\xi^\ast_{t+1}
	\!=
	\xi^\ast_{\theta} = \xi^\ast_0 = \Xi_\theta\xi^\ast_0 = \Phi(\theta,\theta-1)\Phi(\theta-1,0)\xi^\ast_0 = D\sF_{\theta-1}(\phi_{\theta-1}^\ast)\xi^\ast_{\theta-1}
	= D\sF_t(\phi_t^\ast)\xi^\ast_t
$$
whenever $\tilde{t}=\theta-1$. It follows that $\xi^\ast$ is a $\theta$-periodic solution of $(V_{\alpha})$, as desired.\\
The final remark results from the fact that $\hat\xi^\ast$ is uniquely determined by $\xi^\ast_0$, since $\xi^\ast_t=\Phi(t,0)\xi^\ast_0$ for $0\leq t<\theta$. This implies that there is a one-to-one correspondence between elements $\xi^\ast_0\in N(\Xi_\theta-I_X)$ and $\hat\xi^\ast\in N(DG(\hat\phi))$, yielding the claim. 
\end{proof}

Suppose that we have given a duality pairing $\iprod{Y,X}$ such that the dual operator $D_1\sF_t(\phi(\alpha)_t,\alpha)'\in L(Y)$ of the derivative $D_1\sF_t(\phi(\alpha)_t,\alpha)$ exists for all $0\leq t<\theta$. This allows us to introduce the \emph{dual variational equation} (w.r.t.\ the duality pairing)
\begin{equation}
	\tag{$V_\alpha'$}
	\boxed{v_t=D_1\sF_t(\phi(\alpha)_t,\alpha)'v_{t+1},}
	\label{vard}
\end{equation}
which is a linear (backwards) difference equation in $Y$. Its \emph{dual transition operator} possesses the representation
\begin{equation}
	\Phi_\alpha'(t,\tau)
	:=
	\begin{cases}
		D_1\sF_\tau(\phi(\alpha)_\tau,\alpha)'\cdots D_1\sF_{t-1}(\phi(\alpha)_{t-1},\alpha)',&
		t>\tau,\\
		I_Y,&t=\tau
	\end{cases}
	\label{notransvar}
\end{equation}
and thus $\Phi_\alpha'(t,\tau)=\Phi_\alpha(t,\tau)'$, which results in the \emph{dual period operator}
$$
	\Xi_\theta'(\alpha):=\Phi_{\alpha}'(\theta,0)=\Phi_\alpha(\theta,0)'. 
$$
If we introduce the bilinear form
\begin{align}
	\iprod{\cdot,\cdot}_\theta:Y^\theta \tm X^\theta &\to\R,&
	\iprod{\hat y,\hat x}_\theta&:=\sum_{t=0}^{\theta-1}\iprod{y_t,x_t},
	\label{dpper}
\end{align}
then $\iprod{Y^\theta,X^\theta}_\theta$ becomes a duality pairing as well. Boundedness is inherited from $\iprod{\cdot,\cdot}$. As clearly $I_{X^\theta}'=I_{Y^\theta}$, the associate dual operator of $D_1G(\hat u,\alpha)$ is given by
\begin{align*}
	D_1G(\hat u,\alpha)'&\in L(Y^\theta),&
	D_1G(\hat u,\alpha)'\hat v
	=
	\begin{bmatrix}
		D_1\sF_0(u_0,\alpha)'v_1 - v_0\\
		\vdots\\
		D_1\sF_{\theta-2}(u_{\theta-2},\alpha)'v_{\theta-1} - v_{\theta-2}\\
		D_1\sF_{\theta-1}(u_{\theta-1},\alpha)'v_0 - v_{\theta-1}
	\end{bmatrix}
\end{align*}
for all $\hat u\in\hat U$ and tuples $\hat v\in Y^\theta$. 

The relationship between $(V_{\alpha^\ast}')$, $D_1G(\hat\phi^\ast,\alpha^\ast)'$ and $\Xi_\theta'(\alpha^\ast)$ resembles the relationship between $(V_{\alpha^\ast})$, $D_1G(\hat\phi^\ast,\alpha^\ast)$ and $\Xi_\theta(\alpha^\ast)$ described in \pref{proplin}: 
\begin{prop}\label{proplindual}
	If $\theta\in\N$ is a multiple of $\theta_0$ and $\eta^\ast=(\eta_t^\ast)_{t\in\Z}\in\ell_\theta(Y)$, then the following statements are equivalent:
	\begin{enumerate}
		\item $\eta^\ast$ is a solution of $(V_{\alpha^\ast}')$, 

		\item $\hat\eta^\ast\in N(D_1G(\hat\phi^\ast,\alpha^\ast)')$, 

		\item $\eta^\ast_t=\Phi_{\alpha^\ast}'(\theta,t)\eta_0^\ast$ holds for all $0\leq t<\theta$. If $\eta_0^\ast\neq 0$, then $(1,\eta_0^\ast)$ is an eigenpair of $\Xi_\theta'(\alpha^\ast)$. 
	\end{enumerate}
	In particular, we have $\dim N(D_1G(\hat\phi^\ast,\alpha^\ast)')=\dim N(\Xi_\theta(\alpha^\ast)'-I_Y)$. 
\end{prop}
\begin{proof}
	The proof is dual to that of \pref{proplin} and thus omitted. 
\end{proof}

\begin{prop}\label{prop26}
	If $\theta\in\N$ is a multiple of $\theta_0$, then the following holds: 
	\begin{enumerate}
		\item $\dim N(D_1G(\hat\phi^\ast,\alpha^\ast)) = \dim N(D_1G(\hat\phi^\ast,\alpha^\ast)') < \infty$, 

		\item $R(D_1G(\hat\phi^\ast,\alpha^\ast)) = \bigl\{\hat v\in X^\theta:\, \iprod{\hat\eta^\ast,\hat v}_\theta = 0\text{ for all }\hat\eta^\ast\in N(D_1G(\hat\phi^\ast,\alpha^\ast)')\bigr\}$.
	\end{enumerate}
\end{prop}
\begin{proof}
	The statement follows from the Fredholm alternative for duality pairings, as shown in e.g.\ \cite[p.~53, Thm.~4.15]{kress:14} (for (a)) and \cite[p.~55, Thm.~4.17]{kress:14} (for (b)). 
\end{proof}

Our Props.~\ref{propspec}--\ref{prop26} finally culminate in
\begin{corollary}\label{cordual}
	For any $n\in\N$, the following are equivalent:
	\begin{enumerate}
		\item $(V_{\alpha^\ast})$ has exactly $n$ linearly independent $\theta$-periodic solutions (up to multiples), 

		\item $\dim N(D_1G(\hat\phi^\ast,\alpha^\ast))=n$, 

		\item $\dim N(\Xi_\theta(\alpha^\ast)-I_X)=n$, 

		\item $(V_{\alpha^\ast}')$ has exactly $n$ linearly independent $\theta$-periodic solutions (up to multiples), 

		\item $\dim N(D_1G(\hat\phi^\ast,\alpha^\ast)')=n$, 

		\item $\dim N(\Xi_\theta'(\alpha^\ast)-I_Y)=n$.
	\end{enumerate}
	If any (and thus all) of the above hold with $N(D_1G(\hat\phi^\ast,\alpha^\ast)')=\spann\bigl\{\hat\eta^1,\ldots,\hat\eta^n\bigr\}$, then the linear functionals 
	\begin{align*}
		z_j': X^\theta&\rightarrow\R,&
		z_j'(\hat v)
		&:=
		\iprod{\hat \eta^j,\hat v}_\theta\fall 1\leq j\leq n
	\end{align*}
	satisfy $R(D_1G(\hat\phi^\ast,\alpha^\ast))=\bigcap_{j=1}^nN(z_j')$. 
\end{corollary}
\begin{proof}
	We proceed in several steps: 
	
	$(a)\Leftrightarrow(c)$ $\xi\mapsto\Phi_{\alpha^\ast}(\cdot,0)\xi$ is isomorphism from $X$ to the solution space of $(V_{\alpha^\ast})$, thus mapping a basis of $N(\Xi_\theta(\alpha^\ast)-I_X)$ to linearly independent $\theta$-periodic solutions.

	$(c)\Leftrightarrow(b)$ is due to \pref{proplin}. 

	$(b)\Leftrightarrow(e)$ is established in \pref{prop26}. 

	$(e)\Leftrightarrow(f)$ is shown in \pref{proplindual}.

	$(f)\Leftrightarrow(d)$ $\eta^\ast\mapsto\Phi_{\alpha^\ast}'(\theta,\cdot)\eta^\ast$ is an isomorphism between $Y$ and the solution space of $(V_{\alpha^\ast}')$ transferring linearly independent elements of $N(\Xi_\theta'(\alpha^\ast)-I_Y)$ to linearly independent $\theta$-periodic solutions of the dual variational equation $(V_{\alpha^\ast})$. 

	Finally, the Fredholm theory from \cite[pp.~52--58]{kress:14} yields the following equivalences
	\begin{eqnarray*}
		\hat v\in R(T)
		& \Leftrightarrow &
		\hat v\in N(T')^\perp
		\Leftrightarrow
		\iprod{\hat\eta^\ast,\hat v}=0\fall\eta^\ast\in N(T')\\
		& \stackrel{\eqref{notransvar}}{\Leftrightarrow} &
		\iprod{\hat\eta^j,\hat v}=0\fall 1\leq j\leq n\\
		& \Leftrightarrow &
		\sum_{t=0}^{\theta-1}
		\bigl\langle \eta_t^j, v_t\bigr\rangle=0\fall 1\leq j\leq n\\
		& \Leftrightarrow &
		z_j'(\hat v)=0\fall 1\leq j\leq n
		\Leftrightarrow
		\hat v\in\bigcap_{j=1}^nN(z_j')
	\end{eqnarray*}
	with $T:=D_1G(\hat\phi^\ast,\alpha^\ast)$. This leads to our assertion.
\end{proof}
\section{Periodic integrodifference equations}
\label{sec3}
This section applies the above preparations to periodic difference equations \eqref{deq}, whose right-hand side $\sF_t$ is a nonlinear integral operator. We specify concrete mappings $\sF_t$ including both dispersal-growth as well as growth-dispersal (Hammerstein) equations, and formulate standing assumptions guaranteeing sufficient smoothness and complete continuity of $\sF_t$. As an application, a persistence result for periodic solutions of IDEs is given. Furthermore, we determine the dual operator of the linearization of $\sF_t$.

Throughout, we suppose $\Omega$ is a compact metric space such that additionally, $(\Omega,\fA,\mu)$ is a measure space fulfilling $\mu(\Omega)<\infty$, so that the $\sigma$-algebra $\fA$ contains the Borel sets generated by the metric on $\Omega$, and so that $\mu(\Omega')>0$ holds for all nonempty open sets $\Omega'\subseteq\Omega$. In this section, the parameter space $A$ is assumed to be an open subset of a Banach space $P$. It is handy to abbreviate (when $U\subseteq\R^d$)
\begin{align*}
	C(\Omega,U)&:=\set{u:\Omega\to U\,|\,u\text{ is continuous}},&
	C_d&:=C(\Omega,\R^d),
\end{align*}
and we choose $X=C_d$ with the norm $\norm{u}:=\max_{x\in\Omega}\abs{u(x)}$ in what follows. 

The next result motivates our assumption of having no open sets of measure $0$. 
\begin{lemma}\label{lem31}
	The bilinear form 
	\begin{equation}
		\iprod{u,v}:=\int_\Omega\sprod{u(y),v(y)}\d\mu(y)\fall u,v\in C_d
		\label{dpair}
	\end{equation}
	yields a bounded duality pairing $\iprod{C_d,C_d}$. 
\end{lemma}
\begin{proof}
	The proof extends \cite[p.~46, Thm.~4.4]{kress:14} to our more general setting. Begin by noting that for each $u\in C_d\setminus\set{0}$, there exists an $x_0\in\Omega$ with $u(x_0)\neq 0$ and an open neighborhood $\Omega'\subseteq\Omega$ of $x_0$ so that $|u(x)|\geq\tfrac{|u(x_0)|}{2}>0$ for all $x\in\Omega'$. Now
	\begin{align*}
		\iprod{u,u}=\int_\Omega\sprod{u(y),u(y)}\d\mu(y) %= \sum_{t=0}^{\theta-1}\int_\Omega u_t(y)^2\d\mu(y)
		\geq \sum_{t=0}^{\theta-1}\int_{\Omega'} u_t(y)^2\d\mu(y) \geq \frac{|u(x_0)|^2}{4}\mu(\Omega') > 0
	\end{align*}
	holds via the assumption $\mu(\Omega')>0$. Due to the Cauchy-Schwarz inequality in $\R^d$, we have
$$
	\abs{\iprod{u,v}}
	\leq
	\int_\Omega\abs{\sprod{u(y),v(y)}}\d\mu(y)
	\leq
	\int_\Omega\abs{u(y)}\abs{v(y)}\d\mu(y)
	\leq
	\mu(\Omega)\norm{u}\norm{v}
$$
for all $u,v\in C_d$, and therefore $\iprod{\cdot,\cdot}$ is also bounded. 
\end{proof}

The right-hand side of \eqref{deq} is assumed to be of the form
\begin{equation}
	\sF_t(u,\alpha)(x)
	:=
	G_t\intoo{x,\int_\Omega f_t(x,y,u(y),\alpha)\d\mu(y),\alpha}
	\fall x\in\Omega.
	\label{ury}
\end{equation}
In order to deal with periodic IDEs \eqref{deq}, we assume there exists a \emph{basic period} $\theta_0\in\N$ such that $f_t=f_{t+\theta_0}$ and $G_t=G_{t+\theta_0}$, $t\in\Z$; then $\sF_t=\sF_{t+\theta_0}$ holds for all $t\in\Z$. Given a differentiability order $m\in\N$, the following \textbf{standing assumptions} are supposed for every $0\leq t<\theta_0$: 
\begin{itemize}
	\item[$(H_1)$] $f_t:\Omega\tm\Omega\tm U_t^1\tm A\to\R^p$ is continuous with an open, nonempty and convex $U_t^1\subseteq\R^d$ and the derivatives $D_{(3,4)}^j f_t:\Omega\tm\Omega\tm U_t^1\tm A\to L_j(\R^d\tm P,\R^p)$ for $1\leq j\leq m$ exist as continuous functions.

	\item[$(H_2)$] $G_t:\Omega\tm U_t^2\tm A\to\R^d$ is continuous with an open, nonempty and convex $U_t^2\subseteq\R^p$ and the derivatives $D_{(2,3)}^j G_t:\Omega\tm U_t^2\tm A\to L_j(\R^p\tm P,\R^d)$ for $1\leq j\leq m$ exist as continuous functions. 
\end{itemize}

As a result, the \emph{Urysohn operator}
\begin{align*}
	\sU_t:C(\Omega,U_t^1)\tm A&\to C_p,&
	\sU_t(u,\alpha)&:=\int_\Omega f_t(\cdot,y,u(y),\alpha)\d\mu(y)
%	\label{intop}
\end{align*}
is of class $C^m$ and referring to \cite{poetzsche:18a}\footnote{This reference assumes a globally defined operator $\sF_t$, i.e.\ $U_t=C_d$. Yet, the reader can verify that the corresponding proofs merely require the domains $U_t^1,U_t^2$ to be convex (as assumed above).}, this guarantees that the right-hand side \eqref{ury} of \eqref{deq} defined on an ambient subset
$$
	U_t\subseteq\set{u\in C(\Omega,U_t^1):\,\int_\Omega f_t(x,y,u(y),\alpha)\d\mu(y)\in U_t^2\text{ for all }x\in\Omega}
$$
fulfills for every $t\in\Z$: 
\begin{itemize}
	\item[$(P_1)$] $\sF_t\in C^m(U_t\tm A,C_d)$ (see \cite[Proceed as in the proof of Prop.~2.7]{poetzsche:18a}), 

	\item[$(P_2)$] $D_1\sF_t(u,\alpha)\in L(C_d)$ is compact for all $(u,\alpha)\in U_t\times A$ .
\end{itemize}

Working with a rather general measure $\mu$ in \eqref{ury} allows us to capture both classical IDEs, as well as their spatial discretizations in a unified framework: 
\begin{example}[measures]\label{exmeasure}
	(1) In the applications \cite{andersen:91,kot:schaeffer:86,lutscher:petrovskii:08,reimer:bonsall:maini:16,slatkin:73,kirk:lewis:97}, $\mu$ is simply the $\kappa$-dimensional Lebesgue measure yielding the Lebesgue integral in \eqref{ury} and therefore the IDE
	$$
		u_{t+1}(x)
		=
		G_t\intoo{x,\int_\Omega f_t(x,y,u_t(y),\alpha)\d y,\alpha}\fall x\in\Omega.
	$$

	(2) Suppose that the compact set $\Omega\subset\R^\kappa$ is countable, $\eta\in\Omega$ and $w_\eta$ denote non\-negative reals. Then $\mu(\Omega'):=\sum_{\eta\in\Omega'}w_\eta$ defines a measure on the family of countable subsets $\Omega'\subset\R^\kappa$. The assumption $\sum_{\eta\in\Omega}w_\eta<\infty$ guarantees that $\mu(\Omega)$ is finite. The resulting $\mu$-integral $\int_\Omega u\d\mu=\sum_{\eta\in\Omega}w_\eta u(\eta)$ leads to difference equations
	\begin{equation}
		u_{t+1}(x)
		=
		G_t\intoo{x,\sum_{\eta\in\Omega}w_\eta f_t(x,\eta,u_t(\eta),\alpha),\alpha}\fall x\in\Omega,
		\label{ury2}
	\end{equation}
	which cover \emph{Nystr\"om methods} with \emph{nodes} $\eta$ and \emph{weights} $w_\eta$ as they appear in numerical discretizations \cite{atkinson:92}, \cite[pp.~224ff]{kress:14}. Alternatively, this captures models for populations spread between finitely many different patches (\emph{metapopulation models}, see \cite[Example~1]{kot:schaeffer:86}). For singletons $\Omega$, \eqref{ury2} turns into a system of difference equations in $\R^d$ as studied in \cite{poetzsche:12}. 
\end{example}

Now fix a parameter $\alpha^\ast\in A$ and an associate $\theta_1$-periodic solution $\phi^\ast$ of $(\Delta_{\alpha^\ast})$. From \eqref{ury}, one obtains the partial derivative
\begin{align}
	&[D_1\sF_t(\phi_t^\ast,\alpha^\ast)v](x)
	\label{derf10}\\
	&=D_2G_t\intoo{x,\int_\Omega f_t(x,y,\phi_t^\ast(y),\alpha^\ast)\d\mu(y),\alpha^\ast}
	\int_\Omega D_3f_t(x,y,\phi_t^\ast(y),\alpha^\ast)v(y)\d\mu(y) \notag
\end{align}
for all $t\in\Z$, $x\in\Omega$ and $v\in C_d$. Apparently, \eqref{derf10} is the product of a multiplication operator with a Fredholm integral operator and therefore compact. 

\begin{theorem}[persistence of periodic solutions]\label{poinc}
	Let $\alpha^\ast\in A$, $\theta_1\in\N$ and define $\theta:=\lcm(\theta_0,\theta_1)$. If $\phi^\ast$ is an $\theta_1$-periodic solution of $(\Delta_{\alpha^\ast})$ satisfying the \emph{weak hyperbolicity condition}
	\begin{equation}
		1\not\in\sigma_\theta(\alpha^\ast),
		\label{weakhyp}
	\end{equation}
	then there exist $\rho,\eps>0$ and a $C^m$-function $\phi:B_\rho(\alpha^\ast)\to B_\eps(\phi^\ast)\subseteq\ell_\theta(C_d)$ such that the following statements hold for all $\alpha\in B_\rho(\alpha^\ast)$: 
	\begin{itemize}
		\item[(a)] $\phi(\alpha)$ is the unique $\theta$-periodic solution of \eqref{deq} in $\cB_\eps(\phi^\ast)$ and $\phi(\alpha^\ast)=\phi^\ast$,

		\item[(b)] $D\phi(\alpha^\ast)=(\ldots,\underline{\psi_0},\ldots,\psi_{\theta-1},\ldots)$ with $\psi_0,\ldots,\psi_{\theta-1}\in L(P,C_d)$ uniquely given by the cyclic system of Fredholm integral equations of the second kind
		\begin{equation}
			\begin{cases}
				\psi_0=D_1\sF_{\theta-1}(\phi_{\theta-1}^\ast,\alpha^\ast)\psi_{\theta-1}+D_2\sF_{\theta-1}(\phi_{\theta-1}^\ast,\alpha^\ast),\\
				\psi_1=D_1\sF_0(\phi_0^\ast,\alpha^\ast)\psi_0+D_2\sF_0(\phi_0^\ast,\alpha^\ast),\\
				\quad\vdots\\
				\psi_{\theta-1}=D_1\sF_{\theta-2}(\phi_{\theta-2}^\ast,\alpha^\ast)\psi_{\theta-2}+D_2\sF_{\theta-2}(\phi_{\theta-2}^\ast,\alpha^\ast),\\
			\end{cases}
			\label{poinc2}
		\end{equation}

		\item[(c)] in case the solution $\phi^\ast$ is even \emph{hyperbolic}, i.e.\ $\sigma_\theta(\alpha^\ast)\cap\S^1=\emptyset$, then also $\phi(\alpha)$ is hyperbolic with the same Morse index as $\phi^\ast$, 
	\end{itemize}
	where the occurring derivatives are given by \eqref{derf10} and
	\begin{align}
		&[D_2\sF_t(\phi_t^\ast,\alpha^\ast)p](x)
		=D_2G_t\intoo{x,\int_\Omega f_t(x,y,\phi_t^\ast(y),\alpha^\ast)\d\mu(y),\alpha^\ast}\label{derf01}\\
		&
		\quad
		\int_\Omega D_4f_t(x,y,\phi_t^\ast(y),\alpha^\ast)p\d\mu(y) +
		D_3G_t\intoo{x,\int_\Omega f_t(x,y,\phi_t^\ast(y),\alpha^\ast)\d\mu(y),\alpha^\ast}p\notag
	\end{align}
	for all $t\in\Z$, $x\in\Omega$ and $p\in P$. 
\end{theorem}
It is an immediate consequence of statement (c) that the solutions $\phi(\alpha)$ are exponentially stable, provided $\phi^\ast$ is exponentially stable and has Morse index $0$. 
\begin{proof}
	First of all, $\phi^\ast$ is a $\theta$-periodic solution of $(\Delta_{\alpha^\ast})$, and \tref{thmmain} yields $G(\hat\phi^\ast,\alpha^\ast)=0$. Moreover, due to $(P_1)$ and \pref{propG}(a), the mapping $G:\hat U\tm A\to C_d^\theta$ is of class $C^m$. Thanks to $(P_2)$, we obtain that every $D_1\sF_t(\phi_t^\ast,\alpha^\ast)$ is compact. Hence, \pref{propspec} applies, and so hyperbolicity \eqref{weakhyp} implies that the partial derivative $D_1G(\hat\phi^\ast,\alpha^\ast)\in L(C_d^\theta)$ is invertible. Now the implicit function theorem (e.g.\ \cite[pp.~250--251, Thm.~4.E]{zeidler:95}) guarantees the existence of neighborhoods $B_\rho(\alpha^\ast)\subseteq A$, $B_\eps(\hat\phi^\ast)\subseteq\hat U$ and of a $C^m$-function $\hat\phi:B_\rho(\alpha^\ast)\to B_\eps(\hat\phi^\ast)$ such that
	\begin{equation}
		G(\hat\phi(\alpha),\alpha)\equiv 0\on B_\rho(\alpha^\ast). 
		\label{solid}
	\end{equation}

	(a) This results by \tref{thmmain} with $\phi(\alpha):=(\ldots,\underline{\hat\phi(\alpha)_0},\ldots,\hat\phi(\alpha)_{\theta-1},\ldots)\in\ell_\theta(C_d)$, therefore $\phi:B_\rho(\alpha^\ast)\to\ell_\theta(C_d)$ is also of class $C^m$ and the claimed uniqueness is due to the implicit function theorem. 

	(b) Taking the derivative in \eqref{solid} gives $D_1G(\hat\phi^\ast,\alpha^\ast)D\hat\phi(\alpha^\ast)+D_2G(\hat\phi^\ast,\alpha^\ast)=0$ with derivatives given in \eqref{propG1}, \eqref{propG2}. This yields the cyclic system \eqref{poinc2}, which can be uniquely solved because of the hyperbolicity assumption \eqref{weakhyp} and \pref{propspec}. 

	(c) The hyperbolicity of $\phi^\ast$ implies that there exist disjoint sets $\sigma_s,\sigma_u$ with 
	\begin{align*}
		\sigma_\theta(\alpha^\ast)&=\sigma_u\dot\cup\sigma_s,&
		\sigma_u&\subseteq B_1(0),&
		\sigma_s&\subseteq\C\setminus\bar B_1(0)
	\end{align*}
	and since all $\Xi_\theta(\alpha)$ are compact, $\sigma_u$ consists of finitely many eigenvalues (having finite multi\-plicity). By assumption, $\alpha\mapsto D_1\sF_t(\phi(\alpha)_t,\alpha)$ is continuous on $B_\rho(\alpha^\ast)$; thus, $\Xi_\theta: B_\rho(\alpha^\ast)\to L(C_d)$ is continuous. Hence, \cite[pp.~213--214, Sect.~5]{kato:80} shows that the above spectral splitting persists in a neighborhood of $\alpha^\ast$, while the dimension of the unstable subspace of $\Xi_\theta(\alpha)$ remains constant. This implies the claim. 
\end{proof}

If an IDE \eqref{deq} depends on a real parameter $\alpha$, then the effect of parameter changes to the total population can be determined using the following tool.
\begin{remark}[average population vector]
	Suppose that $P=\R$. The value of the function
	\begin{align*}
		M_\theta:B_\rho(\alpha^\ast)&\to\R^d,&
		M_\theta(\alpha)&:=\frac{1}{\theta}\sum_{t=0}^{\theta-1}\int_\Omega\phi(\alpha)_t(y)\d\mu(y)
	\end{align*}
	is understood as average population vector over one period length $\theta$. If the functions $\psi_1,\ldots,\psi_\theta\in C_d$ are given by \eqref{poinc2}, then $\dot M_\theta(\alpha^\ast)=\frac{1}{\theta}\sum_{t=0}^{\theta-1}\int_\Omega\psi_t(y)\d\mu(y)$ allows one to determine whether changes in $\alpha$ near $\alpha^\ast$ lead to an increase ($\dot M_\theta(\alpha^\ast)_i>0$) or a decrease ($\dot M_\theta(\alpha^\ast)_i<0$) in the $i$-th average population, $1\leq i\leq d$. 
\end{remark}

\begin{figure}
	\includegraphics[width=60mm]{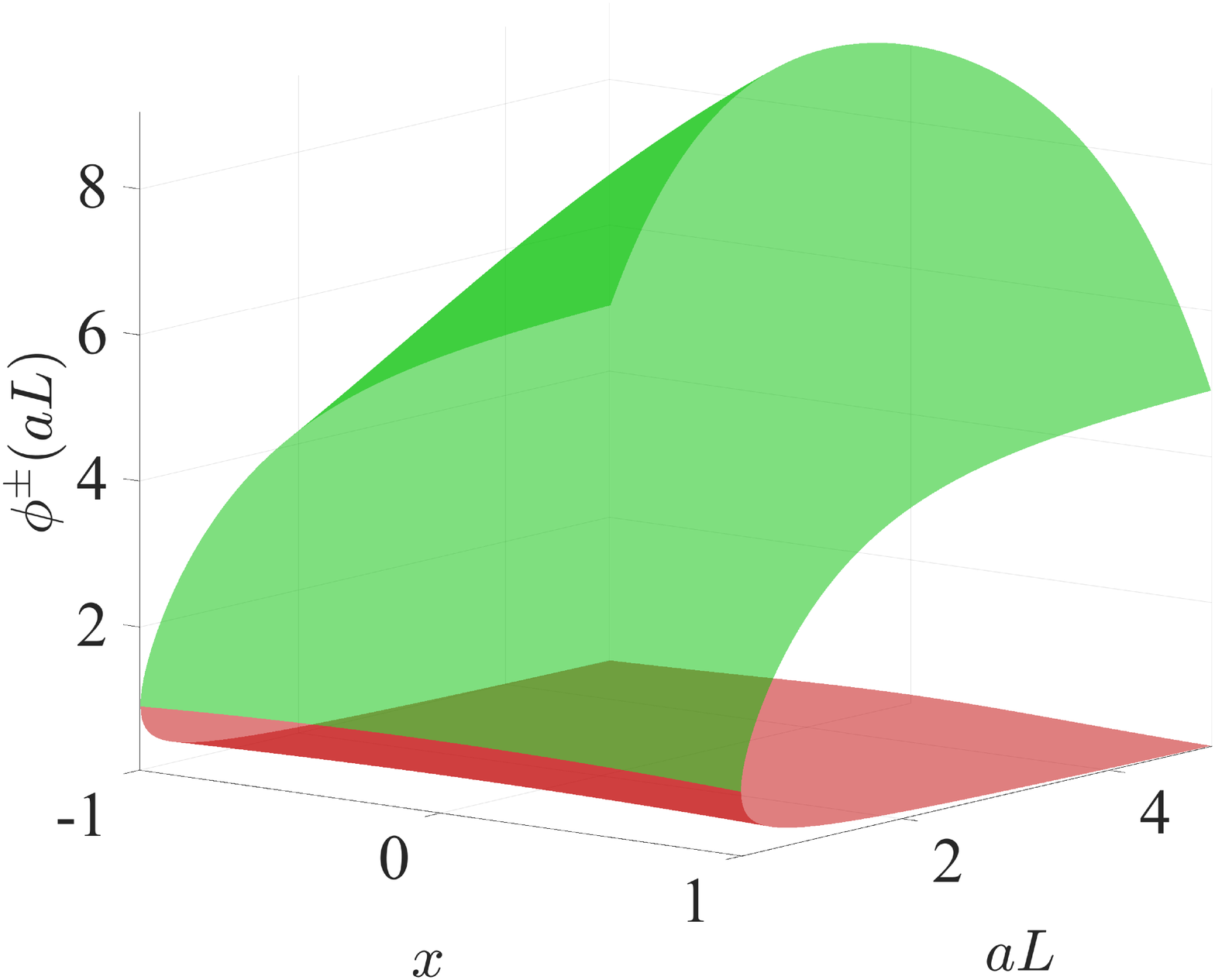}
	\includegraphics[width=60mm]{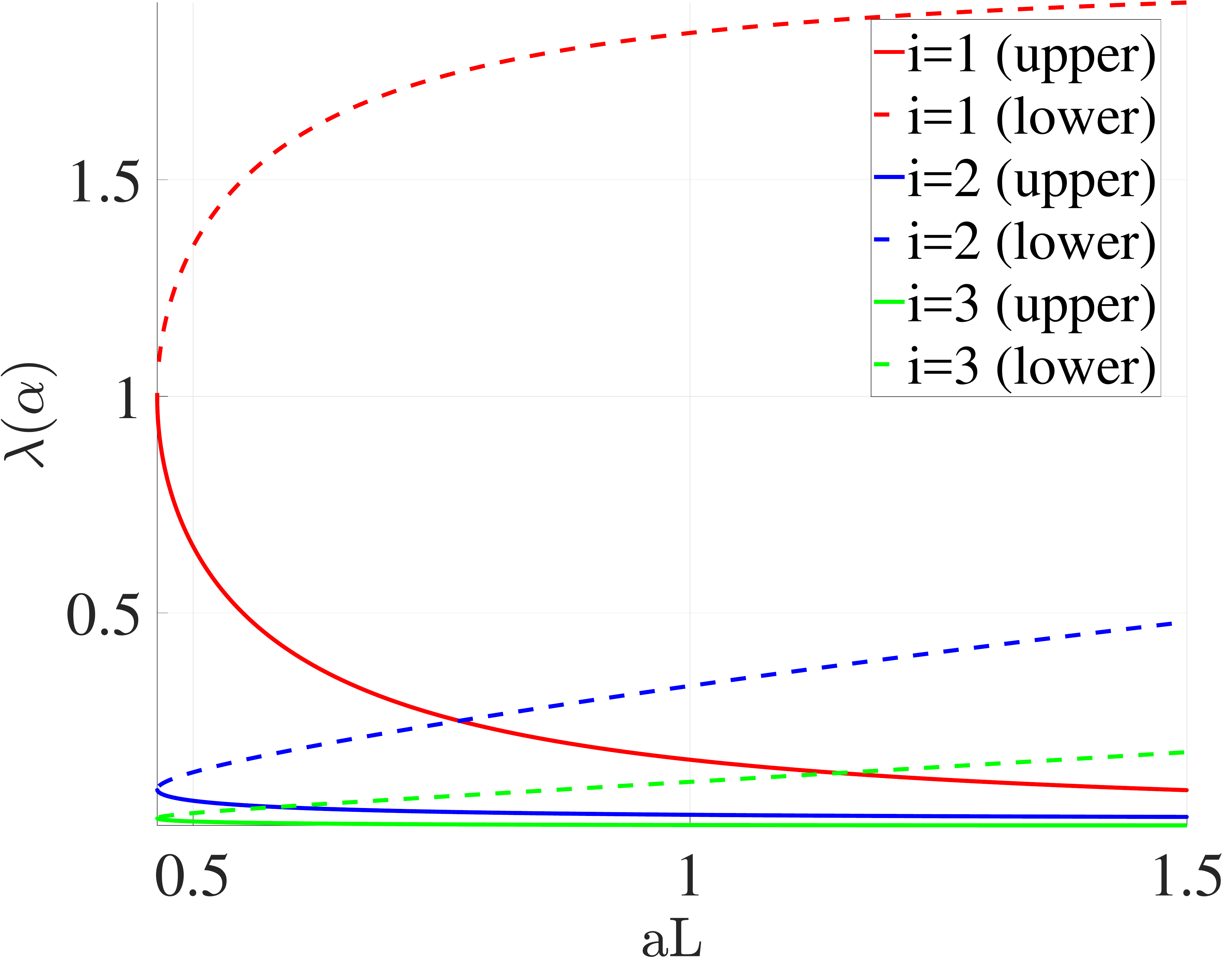}
	\caption{Equilibrium branches $\phi^\pm(aL)$ as functions of the dispersal parameter $\alpha=aL$ (left). Four largest eigenvalues $\lambda^\pm(aL)$ along these two branches of nontrivial solutions to \eqref{noall} (right)}
	\label{figallee}
\end{figure}
The next example relates the dynamics of a scalar difference equation (in $\R$) to its IDE counterpart (in the function space $C_1$). We demonstrate how the size of the habitat can be considered as a parameter (although $\Omega$ is constant in \eqref{ury}) and in which way it affects the stability and bifurcation behavior of our IDE.
\begin{example}[Allee equation]\label{exallee}
	Let $\beta>0$. The scalar \emph{Allee equation} 
	\begin{equation}
		u_{t+1}=\frac{\beta u_t^2}{1+u_t^2}
		\label{allee1}
	\end{equation}
	(cf.~\cite[pp.~54--55]{brauer:castillo:01}) has the trivial solution, which is exponentially stable for every parameter $\beta>0$. However, at $\beta=2$ there is a change in the behavior of \eqref{allee1}:
	\begin{itemize}
		\item For $\beta<2$ it has only the trivial equilibrium.

		\item For $\beta=2$ the nontrivial fixed point $1$ appears. 

		\item For $\beta>2$ there exist two equilibria $\phi^\pm(\beta):=\tfrac{1}{2}(\beta\pm\sqrt{\beta^2-4})\neq 0$, where the lower $\phi^-(\beta)$ is unstable, while the upper one $\phi^+(\beta)$ is exponentially stable. 
	\end{itemize}
	Summarizing, a supercritical fold bifurcation in \eqref{allee1} appears for $\beta=2$. Let us now restrict to the hyperbolic case $\beta=10$, where $\phi^\pm(10)=5\pm 2\sqrt{6}$. We aim to determine the way in which the behavior of \eqref{allee1} changes under the additional effect of dispersal \cite{kirk:lewis:97}. For simplicity, assume an interval $[-\tfrac{L}{2},\tfrac{L}{2}]$ of length $L>0$ as habitat $\Omega$ and consider the corresponding autonomous growth-dispersal IDE
	$$
		u_{t+1}=10\int_{-\tfrac{L}{2}}^{\tfrac{L}{2}}k(\cdot,y)\frac{u_t(y)^2}{1+u_t(y)^2}\d y
	$$
	on $C[-\tfrac{L}{2},\tfrac{L}{2}]$. We introduce the toplinear isomorphism $T_L\in L\bigl(C[-\tfrac{L}{2},\tfrac{L}{2}],C[-1,1]\bigr)$, $(T_Lu)(\xi):=u\bigl(\tfrac{L}{2}\xi\bigr)$ for all $\xi\in[-1,1]$. The change of variables formula implies
	$$
		(T_Lu_{t+1})(\tilde x)
		=
		5L\int_{-1}^{1}k\bigl(\tfrac{L}{2}\tilde x,\tfrac{L}{2}\tilde y\bigr)
		\frac{u_t\bigl(\tfrac{L}{2}\tilde y\bigr)^2}{1+u_t\bigl(\tfrac{L}{2}\tilde y\bigr)^2}\d\tilde y
		\fall\tilde x\in[-1,1].
	$$
	Therefore, the sequence $v_t:=T_Lu_t$ in $C[-1,1]$ satisfies the IDE
	$$
		v_{t+1}(x)
		=
		5L\int_{-1}^{1}k\bigl(\tfrac{L}{2}x,\tfrac{L}{2}y\bigr)
		\frac{v_t(y)^2}{1+v_t(y)^2}\d y
		\fall x\in[-1,1]
	$$
	on the constant habitat $[-1,1]$. In order to become more concrete, choose $k$ as the \emph{Laplace kernel} (cf.~\cite{lutscher:petrovskii:08,reimer:bonsall:maini:16,kirk:lewis:97})
	\begin{equation}
		k(x,y):=\tfrac{a}{2}e^{-a\abs{x-y}}\fall x,y\in\R,
		\label{kerlap}
	\end{equation}
	with some dispersal rate $a>0$. The resulting Hammerstein IDE
	\begin{equation}
		v_{t+1}(x)
		=
		5\frac{aL}{2}\int_{-1}^{1}e^{-\tfrac{aL}{2}\abs{x-y}}
		\frac{v_t(y)^2}{1+v_t(y)^2}\d y
		\fall x\in[-1,1]
		\label{noall}
	\end{equation}
	fits in the setting \eqref{ury} and depends only on the product $aL>0$, which we consider as parameter. Our numerical simulations indicate that \eqref{noall} behaves similarly to its scalar predecessor \eqref{allee1} when $aL$ is varied. For parameters $aL>0.464$ there are two fixed point branches $\phi^\pm(aL)\in C_1$ (see \fref{figallee} (right)\footnote{Here and in the following our coloring scheme is based on stability, where green means exponential stability and increasingly darker tones of red indicate corresponding instability with growing Morse index}) merging at $aL\approx 0.464$. The associate eigenvalue branches depicted in \fref{figallee} (left) indicate that the upper branch $\phi^+$ stays exponentially stable, while the lower branch $\phi^-$ stays unstable. Along $\phi^+$ the total population $M_1(aL)$ is increasing for $aL>0.464$, i.e.\ both larger habitats, as well as larger dispersal rates $a$ are beneficial for the population size. 
\end{example}
\begin{lemma}\label{lemdual}
	The dual operator of $D_1\sF_t(u,\alpha)\in L(C_d)$ exists and is given by
	\begin{align}
		&[D_1\sF_t(u,\alpha)'v](x)
		\label{derf10s}\\
		&=\int_\Omega D_3f_t(y,x,u(x),\alpha)^T
		D_2G_t\intoo{y,\int_\Omega f_t(y,\eta,u(\eta),\alpha)\d\mu(\eta),\alpha}^Tv(y)\d\mu(y)
		\notag
	\end{align}
	for all $t\in\Z$, $x\in\Omega$, $u\in U_t$, $\alpha\in A$ and $v\in C_d$. 
\end{lemma}
\begin{proof}
	Let $t\in\Z$. Notation-wise, it is convenient to neglect the dependence on $\alpha$ in $G_t,f_t,\sF_t$ and to write $M_t(x):=D_2G_t\intoo{x,\int_\Omega f_t(x,\eta,u(\eta))\d\mu(\eta)}\in\R^{d\tm p}$. We have
	\begin{eqnarray*}
		\iprod{w,D\sF_t(u)v}
		& = &
		\int_\Omega\sprod{w(x),[D\sF_t(u)v](x)}\d\mu(x)\\
		& \stackrel{\eqref{derf10}}{=} &
		\int_\Omega\sprod{w(x),M_t(x)\int_\Omega D_3f_t(x,y,u(y))v(y)\d\mu(y)}\d\mu(x)\\
		& = &
		\int_\Omega\sprod{M_t(x)^Tw(x),\int_\Omega D_3f_t(x,y,u(y))v(y)\d\mu(y)}\d\mu(x)\\
		& = &
		\int_\Omega\int_\Omega\sprod{M_t(x)^Tw(x),D_3f_t(x,y,u(y))v(y)}\d\mu(y)\d\mu(x)\\
		& = &
		\int_\Omega\int_\Omega\sprod{D_3f_t(x,y,u(y))^TM_t(x)^Tw(x),v(y)}\d\mu(y)\d\mu(x)
	\end{eqnarray*}
	and Fubini's theorem (e.g. \cite[pp.~159--160, Thm.~5.2.2]{cohn:80}) implies
	\begin{eqnarray*}
		\iprod{w,D\sF_t(u)v}
		& = &
		\int_\Omega\int_\Omega\sprod{D_3f_t(y,x,u(y))^TM_t(y)^Tw(y),v(x)}\d\mu(y)\d\mu(x)\\
		& = &
		\int_\Omega\sprod{\int_\Omega D_3f_t(y,x,u(x))^TM_t(y)^Tw(y)\d\mu(y),v(x)}\d\mu(x)\\
		& = &
		\int_\Omega\sprod{[D\sF_t(u)'w](x),v(x)}\d\mu(x)
		=
		\iprod{D\sF_t(u)'w,v}
	\end{eqnarray*}
	for all $v,w\in C_d$, as well as $u\in U_t$. This proves the claim. 
\end{proof}
\section{Bifurcations in periodic integrodifference equations}
\label{sec4}
Although this section contains only two branching criteria for periodic solutions to IDEs \eqref{deq}, our setting is nevertheless sufficiently flexible to cover fold, transcritical, pitchfork and flip bifurcations. The first three of these address a rather typical feature of models from theoretical ecology, namely monotone right-hand sides. Hence, the Krein-Rutman theorem \cite[p.~228, Thm.~19.3]{deimling:85} guarantees that a simple, real, positive eigenvalue (with positive eigenfunction) is dominant. It crosses the critical value $1$, leading to such a primary bifurcation. 

In the remaining text, we retreat to parameter spaces $A$ being open subsets of the real numbers, i.e.\ $P=\R$. Suppose that $\theta$ is a multiple of both the periods $\theta_0$ of an IDE \eqref{deq} and of a fixed reference solution $\phi^\ast$, and that the assumptions $(H_1$--$H_2)$  on the right-hand side hold. The previous \sref{sec3} showed that qualitative changes in the set of $\theta$-periodic solutions to \eqref{deq} require the weak hyperbolicity condition \eqref{weakhyp} to be violated, i.e.\
\begin{equation}
	1\in\sigma_\theta(\alpha^\ast).
	\label{nohyp}
\end{equation}
Therefore, it is crucial to determine parameter values $\alpha^\ast$ giving rise to such changes, and to understand these changes at least locally. Specifying this, a $\theta_1$-periodic solution $\phi^\ast$ to $(\Delta_{\alpha^\ast})$ \emph{bifurcates} at a parameter $\alpha^\ast\in A$, if there exists a parameter sequence $(\alpha_n)_{n\in\N}$ with limit $\alpha^\ast$ in $A$ and distinct sequences $(\phi_n^1)_{n\in\N}$, $(\phi_n^2)_{n\in\N}$ of $\theta$-periodic solutions to $(\Delta_{\alpha_n})$ satisfying
$
	\lim_{n\to\infty}\phi_n^1
	=
	\lim_{n\to\infty}\phi_n^2.
$

Let us stress that this bifurcation notion is purely "analytical" and that stability changes will be addressed separately. We describe such bifurcations where the pair $(\phi^\ast,\alpha^\ast)$ is contained in a smooth branch of $\theta$-periodic solutions. This means: 
\begin{equation}
	\begin{cases}
		\text{There exist $\eps>0$, open intervals $S\subseteq\R$ containing $0$, $A_0\subseteq A$ containing}\\
		\text{$\alpha^\ast$ and a smooth curve $\svector{\gamma}{\alpha}:S\to B_\eps(\phi^\ast)\tm A\subseteq\ell_\theta(C_d)\tm\R$ such that}\\
		\text{$\gamma(0)=\phi^\ast$, $\alpha(0)=\alpha^\ast$ and each $\gamma(s)$ is an $\theta$-periodic solution of the}\\
		\text{IDE $(\Delta_{\alpha(s)})$ for all $s\in S$. The image $\Gamma=\svector{\gamma}{\alpha}(S)$ is called a \emph{branch}.}
	\end{cases}
	\hspace*{-3mm}
	\label{branch}
\end{equation}
For later use we now abbreviate the solution sets
\begin{align*}
	\Gamma^+&:=\begin{bmatrix}\gamma\\ \alpha\end{bmatrix}(S\cap(0,\infty)),&
	\Gamma^-&:=\begin{bmatrix}\gamma\\ \alpha\end{bmatrix}(S\cap(-\infty,0)),
\end{align*}
obtain $\Gamma=\Gamma^+\dot\cup\set{\svector{\phi^\ast}{\alpha^\ast}}\dot\cup\Gamma^-$, and call them \emph{exponentially stable} or \emph{unstable} if all solutions of \eqref{deq} on them possess the respective stability characteristic. 

In order to deduce sufficient criteria for bifurcation, assume that $\alpha^\ast\in A$ is a \emph{critical parameter} in the sense that the following \textbf{bifurcation conditions} hold: 
\begin{itemize}
	\item[$(B_1)$] $\phi^\ast$ is a $\theta_1$-periodic solution to $(\Delta_{\alpha^\ast})$. 
	
	\item[$(B_2)$] $1$ is a simple Floquet multiplier, i.e.\ there exists $\xi^\ast_0\in C_d\setminus\set{0}$ with
	$$
		N(\Xi_\theta(\alpha^\ast)-I_{C_d})=\spann\set{\xi^\ast_0},
	$$
	giving rise to a $\theta$-periodic solution $\xi^\ast=(\xi^\ast_t)_{t\in\Z}\in\ell_\theta(C_d)$ of the variational equation $(V_{\alpha^\ast})$ (cf.~\pref{proplin}). Furthermore, choose an $\eta^\ast_0\in C_d\setminus\set{0}$ so that
	$$
		N(\Xi_\theta(\alpha^\ast)'-I_{C_d})
		=
		R(\Xi_\theta(\alpha^\ast)-I_{C_d})^\perp
		=
		\spann\set{\eta^\ast_0}
	$$
	(cf.~\cite[p.~294, Prop.~6(ii)]{zeidler:95}), which in turn induces an entire $\theta$-periodic solution $\eta^\ast=(\eta^\ast_t)_{t\in\Z}\in\ell_\theta(C_d)$ of the dual variational equation $(V_{\alpha^\ast}')$ (cf.~\pref{proplindual}). 
\end{itemize}

As final preparation for our subsequent bifurcation criteria we note that a combination of \eqref{dpper} and \lref{lem31} yields the duality pairing $\iprod{C_d^\theta,C_d^\theta}_\theta$ with
$$
	\iprod{\hat w,\hat v}_\theta=\sum_{t=0}^{\theta-1}\int_\Omega\sprod{w_t(y),v_t(y)}\d\mu(y)
	\fall\hat v,\hat w\in C_d^\theta
$$
and we state a technical result. 
\begin{lemma}\label{lemgder}
	If $i,j\in\N_0$ with $1\leq i+j\leq m$, $(i,j)\neq(1,0)$, then
	\begin{multline*}
		\iprod{\hat w,D_1^iD_2^j G(\hat\phi^\ast,\alpha^\ast)\hat v^1\cdots \hat v^{i}}_\theta\\
		=
		\sum_{t=0}^{\theta-1}\int_\Omega\sprod{w_{t+1}(x),\bigl[D_1^iD_2^j\sF_t(\phi_t^\ast,\alpha^\ast)v^1_t\cdots v^{i}_t\bigr](x)}\d\mu(x)
	\end{multline*}
	holds for all $\hat v^k,\hat w\in C_d^\theta$, $1\leq k\leq i$.
\end{lemma}
\begin{proof}
	From \pref{propG}(a) we obtain
	$$
		\bigl[D_1^iD_2^j G(\hat\phi^\ast,\alpha^\ast)\hat v^1\cdots\hat v^{i}\bigr]_{t+1}
		=
		D_1^iD_2^j\sF_t(\phi_t^\ast,\alpha^\ast)v^1_t\cdots v^{i}_t
		\fall t\in\Z
	$$
	and it immediately results that
	\begin{align*}
		\iprod{\hat w,D_1^iD_2^j G(\hat\phi^\ast,\alpha^\ast)\hat v^1\cdots\hat v^i}_\theta 
		&=
		\sum_{t=0}^{\theta-1}
		\int_\Omega
		\sprod{w_t(x),\bigl[D_1^iD_2^j G(\hat\phi^\ast,\alpha^\ast)\hat v^1\cdots\hat v^{i}\bigr]_t(x)}\!\d\mu(x) \\
		&\hspace{-10mm}= 
		\sum_{t=0}^{\theta-1}\int_\Omega\sprod{w_{t+1}(x),\bigl[D_1^iD_2^j G(\hat\phi^\ast,\alpha^\ast)\hat v^1\cdots\hat v^{i}\bigr]_{t+1}(x)}\d\mu(x) \\
		&\hspace{-10mm}= 
		\sum_{t=0}^{\theta-1}\int_\Omega\sprod{w_{t+1}(x),\bigl[D_1^iD_2^j\sF_t(\phi_t^\ast,\alpha^\ast)v^1_t\cdots v^{i}_t\bigr](x)}\d\mu(x)
	\end{align*}
	due to the $\theta$-periodicity of \eqref{deq}. 
\end{proof}
\subsection{Fold bifurcation}
\label{sec41}
We start with our possibly simplest bifurcation pattern.
\begin{theorem}[fold bifurcation]\label{thmfold}
	Let $m\geq 2$, and suppose $(B_1$--$B_2)$ are satisfied. If
	\begin{align*}
		g_{01}
		:=&
		\sum_{t=0}^{\theta-1}\int_\Omega\sprod{\eta^\ast_{t+1}(x),D_2\sF_t(\phi_t^\ast,\alpha^\ast)(x)}\d\mu(x)
		\neq
		0,
	\end{align*}
	then there exists a branch $\Gamma$ of $\theta$-periodic solutions of the IDE \eqref{deq} as in \eqref{branch}, with $C^m$-functions $\gamma,\alpha$ satisfying $\dot\gamma(0)=\xi^\ast$ and $\dot\alpha(0)=0$. Moreover, every $\theta$-periodic solution of \eqref{deq} in $\cB_\eps(\phi^\ast)\tm A_0$ is captured by $\Gamma$. Under the additional assumption
	\begin{align*}
		g_{20}
		:=&
		\sum_{t=0}^{\theta-1}\int_\Omega\sprod{\eta^\ast_{t+1}(x),[D_1^2\sF_t(\phi_t^\ast,\alpha^\ast)(\xi^\ast_t)^2](x)}\d\mu(x)
		\neq
		0,
	\end{align*}
	the $\theta$-periodic solution $\phi^\ast$ of $(\Delta_{\alpha^\ast})$ bifurcates at $\alpha^\ast$ into the branch $\Gamma$, $\ddot\alpha(0)=-\tfrac{g_{20}}{g_{01}}$, and locally in $\cB_\eps(\phi^\ast)\tm A_0$ the following hold (cf.~\fref{fig3}): 
	\begin{itemize}
		\item[(a)] \emph{Subcritical case}: If $g_{20}/g_{01}>0$, then \eqref{deq} has no $\theta$-periodic solution for $\alpha>\alpha^\ast$ and exactly two distinct $\theta$-periodic solutions for $\alpha<\alpha^\ast$. 
		
		\item[(b)] \emph{Supercritical case}: If $g_{20}/g_{01}<0$, then \eqref{deq} has no $\theta$-periodic solution for $\alpha<\alpha^\ast$ and exactly two distinct $\theta$-periodic solutions for $\alpha>\alpha^\ast$, 
	\end{itemize}
	where the occurring derivatives are given by \eqref{derf01} and
	\begin{align}
		&[D_1^2\sF_t(\phi_t^\ast,\alpha^\ast)v\bar v](x)
		=D_2^2G_t\intoo{x,\int_\Omega f_t(x,y,\phi_t^\ast(y),\alpha^\ast)\d\mu(y),\alpha^\ast}\label{derf20}\\
		&
		\quad\int_\Omega D_3f_t(x,y,\phi_t^\ast(y),\alpha^\ast)v(y)\d\mu(y)\int_\Omega D_3f_t(x,y,\phi_t^\ast(y),\alpha^\ast)\bar v(y)\d\mu(y)
		\notag\\
		&
		+ D_2G_t\intoo{x,\int_\Omega f_t(x,y,\phi_t^\ast(y),\alpha^\ast)\d\mu(y),\alpha^\ast}\int_\Omega D_3^2f_t(x,y,\phi_t^\ast(y),\alpha^\ast)v(y)\bar v(y)\d\mu(y)
		\notag
	\end{align}
	for all $t\in\Z$, $x\in\Omega$ and $v,\bar v\in C_d$. 
\end{theorem}
\begin{figure}[ht]
	\setlength{\unitlength}{1cm}
	\begin{center}
	\begin{picture}(12,8.5)
		\includegraphics[scale=0.6]{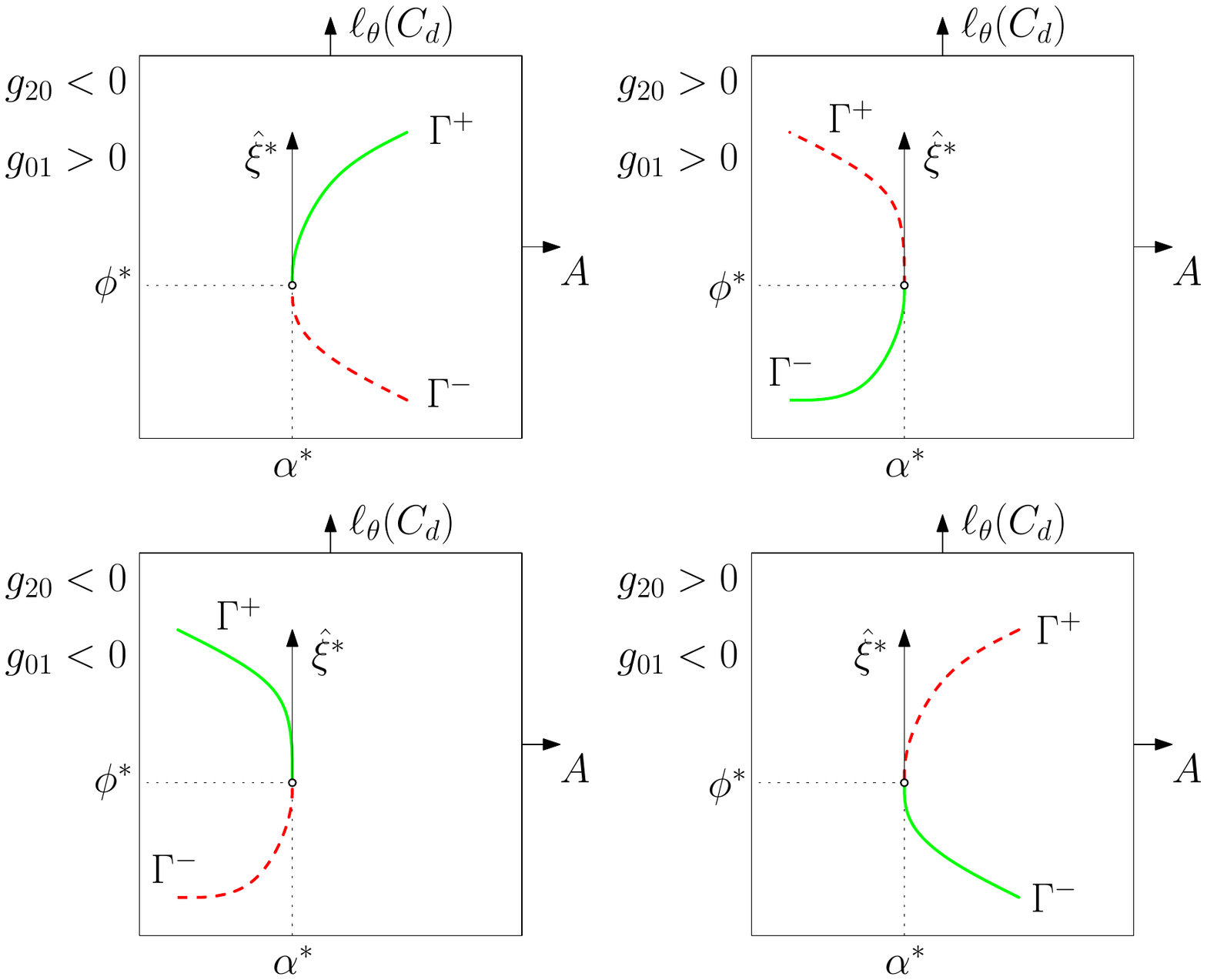}
	\end{picture}
	\end{center}
	\caption{Subcritical ($\tfrac{g_{20}}{g_{11}}>0$) and supercritical ($\tfrac{g_{20}}{g_{11}}<0$) fold bifurcation of $\theta$-periodic solutions to \eqref{deq} described in \tref{thmfold}, as well as the exchange of stability between the branches $\Gamma^+$ and $\Gamma^-$ from unstable (dashed line) to exponentially stable (solid) covered in \cref{bifurmorse}}
	\label{fig3}
\end{figure}
\begin{proof}
	For the sake of brief notation, we will repeatedly employ the abbreviation
	\begin{align*}
		G_{ij}&:=D_1^iD_2^jG(u^\ast,\alpha^\ast),&
		g_{ij}&:=z'(G_{ij}(\xi^\ast)^i)\fall i,j\in\N_0,\,i, j\leq m
	\end{align*}
	with the linear functional $z':C_d^\theta\to\R$ given by $z'(\hat v):=\iprod{\hat\eta^\ast,\hat v}_\theta$. Let us subdivide the proof into two steps: 

	(I) Our goal is to apply \tref{asnbif} to the abstract equation $G(\hat\phi,\alpha)=0$ between the same Banach spaces $X=C_d^\theta$ and $Z=C_d^\theta$. Above all, it follows from $(P_1)$ and \pref{propG}(a) that $G:\hat U\tm A\to C_d^\theta$ is of class $C^m$, $m\geq 2$. Thanks to $(B_1)$, $\phi^\ast$ is a $\theta$-periodic solution of $(\Delta_{\alpha^\ast})$, and \tref{thmmain} implies $G(\hat\phi^\ast,\alpha^\ast)=0$, i.e.\ \eqref{zero} holds. Moreover, due to \pref{propG}(b), the derivative $D_1G(\hat\phi^\ast,\alpha^\ast)$ is Fredholm of index $0$. As a consequence of $(B_2)$ and \pref{proplin}, \ref{proplindual}, we have
	\begin{align*}
		N(G_{10})&=\spann\{\hat\xi^\ast\},&
		N(G_{10}')&=\spann\{\hat\eta^\ast\},
	\end{align*}
	so that \eqref{fred} holds. According to \cref{cordual}, the linear functional $z':C_d^\theta\to\R$ satisfies $N(z')=R(G_{10})$ and is clearly bounded, yielding \eqref{nozero}. In conclusion, we are in the abstract setting of App.~\ref{appA} with $u^\ast=\hat\phi^\ast$. 
	
	(II) \lref{lemgder} immediately guarantees
	\begin{align*}
		z'(G_{01})
		&=
		\iprod{\hat\eta^\ast,G_{01}}_\theta
		=
		\sum_{t=0}^{\theta-1}\int_\Omega\sprod{\eta^\ast_{t+1}(x),D_2\sF_t(\phi_t^\ast,\alpha)(x)}\d\mu(x)
		=
		g_{01},\\
		z'(G_{20}(\xi^\ast)^2)%z'(D_1^2G(\hat\phi^\ast,\alpha^\ast)(\hat\xi^\ast)^2)
		&=
		\iprod{\hat\eta^\ast,D_1^2G(\hat\phi^\ast,\alpha^\ast)(\hat\xi^\ast)^2}_\theta\\
		&=
		\sum_{t=0}^{\theta-1}\int_\Omega\sprod{\eta^\ast_{t+1}(x),\left[D_1^2\sF_t(\phi_t^\ast,\alpha)(\xi^\ast_t)^2\right](x)}\d\mu(x)
		=
		g_{20}.
	\end{align*}
	\tref{asnbif} now applies to $G(\hat\phi,\alpha)=0$, and \tref{thmmain} implies the desired results. 
\end{proof}
If the Floquet multiplier $1$ is the unique element of $\sigma_\theta(\alpha^\ast)$ on the unit circle, i.e.\
\begin{equation}
	\sigma_\theta(\alpha^\ast)\cap\S^1=\set{1},
	\label{morse}
\end{equation}
then more can be said on the dynamics near the branch $\Gamma$. More specifically, when the remaining Floquet spectrum is contained in the open disk $B_1(0)$, i.e.\
\begin{equation}
	\sigma_\theta(\alpha^\ast)\setminus\set{1}\subseteq B_1(0),
	\label{stab}
\end{equation}
then a bifurcation goes hand in hand with a stability change for $\phi^\ast$ (see \fref{fig3}): 
\begin{corollary}[stability along $\Gamma$]\label{bifurmorse}
	Suppose that \eqref{morse} holds (cf.\ \fref{fig3}): 
	\begin{itemize}
		\item[(a)] If $g_{20}>0$, then the Morse index along $\Gamma$ increases by $1$ as $s$ grows through $0$ in \eqref{branch}. In particular, under \eqref{stab}, $\Gamma^-$ is exponentially stable and $\Gamma^+$ is unstable.

		\item[(b)] If $g_{20}<0$, then the Morse index along $\Gamma$ decreases by $1$ as $s$ grows through $0$ in \eqref{branch}. In particular, under \eqref{stab}, $\Gamma^+$ is exponentially stable and $\Gamma^-$ is unstable. 
	\end{itemize}
\end{corollary}
\begin{proof}
	If $\tilde\Xi_\theta(s)\in L(C_d)$ denotes the period operator of the variational equation $v_{t+1}=D_1\sF_t(\gamma(s)_t,\alpha(s))v_t$, then the following holds true in a vicinity of $s=0$: Since the embedding operator $J$ in \eqref{embed} is the identity on $C_d^\theta$, \lref{lemma:eigder} applies. By the spectral mapping theorem, there is a smooth curve of simple eigenpairs $(\lambda(s),\xi(s))$ for $I_{C_d^\theta}+D_1G(\gamma(s),\alpha(s))$ with $(\lambda(0),\xi(0))=(1,\hat\xi^\ast)$. Therefore, the function $\rho(s):=\abs{\lambda(s)}$ is differentiable in a neighborhood of $0$. Due to \tref{asnbif}(c), one has $\dot\lambda(0)=g_{20}$, and consequently $\dot\rho(0)=g_{20}$.

	(a) For $g_{20}>0$, the function $\rho$ is strictly increasing near $0$. Whence, \pref{propspec} implies that the Morse index along the solution branch $\Gamma$ increases as $s$ grows through the value $0$. Since $s\mapsto\lambda(s)$ is a curve of (algebraically) simple eigenvalues, also their geometric multiplicity is $1$, and so the Morse index increases by $1$. In particular, \eqref{stab} implies that $\sigma(\tilde\Xi_\theta(s))\subseteq B_1(0)$ for $s<0$, while $\tilde\Xi_\theta(s)$ possesses an eigenvalue of modulus $>1$ for $s>0$. 

	(b) For $g_{20}<0$, a dual argument shows that the Morse index along $\Gamma$ drops by $1$ as $s$ increases through $0$. If \eqref{stab} holds, then $\sigma(\tilde\Xi_\theta(s))\subset B_1(0)$ for $s>0$, whereas $s<0$ leads to Floquet spectrum outside the closed unit disk in $\C$. 
\end{proof}

Note that \eref{exallee} and \fref{figallee} feature a supercritical fold bifurcation of fixed points. 
\subsection{Crossing curve bifurcation}
\label{sec42}
The following bifurcation patterns require
\begin{itemize}
	\item[$(B_3)$] For all $0\leq t<\theta$, one has
	\begin{multline*}
		D_2G_t\intoo{x,\int_\Omega f_t(x,y,\phi_t^\ast(y),\alpha^\ast)\d\mu(y),\alpha^\ast}
		\int_\Omega D_4f_t(x,y,\phi_t^\ast(y),\alpha^\ast)\d\mu(y)\\
		+
		D_3G_t\intoo{x,\int_\Omega f_t(x,y,\phi_t^\ast(y),\alpha^\ast)\d\mu(y),\alpha^\ast}
		\equiv 
		0\on\Omega
	\end{multline*}
\end{itemize}
as further \textbf{bifurcation condition}. Note that $(B_3)$ means $D_2\sF_t(\phi_t^\ast,\alpha^\ast)\equiv 0$ on $\Z$. 
\begin{theorem}[crossing curve bifurcation]\label{thmcross}
	Let $m\geq 2$, and suppose $(B_1$--$B_3)$ are satisfied. If
	\begin{align*}
		g_{11}
		:=&
		\sum_{t=0}^{\theta-1}\int_\Omega\sprod{\eta^\ast_{t+1}(x),[D_1D_2\sF_t(\phi_t^\ast,\alpha^\ast)\xi^\ast_t](x)}\d\mu(x)
		\neq
		0,\\
		&
		\sum_{t=0}^{\theta-1}\int_\Omega\sprod{\eta^\ast_{t+1}(x),[D_2^2\sF_t(\phi_t^\ast,\alpha^\ast)](x)}\d\mu(x)
		=0
	\end{align*}
	hold, then the $\theta$-periodic solution $\phi^\ast$ of an IDE $(\Delta_{\alpha^\ast})$ bifurcates at $\alpha^\ast$ as follows: The pair $(\phi^\ast,\alpha^\ast)$ is the intersection of two branches $\Gamma_1,\Gamma_2$ of $\theta$-periodic solutions as in \eqref{branch}, and every $\theta$-periodic solution of \eqref{deq} in $\cB_\eps(\phi^\ast)$ is captured by $\Gamma_1$ or $\Gamma_2$. We have
	\begin{enumerate}
		\item[(a)] $\Gamma_1=\set{(\phi_1(\alpha),\alpha)\in\ell_\theta(C_d)\tm\R:\,\alpha\in A_0}$ with a $C^{m-1}$-function $\phi_1:A_0\to B_\eps(\phi^\ast)$ of $\theta$-periodic solutions $\phi_1(\alpha)$ to \eqref{deq} satisfying
		\begin{align*}
			\phi_1(\alpha^\ast)&=\phi^\ast,&
			\dot\phi_1(\alpha^\ast)&=0,
		\end{align*}
		
		\item[(b)] $\Gamma_2=\svector{\gamma_2}{\alpha_2}(S)$ with a $C^{m-1}$-curve $\svector{\gamma_2}{\alpha_2}:S\to B_\eps(\phi^\ast)\tm A_0$ such that
		\begin{equation*}
			\begin{split}
				\gamma_2(s)&=\phi^\ast+s\hat\xi^\ast+o(s),\\
				\alpha_2(s)&=\alpha^\ast-\frac{s}{2g_{11}}
				\sum_{t=0}^{\theta-1}\int_\Omega\sprod{\eta^\ast_{t+1}(x),[D_1^2\sF_t(\phi_t^\ast,\alpha^\ast)(\xi^\ast_t)^2](x)}\d\mu(x)+o(s),
			\end{split}
		\end{equation*}
	\end{enumerate}
	where the occurring derivatives are given by \eqref{derf20} and
	\begin{align*}
		&[D_1D_2\sF_t(\phi_t^\ast,\alpha^\ast)v](x)
		\\
		&
		=D_2G_t\intoo{x,\int_\Omega f_t(x,y,\phi_t^\ast(y),\alpha)\d\mu(y),\alpha^\ast}\int_\Omega D_3D_4f_t(x,y,\phi_t^\ast(y),\alpha^\ast)v(y)\d\mu(y)\\
		&+
		D_2D_3G_t\intoo{x,\int_\Omega f_t(x,y,\phi_t^\ast(y),\alpha^\ast)\d\mu(y),\alpha^\ast}
		\int_\Omega D_3f_t(x,y,\phi_t^\ast(y),\alpha^\ast)v(y)\d\mu(y)
		\\
		&+
		D_2^2G_t\intoo{x,\int_\Omega f_t(x,y,\phi_t^\ast(y),\alpha)\d\mu(y),\alpha^\ast}
		\int_\Omega D_3f_t(x,y,\phi_t^\ast(y),\alpha^\ast)v(y)\d\mu(y)\\
		&
		\quad\int_\Omega D_4f_t(x,y,\phi_t^\ast(y),\alpha^\ast)\d\mu(y)
%		\label{derf11}
	\end{align*}
	for all $t\in\Z$, $x\in\Omega$ and $v\in C_d$. 
\end{theorem}
In the bifurcation diagram (cf.\ Figs.~\ref{fig5}--\ref{fig6}), the branch $\Gamma_1$ is graph of a function over the $\alpha$-axis, whose tangent in $(\phi^\ast,\alpha^\ast)$ is $\alpha\mapsto(\phi^\ast,\alpha)$. 
\begin{remark}
	The assumption $(B_3)$ is satisfied, if $\phi^\ast$ is embedded into a constant branch of solutions, i.e.\ $\phi_{t+1}^\ast=\sF_t(\phi_t^\ast,\alpha)$ for all $t\in\Z$, $\alpha\in A$ holds. In this situation, one has $\phi_1(\alpha)\equiv\phi^\ast$ on $A_0$, that is, $\Gamma_1=\set{(\phi^\ast,\alpha)\in\ell_\theta(C_d)\tm\R:\,\alpha\in A_0}$. 
\end{remark}
\begin{proof}
	As shown in step (I) of the proof for \tref{thmfold}, we are in the set-up of App.~\ref{appA}. Under the present assumptions, \lref{lemgder} leads to
	\begin{align*}
		z'(G_{11}\hat\xi^\ast)
		&=
		\iprod{\hat\eta^\ast,G_{11}\hat\xi^\ast}_\theta\\
		&=
		\sum_{t=0}^{\theta-1}\int_\Omega\sprod{\eta^\ast_{t+1}(x),\left[D_1D_2\sF_t(\phi_t^\ast,\alpha)\xi^\ast_t\right](x)}\d\mu(x)
		=
		g_{11},\\
		z'(G_{02})
		&=
		\iprod{\hat\eta^\ast,G_{02}}_\theta
		=
		\sum_{t=0}^{\theta-1}\int_\Omega\sprod{\eta^\ast_{t+1}(x),D_2^2\sF_t(\phi_t^\ast,\alpha)(x)}\d\mu(x).
	\end{align*}
	Therefore, \tref{thmcrosbif} applies and yields a solution branch $\Gamma_1=\svector{\gamma_1}{\alpha_1}(S)$ given by a $C^{m-1}$-curve $\svector{\gamma_1}{\alpha_1}:S\to B_\eps(\phi^\ast)\tm A_0$ such that 
	\begin{align}
		\gamma_1(0)&=\phi^\ast,&
		\alpha_1(s)&=\alpha^\ast+s,&
		\dot\gamma_1(0)&=0,
		\label{gamma1ass}
	\end{align}
	as well as a branch $\Gamma_2$ having the properties claimed in (b). The assertion (a) follows from \eqref{gamma1ass} if we define $\phi_1(\alpha):=\gamma_1(\alpha-\alpha^\ast)$. 
\end{proof}

\begin{corollary}[stability along $\Gamma_1$]\label{trbif}
	Suppose that \eqref{morse} holds (cf.\ \fref{fig5} and \ref{fig6}): 
	\begin{itemize}
		\item[(a)] If $g_{11}>0$, then $m_\ast(\phi_1(\alpha))$ increases by $1$ as $\alpha$ grows through~$\alpha^\ast$. In particular, under \eqref{stab} the $\theta$-periodic solution $\phi_1(\alpha)$ of \eqref{deq} is exponentially stable for $\alpha<\alpha^\ast$ and unstable for $\alpha>\alpha^\ast$. 

		\item[(b)] If $g_{11}<0$, then $m_\ast(\phi_1(\alpha))$ decreases by $1$ as $\alpha$ grows through $\alpha^\ast$. In particular, under \eqref{stab} the $\theta$-periodic solution $\phi_1(\alpha)$ of \eqref{deq} is unstable for $\alpha<\alpha^\ast$ and exponentially stable for $\alpha>\alpha^\ast$. 
	\end{itemize}
\end{corollary}
\begin{proof}
	The argument parallels the proof for \cref{bifurmorse} (with $\Gamma_1$ instead of $\Gamma$), except we now apply \tref{thmcrosbif}(c) in order to deduce $\dot\rho(0)=g_{11}$: 
	
	(a) In case $g_{11}>0$, the Morse index along $\Gamma_1$ increases by $1$ as $s$ grows through the value $0$. Due to $\phi_1(\alpha)=\gamma_1(\alpha-\alpha^\ast)$ and \eqref{gamma1ass}, this yields 
	\begin{align}
		m_\ast(\phi_1(\alpha))&=m_\ast(\phi^\ast)\text{ for all }\alpha\leq\alpha^\ast,&
		m^\ast(\phi_1(\alpha))&=m^\ast(\phi^\ast)\text{ for all }\alpha\geq\alpha^\ast,
		\label{morse1}
	\end{align}
	the claimed (stability) assertions for the variational equation \eqref{var} along $\phi=\phi_1$. 

	(b) The argument in case $g_{11}<0$ is dual. In particular, 
	\begin{align}
		m_\ast(\phi_1(\alpha))&=m_\ast(\phi^\ast)\text{ for all }\alpha\geq\alpha^\ast,&
		m^\ast(\phi_1(\alpha))&=m^\ast(\phi^\ast)\text{ for all }\alpha\leq\alpha^\ast
		\label{morse2}
	\end{align}
	implies the assertions. 
\end{proof}

Further information on the derivatives of the right-hand sides $\sF_t$ yields a more detailed description of the local branch $\Gamma_2$, and we encounter two well-known bifurcation patterns. In the generic case, also $\Gamma_2$ can be represented as graph of a function over the $\alpha$-axis (see \fref{fig5}):
\begin{figure}[ht]
	\begin{center}
		\includegraphics[scale=0.6]{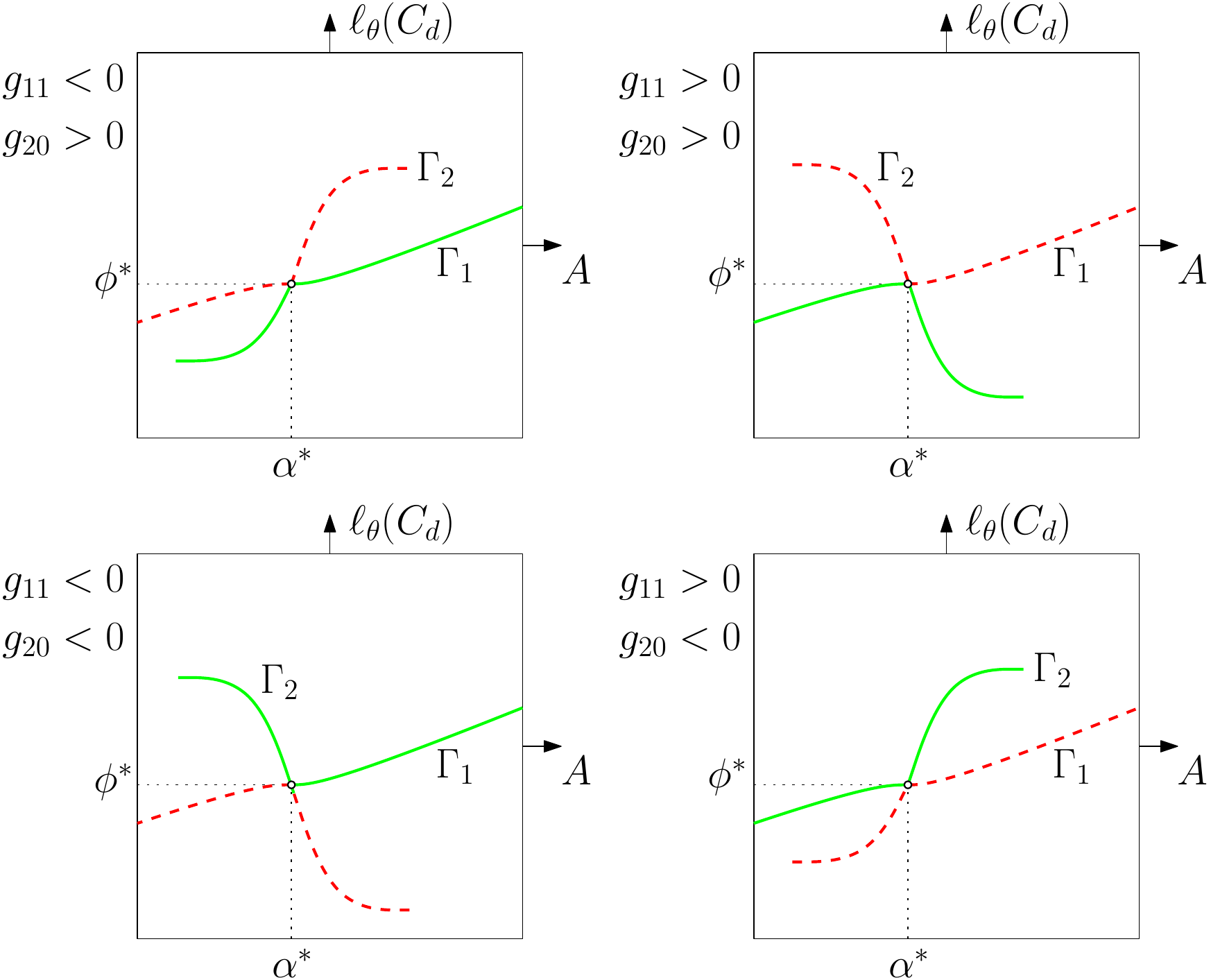}
	\end{center}
	\caption{Transcritical bifurcation of $\theta$-periodic solutions to \eqref{deq} from a branch $\Gamma_1$ into $\Gamma_2$ described in \pref{proptrans}, as well as the exchange of stability from unstable (dashed line) to exponentially stable (solid)}
	\label{fig5}
\end{figure}
\begin{prop}[transcritical bifurcation]\label{proptrans}
	Under the additional assumption
	\begin{align*}
		g_{20}
		=&
		\sum_{t=0}^{\theta-1}\int_\Omega\sprod{\eta^\ast_{t+1}(x),[D_1^2\sF_t(\phi_t^\ast,\alpha^\ast)(\xi^\ast_t)^2](x)}\d\mu(x)
		\neq
		0
	\end{align*}
	we obtain that $\Gamma_2=\set{(\phi_2(\alpha),\alpha)\in\ell_\theta(C_d)\tm\R:\,\alpha\in A_0}$ holds with a $C^{m-1}$-function $\phi_2:A_0\to B_\eps(\phi^\ast)$ of $\theta$-periodic solutions $\phi_2(\alpha)$ to \eqref{deq} satisfying
	\begin{align*}
		\phi_2(\alpha^\ast)&=\phi^\ast,&
		\dot\phi_2(\alpha^\ast)&=-2\tfrac{g_{11}}{g_{20}}\hat\xi^\ast.
	\end{align*}
	Locally in $\cB_\eps(\phi^\ast)\tm A_0$, $\phi_2(\alpha)$ is the unique $\theta$-periodic solution of \eqref{deq} distinct from $\phi_1(\alpha)$ for $\alpha\neq\alpha^\ast$, and $\phi^\ast$ is the unique $\theta$-periodic solution of $(\Delta_{\alpha^\ast})$ (cf.~\fref{fig5}). Furthermore, in case \eqref{morse} one additionally has: 
	\begin{itemize}
		\item[(a)] If $g_{11}>0$, then $m_\ast(\phi_2(\alpha))$ decreases by $1$ as $\alpha$ grows through~$\alpha^\ast$. In particular, under \eqref{stab} the $\theta$-periodic solution $\phi_2(\alpha)$ of \eqref{deq} is unstable for $\alpha<\alpha^\ast$ and exponentially stable for $\alpha>\alpha^\ast$. 

		\item[(b)] If $g_{11}<0$, then $m_\ast(\phi_2(\alpha))$ increases by $1$ as $\alpha$ grows through~$\alpha^\ast$. In particular, under \eqref{stab} the $\theta$-periodic solution $\phi_2(\alpha)$ of \eqref{deq} is exponentially stable for $\alpha<\alpha^\ast$ and unstable for $\alpha>\alpha^\ast$. 
	\end{itemize}
\end{prop}
\begin{proof}
	Note that \cref{cortransbif} immediately applies, since \lref{lemgder} guarantees
	\begin{align*}
		z'(G_{20}(\hat\xi^\ast)^2)
		&=
		\iprod{\hat\eta^\ast,G_{20}(\hat\xi^\ast)^2}_\theta\\
		&=
		\sum_{t=0}^{\theta-1}\int_\Omega\sprod{\eta^\ast_{t+1}(x),\left[D_1^2\sF_t(\phi_t^\ast,\alpha)(\xi^\ast_t)^2\right](x)}\d\mu(x) = g_{20}.
	\end{align*}
	Given the branch $\Gamma_2$ from \tref{thmcross}(b), one has $\dot\alpha_2(0)=-\tfrac{g_{20}}{2g_{11}}\neq 0$. Therefore, we can, thanks to the inverse function theorem, define $\phi_2(\alpha):=\gamma_2(\alpha_2^{-1}(\alpha))$ in a neighborhood of $\alpha^\ast$, which w.l.o.g.\ will also be denoted as $A_0$. By construction, each $\phi_2(\alpha)$ is a $\theta$-periodic solution of \eqref{deq}, and has the properties
	\begin{align*}
		\phi_2(\alpha^\ast)&=\phi^\ast,&
		\dot\phi_2(\alpha^\ast)&=\tfrac{1}{\dot\alpha_2(0)}\dot\gamma_2(\alpha^\ast)=-2\tfrac{g_{11}}{g_{20}}\hat\xi^\ast.
	\end{align*}
	The statements on the Morse index along $\Gamma_2$ result from \lref{lemma:eigder} as in the proof of \cref{bifurmorse}. In the notation used there, \cref{cortransbif} leads to $\dot\rho(0)=\tfrac{g_{20}}{2}$ along $\Gamma_2$. 
	\begin{itemize}
		\item In case $g_{20}>0$ the Morse index increases. If $-\tfrac{g_{20}}{g_{11}}>0$ (equivalently $g_{11}<0$), then $\alpha_2^{-1}$ increases and we conclude an increase in the Morse index of $\phi_2(\alpha)$ as $\alpha$ grows through $\alpha^\ast$. If $-\tfrac{g_{20}}{g_{11}}<0$ (equivalently $g_{11}>0$), then $\alpha_2^{-1}$ decreases and we derive a decrease in the Morse index of $\phi_2(\alpha)$ as $\alpha$ grows through $\alpha^\ast$. 

		\item In case $g_{20}<0$ the Morse index decreases. If $-\tfrac{g_{20}}{g_{11}}>0$ (equivalently $g_{11}>0$), then $\alpha_2^{-1}$ increases and we conclude a decrease in the Morse index of $\phi_2(\alpha)$ as $\alpha$ grows through $\alpha^\ast$. If $-\tfrac{g_{20}}{g_{11}}<0$ (equivalently $g_{11}<0$), then $\alpha_2^{-1}$ decreases and we derive an increase in the Morse index of $\phi_2(\alpha)$ as $\alpha$ grows through $\alpha^\ast$. 
	\end{itemize}
	In conclusion, the sign of the coefficient $g_{11}$ determines an increase resp.\ decrease in $m_\ast(\phi_2(\alpha))$ as $\alpha$ grows through $\alpha^\ast$. 
\end{proof}

\begin{figure}[ht]
	\begin{center}
		\includegraphics[scale=0.6]{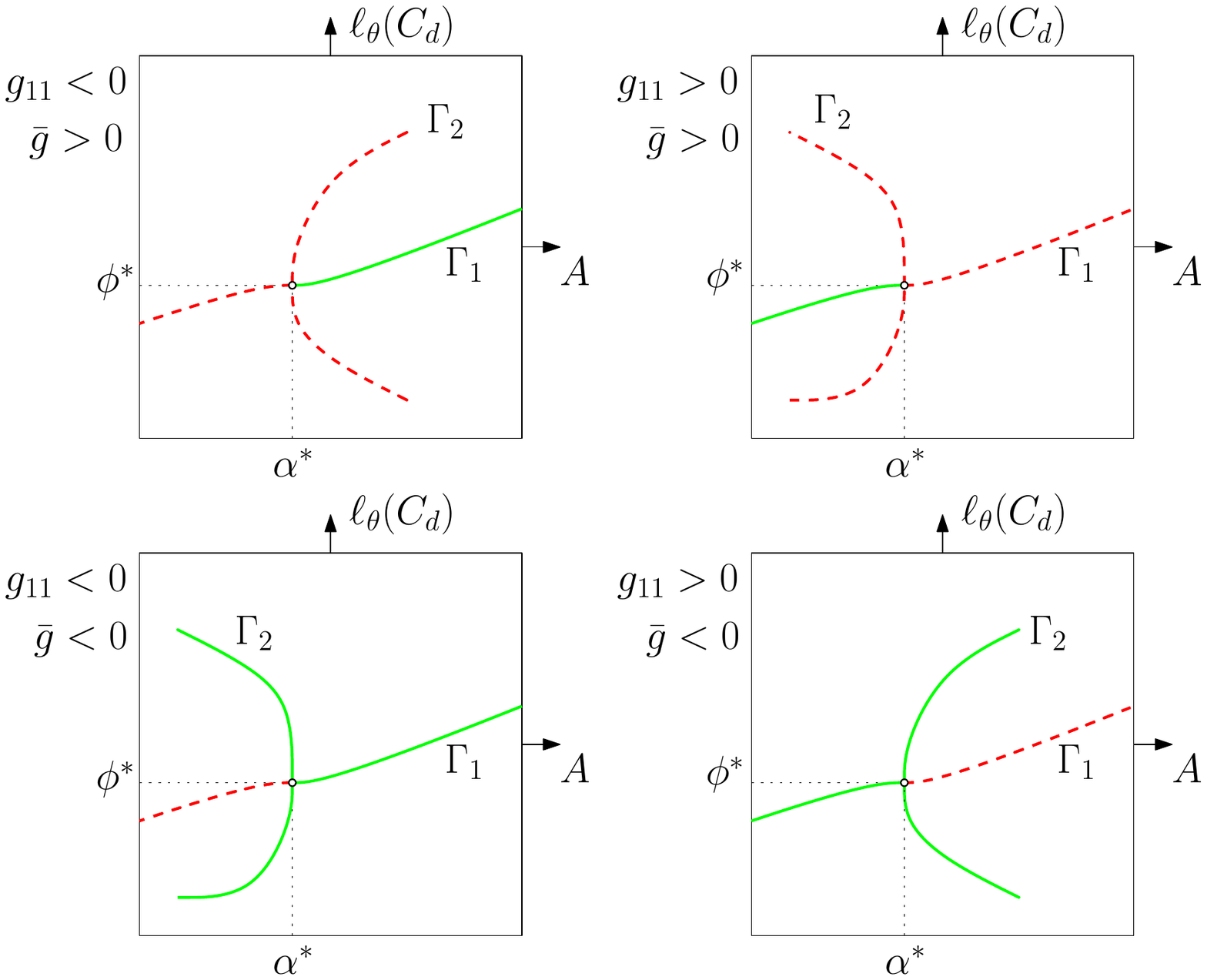}
	\end{center}
	\caption{Subcritical $(\bar g/g_{11}>0)$ and supercritical $(\bar g/g_{11}<0)$ pitchfork bifurcation of $\theta$-periodic solutions to \eqref{deq} from a branch $\Gamma_1$ into $\Gamma_2$ described in \pref{proppitch}, as well as the exchange of stability from unstable (dashed line) to exponentially stable (solid)}
	\label{fig6}
\end{figure}
The complementary situation $g_{20}=0$ leads to a, in general, nongeneric bifurcation where the branch $\Gamma_2$ can no longer be written as graph over the $\alpha$-axis (see \fref{fig6}). However, symmetry properties of the right-hand side of \eqref{deq} might enforce $g_{20}=0$ to hold. For instance, the condition that
\begin{equation}
	u\mapsto\sF_t(\phi_t^\ast-u,\alpha^\ast)+\sF_t(\phi_t^\ast+u,\alpha^\ast)
	\quad\text{is affine-linear}
	\label{symcond}
\end{equation}
implies $D_1^2\sF_t(\phi_t^\ast,\alpha^\ast)=0$ for all $0\leq t<\theta$. Moreover, the subsequent result is also crucial in the framework of \sref{sec43}.
\begin{prop}[pitchfork bifurcation]\label{proppitch}
	Let $m\geq 3$, and suppose $\bar\psi\in\ell_\theta(C_d)$ is the uniquely determined solution of the linear system of Fredholm integral equations
	\begin{equation}
		\begin{cases}
			\bar\psi_0=D_1\sF_{\theta-1}(\phi_{\theta-1}^\ast,\alpha^\ast)\bar\psi_{\theta-1}+D_1^2\sF_{\theta-1}(\phi_{\theta-1}^\ast,\alpha^\ast)(\xi^\ast_{\theta-1})^2,\\
			\bar\psi_1=D_1\sF_0(\phi_0^\ast,\alpha^\ast)\bar\psi_0+D_1^2\sF_0(\phi_0^\ast,\alpha^\ast)(\xi^\ast_0)^2,\\
			\quad\vdots\\
			\bar\psi_{\theta-1}=D_1\sF_{\theta-2}(\phi_{\theta-2}^\ast,\alpha^\ast)\bar\psi_{\theta-2}+D_1^2\sF_{\theta-2}(\phi_{\theta-2}^\ast,\alpha^\ast)(\xi^\ast_{\theta-2})^2,\\
			0=\sum_{t=0}^{\theta-1}\int_\Omega\sprod{\eta^\ast_{t}(x),\bar\psi_t(x)}\d\mu(x).
		\end{cases}
		\label{wbar}
	\end{equation}
	Under the additional assumptions
	\begin{align}
		&
		\sum_{t=0}^{\theta-1}\int_\Omega\sprod{\eta^\ast_{t+1}(x),[D_1^2\sF_t(\phi_t^\ast,\alpha^\ast)(\xi^\ast_t)^2](x)}\d\mu(x)
		=
		0,
		\label{crossrel}\\
		\bar g
		:=&
		\sum_{t=0}^{\theta-1}\int_\Omega\sprod{\eta^\ast_{t+1}(x),[D_1^3\sF_t(\phi_t^\ast,\alpha^\ast)(\xi^\ast_t)^3](x)}\d\mu(x)
		\notag\\
		&+
		3\sum_{t=0}^{\theta-1}\int_\Omega\sprod{\eta^\ast_{t+1}(x),[D_1^2\sF_t(\phi_t^\ast,\alpha^\ast)\xi^\ast_t\bar\psi_t](x)}\d\mu(x)
		\neq
		0
		\notag
	\end{align}
	we have $\dot\alpha_2(0)=0$, $\ddot\alpha_2(0)=-\frac{\bar g}{3g_{11}}$, $\ddot\gamma_2(0)=\bar\psi$, and locally in $\cB_\eps(\phi^\ast)\tm A_0$ (cf.~\fref{fig6}):
	\begin{itemize}
		\item[(a)] \emph{Subcritical case}: If $\bar g/g_{11}>0$, then $\phi_1(\alpha)$ is the unique $\theta$-periodic solution for $\alpha\geq\alpha^\ast$, and \eqref{deq} possesses exactly two $\theta$-periodic solutions distinct from $\phi_1(\alpha)$ for $\alpha<\alpha^\ast$. 
		In case \eqref{morse}, the Morse index along $\Gamma_2$ is given by $m_\ast(\phi_1(\alpha))+\sgn g_{11}$ for all $\alpha<\alpha^\ast$. 
		In particular, under \eqref{stab}, $g_{11}>0$ implies that $\Gamma_2$ is unstable, while $g_{11}<0$ implies exponential stability of $\Gamma_2$,

		\item[(b)] \emph{Supercritical case}: If $\bar g/g_{11}<0$, then $\phi_1(\alpha)$ is the unique $\theta$-periodic solution for $\alpha\leq\alpha^\ast$, and \eqref{deq} possesses exactly two $\theta$-periodic solutions distinct from $\phi_1(\alpha)$ for $\alpha>\alpha^\ast$. 
		In case \eqref{morse}, the Morse index along $\Gamma_2$ is given by $m_\ast(\phi_1(\alpha))-\sgn g_{11}$ for all $\alpha>\alpha^\ast$. 
		In particular, under \eqref{stab}, $g_{11}>0$ implies that $\Gamma_2$ is exponentially stable, while $g_{11}<0$ implies instability of $\Gamma_2$, 
	\end{itemize}
	where the occurring derivatives are given by \eqref{derf20} and
	\begin{align*}
		&[D_1^3\sF_t(\phi_t^\ast,\alpha^\ast)v^3](x)
		\\
		&
		=D_2^3G_t\intoo{x,\int_\Omega f_t(x,y,\phi_t^\ast(y),\alpha^\ast)\d\mu(y),\alpha^\ast}
		\intoo{\int_\Omega D_3f_t(x,y,\phi_t^\ast(y),\alpha^\ast)v(y)\d\mu(y)}^3\\
		&+3D_2^2G_t\intoo{x,\int_\Omega f_t(x,y,\phi_t^\ast(y),\alpha^\ast)\d\mu(y),\alpha^\ast}\intoo{\int_\Omega D_3f_t(x,y,\phi_t^\ast(y),\alpha^\ast)v(y)\d\mu(y)}
		\\
		&
		\qquad\intoo{\int_\Omega D_3^2f_t(x,y,\phi_t^\ast(y),\alpha^\ast)v(y)^2\d\mu(y)}\\
		&+D_2G_t\intoo{x,\int_\Omega f_t(x,y,\phi_t^\ast(y),\alpha^\ast)\d\mu(y),\alpha^\ast}\int_\Omega D_3^3f_t(x,y,\phi_t^\ast(y),\alpha^\ast)v(y)^3\d\mu(y)
		%		\label{derf30}
	\end{align*}
	for all $t\in\Z$, $x\in\Omega$ and $v\in C_d$. 
\end{prop}
We point out that under the symmetry condition \eqref{symcond}, the linear system \eqref{wbar} has the trivial solution $\bar\psi$, and the expression for $\bar g$ simplifies to the first sum.
\begin{proof}
	(I) \cref{corpitch} applies, since \lref{lemgder} guarantees
	\begin{align*}
		z'(G_{20}(\hat\xi^\ast)^2)
		=
		\iprod{\hat\eta^\ast,G_{20}(\hat\xi^\ast)^2}_\theta
		=
		\sum_{t=0}^{\theta-1}\int_\Omega\sprod{\eta^\ast_{t+1}(x),\left[D_1^2\sF_t(\phi_t^\ast,\alpha)(\xi^\ast_t)^2\right](x)}\d\mu(x)% = g_{20}.
	\end{align*}
	and
	\begin{align*}
		z'(G_{30}(\hat\xi^\ast)^3)+3z'(G_{20}\hat\xi^\ast\hat{\bar\psi})
		&=
		\iprod{\hat\eta^\ast,G_{30}(\hat\xi^\ast)^3}_\theta + 3\iprod{\hat\eta^\ast,G_{20}\hat\xi^\ast\hat{\bar\psi}}_\theta\\
		&=
		\sum_{t=0}^{\theta-1}\int_\Omega\sprod{\eta^\ast_{t+1}(x),\left[D_1^3\sF_t(\phi_t^\ast,\alpha)(\xi^\ast_t)^3\right](x)}\d\mu(x) \\
		&+ 3\sum_{t=0}^{\theta-1}\int_\Omega\sprod{\eta^\ast_{t+1}(x),\left[D_1^2\sF_t(\phi_t^\ast,\alpha)\xi^\ast_t\bar\psi_t\right](x)}\d\mu(x) = \bar{g}.
	\end{align*}
	The claim on the solution structure in $\cB_\eps(\phi^\ast)\tm A_0$ follows with the solutions $\phi_1(\alpha)=\gamma_1(\alpha_1^{-1}(\alpha))\in\ell_\theta(C_d)$ of \eqref{deq} from \tref{thmcross}(a) and the properties of $\Gamma_2$ determined by $\dot\alpha_2(0)=0$, $\ddot\alpha_2(0)=-\frac{\bar g}{3g_{11}}$, which contains the remaining $\theta$-periodic solutions. 
	
	(II) Along the branch $\Gamma_2$, we get from \lref{lemma:eigder} and \cref{corpitch} that $\ddot\lambda(0)=\tfrac{3}{2}\bar g$. For $\bar g>0$, the critical Floquet multiplier leaves the closed unit disk for $s\neq 0$, and we thus obtain a Morse index along $\Gamma_2$ given by
	\begin{equation}
		m_\ast(\Gamma_2)=m_\ast(\phi^\ast)+1.
		\label{no428}
	\end{equation}
	Conversely, for $\bar g<0$ the critical Floquet multiplier enters the open unit disk for $s\neq 0$, and thus we obtain an upper Morse index along $\Gamma_2$ given by
	\begin{equation}
		m^\ast(\Gamma_2)=m^\ast(\phi^\ast)-1.
		\label{no429}
	\end{equation}

	(III) Suppose that \eqref{morse} holds. First, let $\tfrac{\bar g}{g_{11}}>0$ and $\alpha<\alpha^\ast$. The case $g_{11}>0$ requires $\bar g>0$ and thus
	$
		m_\ast(\Gamma_2)
		\stackrel{\eqref{no428}}{=}
		m_\ast(\phi^\ast)+1
		\stackrel{\eqref{morse1}}{=}
		m_\ast(\phi_1(\alpha))+1.
	$
	The case $g_{11}<0$ leads to $\bar g<0$. Since the solutions on $\Gamma_2$ for $s\neq 0$ resp.\ the solutions $\phi_1(\alpha)$ are hyperbolic, it follows that
	$$
		m_\ast(\Gamma_2)
		=
		m^\ast(\Gamma_2)
		\stackrel{\eqref{no429}}{=}
		m^\ast(\phi^\ast)-1
		\stackrel{\eqref{morse2}}{=}
		m^\ast(\phi_1(\alpha))-1
		=
		m_\ast(\phi_1(\alpha))-1,
	$$
	which settles the subcritical situation. Second, let $\tfrac{\bar g}{g_{11}}<0$ and $\alpha>\alpha^\ast$. The case $g_{11}>0$ requires $\bar g<0$, and so hyperbolicity of the solution branches implies
	$$
		m_\ast(\Gamma_2)
		=
		m^\ast(\Gamma_2)
		\stackrel{\eqref{no429}}{=}
		m^\ast(\phi^\ast)-1
		\stackrel{\eqref{morse1}}{=}
		m^\ast(\phi_1(\alpha))-1
		=
		m_\ast(\phi_1(\alpha))-1.
	$$
	The case $g_{11}<0$ leads to $\bar g>0$ and
	$
		m_\ast(\Gamma_2)
		\stackrel{\eqref{no428}}{=}
		m_\ast(\phi^\ast)+1
		\stackrel{\eqref{morse2}}{=}
		m_\ast(\phi_1(\alpha))+1, 
	$
	which settles also the supercritical situation. 

	(IV) Stability properties of the solutions on $\Gamma_2$ result from \cref{trbif}. 
\end{proof}

Although our assumptions were formulated so as not to require the explicit knowledge of $\theta$-periodic solutions $\phi_0(\alpha)$ to \eqref{deq} in a whole neighborhood $A_0\subseteq A$ of $\alpha^\ast$, such knowledge is exceedingly helpful. 
\begin{remark}[equation of perturbed motion]\label{remshift}
	Let $\phi: A_0\to B_\eps(\phi^\ast)\subset\ell_\theta(C_d)$, $\eps>0$, be of class $C^m$ with $\phi(\alpha^\ast)=\phi^\ast$, and consider the \emph{equation of perturbed motion}
	\begin{equation}
		\tag{$\tilde\Delta_\alpha$}
		\boxed{u_{t+1} = \tilde\sF_t(u_t,\alpha) := \sF_t(u_t+\phi(\alpha)_t,\alpha) - \phi(\alpha)_{t+1},}
		\label{deqp}
	\end{equation}
	which is $\theta$-periodic, has the trivial solution for all $\alpha\in A_0$ and satisfies $(H_1$--$H_2)$ on a (possibly smaller) neighborhood $A_0$ of $\alpha^\ast$. For $1\leq i\leq m$ and $t\in\Z$ one obtains
	\begin{align*}
		D_1^i\tilde\sF_t(0,\alpha^\ast) & = D_1^i\sF_t(\phi_t^\ast,\alpha^\ast),&
		D_2^i\tilde\sF_t(0,\alpha^\ast) & = 0, \\
		D_1D_2\tilde\sF_t(0,\alpha^\ast) & = D_1D_2\sF_t(\phi_t^\ast,\alpha^\ast) + D_1^2\sF_t(\phi_t^\ast,\alpha^\ast)\dot\phi(\alpha^\ast)_t,
	\end{align*}
	and thus $(\tilde\Delta_\alpha)$ satisfies $(B_3)$. Besides the value $\dot\phi(\alpha^\ast)\in C_d^\theta$, no further information on $\phi$ is needed; $\dot\phi(\alpha^\ast)$ is given by \tref{poinc}(b) if $1\notin\sigma_\theta(\alpha^\ast)$, which in turn typically holds provided the bifurcating branch is of period greater than $\theta$. The punch line is that \eqref{deq} undergoes a crossing curve bifurcation (in the sense of \tref{thmcross}) at $(\phi^\ast,\alpha^\ast)$ if and only if \eqref{deqp} does at $(0,\alpha^\ast)$. The same holds for transcritical and pitchfork bifurcations ({\`a} la \pref{proptrans} resp.\ \pref{proppitch}). With $\tilde\Gamma_2=\svector{\gamma_2}{\alpha_2}(S)$ as the nontrivial solution branch of equation $(\tilde\Delta_\alpha)$ near $(0,\alpha^\ast)$, the branches of \eqref{deq} around $(\phi^\ast,\alpha^\ast)$ are $\Gamma_1 = \{(\phi(\alpha),\alpha)\in\ell_\theta(C_d)\tm\R:\,\alpha\in A_0\}$ and $\Gamma_2 = \svector{\gamma_2 + \phi \circ \alpha_2}{\alpha_2}(S)$.
\end{remark}
\subsection{Period doubling}
\label{sec43}
Assume that a solution $\phi^\ast\in\ell_\theta(C_d)$ of a $\theta$-periodic difference equation $(\Delta_{\alpha^\ast})$ possesses a Floquet multiplier $\nu\neq 1$, but $\nu^l=1$ holds for some $l\in\N$. Then the spectral mapping theorem shows $1\in\sigma_{l\theta}(\alpha^\ast)$, which means that \eqref{nohyp} holds with the period $\theta$ replaced by $l\theta$. Therefore, provided the further assumptions of our above results from Sects.~\ref{sec41}--\ref{sec42} are satisfied, they ensure that $l\theta$-periodic solutions to \eqref{deq} bifurcate from $\phi^\ast$. 

The case $l=2$ (i.e.\ a Floquet multiplier $-1$) deserves particular attention. Since $\Xi_{2\theta}(\alpha^\ast)$ has the eigenvalue $(-1)^2=1$, the $2\theta$-periodic sequences $\xi^\ast$ and $\eta^\ast$ from assumption $(B_2)$ actually satisfy
\begin{align*}
	\xi_{t+\theta}^\ast&=-\xi_t^\ast,&
	\eta_{t+\theta}^\ast&=-\eta_t^\ast\fall t\in\Z
\end{align*}
as a result of \pref{proplin}(c) and \pref{proplindual}(c). Due to the $\theta$-periodicity of $\phi^\ast$ and $\sF_t$, this implies 
\begin{align*}
	g_{20}
	=&
	\sum_{t=0}^{\theta-1}\int_\Omega\sprod{\eta^\ast_{t+1}(x),[D_1^2\sF_t(\phi_t^\ast,\alpha^\ast)(\xi^\ast_t)^2](x)}\d\mu(x)\\
	&+
	\sum_{t=0}^{\theta-1}\int_\Omega\sprod{\eta^\ast_{t+\theta+1}(x),[D_1^2\sF_{t+\theta}(\phi_{t+\theta}^\ast,\alpha^\ast)(\xi_{t+\theta}^\ast)^2](x)}\d\mu(x)
	=
	0
\end{align*}
and consequently neither \tref{thmfold} (fold bifurcation) nor \pref{proptrans} (transcritical bifurcation) apply. This leads to the territory of \pref{proppitch} guaranteeing a pitchfork bifurcation of $2\theta$-periodic solutions from $\phi^\ast\in\ell_\theta(C_d)$. For this reason, one speaks of a \emph{flip} or \emph{period doubling bifurcation}. Here, the bifurcation indicators simplify to
$$
	g_{11}=2\sum_{t=0}^{\theta-1}\int_\Omega\sprod{\eta^\ast_{t+1}(x),[D_1D_2\sF_t(\phi_t^\ast,\alpha^\ast)\xi^\ast_t](x)}\d\mu(x),
$$
and since every solution $\bar\psi\in\ell_\theta(C_d)$ of \eqref{wbar} also solves this linear equation for $\theta$ replaced by $2\theta$, one obtains
	\begin{align*}
		\bar g
		=&
		2\sum_{t=0}^{\theta-1}\int_\Omega\sprod{\eta^\ast_{t+1}(x),[D_1^3\sF_t(\phi_t^\ast,\alpha^\ast)(\xi^\ast_t)^3](x)}\d\mu(x)\\
		&+
		6\sum_{t=0}^{\theta-1}\int_\Omega\sprod{\eta^\ast_{t+1}(x),[D_1^2\sF_t(\phi_t^\ast,\alpha^\ast)\xi^\ast_t\bar\psi_t](x)}\d\mu(x)
		\neq
		0, 
	\end{align*}
i.e.\ the bifurcation indicators double their values. 
\subsection{Bifurcation of symmetric solutions}
This subsection explains how symmetry properties of the functions $G_t,f_t$ in $x,y$ affect the bifurcating branches. 
Suppose the habitat $\Omega\subset\R^\kappa$ is symmetric, i.e. $\Omega=-\Omega$. Given $u\in C_d$, we define the involution $u^-(x):=u(-x)$ and the projections
\begin{align*}
	u^o&:=\tfrac{1}{2}(u-u^-),& 
	u^e&:=\tfrac{1}{2}(u+u^-)
\end{align*}
on $C_d$. This yields a unique decomposition of $u$ into its \emph{odd} and \emph{even parts}, respectively, that is, if odd resp.\ even functions $v^o$, $v^e\in C_d$ with $u = v^o + v^e$ exist, then $v^o = u^o$ and $v^e = u^e$. Inspired by this, we will call an IDE \eqref{deq} \emph{even} if 
\begin{align*}
	G_t(-x,z_2,\alpha)&=G_t(x,z_2,\alpha),&
	f_t(-x,-y,z_1,\alpha)&=f_t(x,y,z_1,\alpha)
\end{align*}
holds for all $0\leq t<\theta_0$, $x,y\in\Omega$, $z_1\in U^1_t$, $z_2\in U^2_t$ and $\alpha\in A$.

\begin{prop}\label{prop_odd-even}
	If \eqref{deq} is even, then the following hold: 
	\begin{enumerate}
		\item $\phi$ is a solution of \eqref{deq} if and only if $\phi^-$ is.

		\item Assume $u\in U_t$ is even. Then for every $i,j\in\N_0$, $i+j\leq m$, the function $D_1^{i}D_2^{j}\sF_t(u,\alpha)v^1\cdots v^i\in C_d$ is odd if an odd number of $v^k$, $1\leq k\leq i$ are odd and the rest are even, and even if an even number of $v^k$, $1\leq k\leq i$ are odd and the rest are even. 
	\end{enumerate}
\end{prop}
As a consequence of (b), it holds that for even IDEs, $\sF_t(u,\alpha)$ is an even function whenever $u$ is, that is, the set $\set{u\in U_t \, | \, \text{$u_t$ is even}}$ is positively invariant under $\sF_t(\cdot,\alpha)$ for each $0\leq t<\theta_0$, $\alpha\in A$, justifying the choice of nomenclature. 

Given $\theta\in\N$, we denote a sequence $\phi\in\ell_\theta(C_d)$ as \emph{odd} if each of its members $\phi_t\in C_d$, $t\in\Z$, is odd, and \emph{even} if all members are even. Due to
$$
	\phi_t^- 
	=
	\tfrac{1}{2}(\phi_t+\phi_t^-) - \tfrac{1}{2}(\phi_t-\phi_t^-)
	=
	\phi_t^e - \phi_t^o,
$$
it follows that for any solution $\phi$ of \eqref{deq} which is not even, $\phi^-$ is a distinct solution.
\begin{proof}
	Let $\alpha\in A$. 

	(a) If $\phi$ is a solution of \eqref{deq}, then for all $x\in\Omega$ one has
\begin{align*}
	\phi_{t+1}^-(x) 
	& \stackrel{\eqref{ury}}{=} 
	G_t\intoo{-x,\int_\Omega f_t(-x,y,\phi_t(y),\alpha)\d\mu(y),\alpha} \\
	& = 
	G_t\intoo{-x,\int_\Omega f_t(-x,-y,\phi_t(-y),\alpha)\d\mu(y),\alpha} \\
	& = 
	G_t\intoo{x,\int_\Omega f_t(x,y,\phi_t(-y),\alpha)\d\mu(y),\alpha}
	\stackrel{\eqref{ury}}{=}
	\sF_t(\phi_t^-,\alpha)(x), 
\end{align*}
where we have used the change of variables formula \cite[p.~332, 9.3.1~Thm.]{rana:02} applied with the reflection $\rho: x\in\Omega\mapsto-x\in\Omega$ satisfying $\rho(\Omega) = -\Omega = \Omega$.

	(b) is again an immediate consequence of the above change of variables formula together with the observation that both $D_2^{i}D_3^{j}G_t(-x,z_2,\alpha) = D_2^{i}D_3^{j}G_t(x,z_2,\alpha)$ and $D_3^{i}D_4^{j}f_t(-x,-y,z_1,\alpha) = D_3^{i}D_4^{j}f_t(x,y,z_1,\alpha)$ hold for all orders $i+j\leq m$. 
\end{proof}
\begin{corollary}\label{cor_eigv-odd-or-even}
	Consider an even, $\theta$-periodic solution $\phi^\ast$ of $(\Delta_{\alpha_\ast})$. If $(\lambda,\xi_0)$ is a simple eigenpair of $\Xi_{\theta}(\alpha^\ast)$, then the eigenfunction $\xi_0$ is either odd or even. Moreover, if $(\lambda,\eta_0)$ is a simple eigenpair of $\Xi_{\theta}(\alpha^\ast)'$, then $\eta_0$ is either odd or even.
\end{corollary}
Although this result fits into symmetric bifurcation theory, the following ad hoc proof is available. 
\begin{proof}
	Assume $\xi_0\in C_d$ is neither odd nor even, and consider $\xi_0^o$ and $\xi_0^e$, its unique decomposition into non-zero odd and even parts. By construction, we have
	$$
	\lambda\xi_0^o + \lambda\xi_0^e = \lambda\xi_0 = \Xi_{\theta}(\alpha^\ast)\xi_0 = \Xi_{\theta}(\alpha^\ast)\xi_0^o + \Xi_{\theta}(\alpha^\ast)\xi_0^e.
	$$
	By \pref{prop_odd-even} and finite induction, $\Xi_{\theta}(\alpha^\ast)\xi_0^o$ is odd and $\Xi_{\theta}(\alpha^\ast)\xi_0^e$ is even. Thanks to the uniqueness of the decomposition, one necessarily has $\Xi_{\theta}(\alpha^\ast)\xi_0^o = \lambda\xi_0^o$ and $\Xi_{\theta}(\alpha^\ast)\xi_0^e=\lambda\xi_0^e$. Due to $\xi_0^o\neq 0$ and $\xi_0^e\neq 0$, it follows that $(\lambda,\xi_0^o)$ and $(\lambda,\xi_0^e)$ are distinct eigenpairs of $\Xi_{\theta}(\alpha^\ast)$, contradicting the simplicity of the eigenvalue $\lambda$. The claim follows for $\Xi_{\theta}(\alpha^\ast)$, and also for $\Xi_{\theta}(\alpha^\ast)'$ by an analogous argument.
\end{proof}

We will consider for any fixed multiple $\theta\in\N$ of $\theta_0$ the set
$$
	\fC 
	:=
	\set{(\phi,\alpha)\in(\ell_\theta(C_d)\setminus\set{0}) \tm A
	\left|
	\begin{array}{l} 
		\phi\text{ is an entire solution of \eqref{deq} and}\\ 
		\text{$(\phi,\alpha)$ is not a bifurcation point of \eqref{deq}}
	\end{array}
	\right.},
$$
that is, the set of $\theta$-periodic nontrivial solutions that are not $\theta$-periodic bifurcation points, and for each $\theta$-periodic bifurcation point $(\phi^\ast,\alpha^\ast)$ of \eqref{deq} the family
$$
	\fC_{(\phi^\ast,\alpha^\ast)} 
	:=
	\set{C \subset \fC 
	\left|
	\begin{array}{l}
		C\text{ is a connected component of $\fC$ containing a sequence}\\
		((\phi_n,\alpha_n))_{n\in\N}\text{ satisfying }\lim_{n\to\infty}(\phi_n,\alpha_n)=(\phi^\ast,\alpha^\ast)
	\end{array}
	\right.}.
$$
\begin{prop}\label{prop_branches-odd-even}
	Suppose \eqref{deq} is even and satisfies $\sF_t(0,\alpha)\equiv 0$ on $\Z$ for all $\alpha\in A$. If the assumptions of \tref{thmcross} are satisfied in a point $(0,\alpha^\ast)$, with $(1,\xi_0^\ast)$ a simple eigenpair of $\Xi_\theta(\alpha^\ast)$, then exactly one of the following hold:
	\begin{enumerate}
		\item $\xi_0^\ast$ is even. If the assumptions for a transcritical bifurcation from \pref{proptrans} hold, then for all $C\in\fC_{(0,\alpha^\ast)}\neq\emptyset$ and all $(\phi,\alpha)\in C$, $\phi$ is even.

		\item $\xi_0^\ast$ is odd, and $\fC_{(0,\alpha^\ast)} = \set{C_+,\,C_-}$ with disjoint branches $C_+\cap C_- = \emptyset$, where $(\phi,\alpha)\in C_+\Leftrightarrow(\phi^-,\alpha)\in C_-$ for all $(\phi,\alpha)\in \ell_\theta(C_d)\tm A$. In particular, a pitchfork bifurcation takes place at $(0,\alpha^\ast)$; as a consequence, due to the validity of \eqref{crossrel} the assumptions of \pref{proptrans} do not hold. Moreover, both branches possess the same total population at each time step, i.e.\
		\begin{multline*}
			\set{\intoo{\int_\Omega\phi_t(x)\d\mu(x),\alpha} \in \R^d\tm A | \, (\phi,\alpha)\in C_+}\\
			=
			\set{\intoo{\int_\Omega\phi_t(x)\d\mu(x),\alpha} \in \R^d\tm A | \, (\phi,\alpha)\in C_-}
			\fall t\in\Z.
		\end{multline*}	
	\end{enumerate}
\end{prop}
\begin{proof}
	That $\fC_{(0,\alpha^\ast)}$ is nonempty follows directly from \tref{thmcross}, and that $\xi_0^\ast$ is either odd or even is due to \cref{cor_eigv-odd-or-even}. We consider each case.
	
	(a) Assume first that $\xi_0^\ast$ is even, and that \pref{proptrans} applies. By \tref{thmcross}, the locally unique nontrivial solution branch is $\phi_2: A_0\to B_\eps(0)$. Assume that $\phi_2(\alpha)$ is not even for some $\alpha\in A_0$. By \pref{prop_odd-even}(a), it follows that $\phi_2(\alpha)^-$ is a distinct nontrivial solution of \eqref{deq}. Evidently, $\left\|\phi_2(\alpha)^-\right\| = \left\|\phi_2(\alpha)\right\| < \eps$, contradicting the local uniqueness of $\phi_2$. Thus, $\phi_2(\alpha)$ is even for all $\alpha\in A_0$.

	Fix now $C\in\fC_{(0,\alpha^\ast)}$, and consider the subset
	$$
		C_{ne} := \set{(u,\alpha)\in C \, | \, \text{$u$ is not even}}.
	$$
	By the above, $C_{ne}\neq C$, and by the decomposition $u = u^o + u^e$, $u^o = \tfrac{1}{2}(u-u^-)$, $C_{ne}$ is open as a subset of $C$. Assume $C_{ne}\neq\emptyset$. Now $\emptyset\neq \partial C_{ne}\subset C$, so there exists $(u^\dag,\alpha^\dag)\in\partial C_{ne}$; necessarily, $u^\dag$ is even. This allows us to select some convergent sequence $((u_n,\alpha_n))_{n\in\N}$ in $C_{ne}$ with limit $(u^\dag,\alpha^\dag)$. But by an argument similar to the above, $(u_n^-,\alpha_n)_{n\in\N}$ is now also a sequence of $\theta$-periodic solutions converging towards $((u^\dag)^-,\alpha^\dag)=(u^\dag,\alpha^\dag)$, and as $u_n$ is not even for any $n\in\N$, one has $u_n\neq u_n^-$ for all $n\in\N$. Thus, $(u^\dag,\alpha^\dag)\in C$ is a bifurcation point, which is impossible by construction. It follows that $C_{ne}=\emptyset$, that is, $u$ is even for all $(u,\alpha)\in C$.

	(b) Assume $\xi_0^\ast$ is odd. Considering the solution branch $\svector{\gamma_2}{\alpha_2}: S\to B_\eps(0) \tm A_0$, there must exist, by construction, some $C_+\in\fC_{(0,\alpha_\ast)}$ such that $(\gamma_2(s),\alpha_2(s))\in C_+$ for all sufficiently small positive $s\in S$.
	We suppose that $\gamma_2(s)$ is even for all such $s$, that is, $\gamma_2(s)^o=0$ for all positive $s\in S$ close to $0$. As $\gamma_2(s)_0 = s\xi_0^\ast + \tilde{\gamma}_2(s)_0$, one necessarily has $\tilde{\gamma}_2(s)_0^o \equiv -s\xi_0^\ast$ near $0$. Due to $\dot{\tilde{\gamma}}_2(0) = 0$, this implies $\xi_0^\ast = (\dot{\tilde{\gamma}}_2(0)_0)^e$, which is impossible, since no non-zero function is both odd and even. Hence, with 
	$$
		C_{ne}^+ := \set{(u,\alpha)\in C_+ \, | \, \text{$u$ is not even}},
	$$
	we necessarily have $C_{ne}^+\neq\emptyset$. If $C_{ne}^+\neq C_+$, then again $C_{ne}^+$ is open as a subset of $C_+$ with $\partial C_{ne}^+\neq\emptyset$, and we again obtain a contradiction. Thus $C_{ne}^+=C_+$, that is, for any $(u,\alpha)\in C_+$, $u$ is not even, implying $u\neq u^-$. As $C_+$ contains no bifurcation points, each $\alpha$ can appear at most once in $C_+$, implying $(u^-,\alpha)\notin C_+$. By \pref{prop_odd-even}, there consequently exists $C_-\in\fC_{(0,\alpha^\ast)}\setminus\set{C_+}$ so that $(u^-,\alpha)\in C_-$, and by \tref{thmcross}, $\fC_{(0,\alpha^\ast)}$ consists exactly of the two disjoint components $C_+$ and $C_-$. It follows that $(0,\alpha^\ast)$ is a pitchfork bifurcation point; therefore, \pref{proptrans} cannot apply. As the conditions of \tref{thmcross} are assumed, \pref{proptrans} can only fail if \eqref{crossrel} holds. The final statement is another consequence of the change of variables formula from \cite[p.~332, 9.3.1~Thm.]{rana:02}. 
\end{proof}
\section{Applications}
\label{sec5}
Our setting is now simpler than above, since merely single IDEs rather than systems are considered. We write $C(\Omega)$ for the Banach algebra of continuous functions $u:\Omega\to\R$ equipped with the $\sup$-norm $\norm{\cdot}$, and
$$
	\ell_\theta:=\set{(\phi_t)_{t\in\Z}:\,\phi_t\in C(\Omega)\text{ and }\phi_{t+\theta}=\phi_t\text{ for all }t\in\Z}
$$
abbreviates the space of $\theta$-periodic sequences in $C(\Omega)$. 

The subsequent types of IDEs essentially serve two purposes: First, they illustrate all types of bifurcations introduced in \sref{sec4}. Second, they demonstrate different degrees of numerical effort in order to verify the required assumptions.
\subsection{Degenerate kernels}
\label{sec51}
The example class of degenerate kernel IDEs is the simplest one, since a bifurcation analysis can be kept on a purely analytical level and does not require numerical tools. For a logistic nonlinearity in \eqref{ide1}, this is demonstrated in \cite{bramburger:lutscher:18}, while our growth functions in nature are academic and not motivated by ecology. On the parameter space $A=\R$ we study Hammerstein IDEs with right-hand side 
$$
	\sF_t(u,\alpha)=\int_\Omega k(\cdot,y)g_t(y,u(y),\alpha)\d\mu(y), 
$$ 
whose kernels can be represented as $k(x,y) = \sum_{i=1}^n k_i(y)e_i(x)$ for all $x,y\in\Omega$ and some $n\in\N$. Here, $k_1,\ldots,k_n\in C(\Omega)$ and the functions $e_1,\ldots,e_n\in C(\Omega)$ are supposed to be linearly independent. This results in $\sF_t(u,\alpha)\in\spann\set{e_1,\ldots,e_n}$ for all $u,\alpha$, and the dynamical behavior of \eqref{deq} is immediately finite-dimensional, i.e.\ completely determined by a difference equation in $\R^n$.

Nevertheless, degenerate kernel IDEs can still be considered as infinite-dimen\-sional dynamical systems, and showcase our methods for periodic equations, dual operators and so on; they are somewhat artificial, to allow for explicit computation, but non-artificial examples are largely relegated to numerics.
\begin{example}[cosine kernel]\label{excos}
	On the habitat $\Omega=[-\tfrac{L}{2},\tfrac{L}{2}]$ equipped with the Lebesgue measure $\mu$ and real numbers $a,L>0$ satisfying $aL\leq\tfrac{1}{2}$, we suppose that $k(x,y):=k_0(x-y)$ is given as finite radius dispersal kernel 
	\begin{align}
		k_0:\R&\to\R,&
		k_0(r)&:=
		\frac{\pi a}{2}
		\begin{cases}
			\cos\intoo{\pi ar},&2a\abs{r}\leq 1,\\
			0,&2a\abs{r}>1
		\end{cases}
		\label{kerfinrad}
	\end{align}
	(cf.\ \cite{bramburger:lutscher:18,kot:schaeffer:86}), which is degenerate, because it can be written as
	$$
		k_0(x-y)
		=
		\frac{\pi a}{2}
		\begin{cases}
			e_1(x)e_1(y)+e_2(x)e_2(y),&2a\abs{x-y}\leq 1,\\
			0,&2a\abs{x-y}>1
		\end{cases}
	$$
	with the odd resp.\ even functions $e_1,e_2:[-\tfrac{L}{2},\tfrac{L}{2}]\to\R$, 
	\begin{align*}
		e_1(x)&:=\cos(\pi ax),&
		e_2(x)&:=\sin(\pi ax).
	\end{align*}
	In this context, it is handy to abbreviate $\omega_{ij}:=\int_{-\tfrac{L}{2}}^{\tfrac{L}{2}}e_1(x)^ie_2(x)^j\d x$, $i,j\in\N_0$, satisfying $\omega_{ij} > 0$ for even $j$ and $\omega_{ij}=0$ otherwise. Note that in both subsequent Exs.~\ref{exfold} and \ref{exaarset} equipped with the cosine kernel, the bifurcation points only depend on the product $aL\in(0,\tfrac{1}{2}]$, and not on $a,L$ individually. 
\end{example}
\subsubsection{Fold bifurcation}
Our first example is an autonomous IDE \eqref{deq} having the right-hand side
\begin{equation}
	\sF(u,\alpha):=\int_\Omega k(\cdot,y)\intoo{2\alpha+u(y)^2}\d\mu(y)
	\label{exfold1}
\end{equation}
defined on the sets $U_t\equiv C(\Omega)$. The partial derivatives compute as
\begin{align*}
	D_1\sF(u,\alpha)v&=2\int_\Omega k(\cdot,y)u(y)v(y)\d\mu(y),&
	D_2\sF(u,\alpha)&=2\int_\Omega k(\cdot,y)\d\mu(y)
\end{align*}
and $D_1^2\sF(u,\alpha)v^2=2\int_\Omega k(\cdot,y)v(y)^2\d\mu(y)$ for $u,v\in C(\Omega)$, $\alpha\in\R$. The variational equation \eqref{var} along a fixed-point $\phi^\ast$ becomes $v_{t+1}=2\int_\Omega k(\cdot,y)\phi^\ast(y)v_t(y)\d\mu(y)$, and has the dual variational equation (see \lref{lemdual})
$$
	v_t\stackrel{\eqref{derf10s}}{=}-2\phi^\ast\int_\Omega k(y,\cdot)v_{t+1}(y)\d\mu(y). 
$$
In general, it is difficult to verify a fold bifurcation in IDEs explicitly, since the fold point $(\phi^\ast,\alpha^\ast)$ is unknown and has to be determined numerically using e.g.\ path-following techniques. This situation simplifies for specific kernels.
\begin{example}[cosine kernel]\label{exfold}
	For the cosine kernel \eqref{kerfinrad} from \eref{excos}, the inclusion $\sF(u,\alpha)\in\spann\set{e_1,e_2}$ holds for all $u$ and $\alpha$, since
$$
	\sF(u,\alpha)
	\stackrel{\eqref{exfold1}}{=}
	\frac{\pi a}{2}\sum_{i=1}^2\int_{-\tfrac{L}{2}}^{\tfrac{L}{2}}e_i(y)(2\alpha+u(y)^2)\d y\,e_i. 
$$
Thus, \eqref{deq} is an even IDE. In a fixed point $\phi^\ast$ of $(\Delta_{\alpha^\ast})$ one has
\begin{align*}
	D_1\sF(\phi^\ast,\alpha^\ast)v
	&=
	\pi a\sum_{i=1}^2\int_{-\tfrac{L}{2}}^{\tfrac{L}{2}}e_i(y)\phi^\ast(y)v(y)\d y\,e_i,\\
	D_1^2\sF(\phi^\ast,\alpha^\ast)v^2
	&=
	\pi a\sum_{i=1}^2\int_{-\tfrac{L}{2}}^{\tfrac{L}{2}}e_i(y)v(y)^2\d y\,e_i,\\
	D_2\sF(\phi^\ast,\alpha^\ast)
	&=
	\pi a\sum_{i=1}^2\int_{-\tfrac{L}{2}}^{\tfrac{L}{2}}e_i(y)\d y\,e_i
	=
	\pi a\omega_{10}e_1%2\sin\tfrac{\pi aL}{2}e_1
\end{align*}
and
$
	D_1\sF(\phi^\ast,\alpha^\ast)'v
	=
	\pi a\phi^\ast\sum_{i=1}^2\int_{-\tfrac{L}{2}}^{\tfrac{L}{2}}v(y)\d y\,e_i
$
	for all $v\in C[-\tfrac{L}{2},\tfrac{L}{2}]$. In order to apply \tref{thmfold}, consider the (potential) bifurcation point $(\phi^\ast,\alpha^\ast)$ with 
	\begin{align*}
		\phi^\ast
		&:=
		\frac{2}{a\pi\omega_{30}}e_1 =
		\frac{6}{\sin\frac{\pi a L}{2}(5 + \cos (\pi a L))}e_1,&%\frac{6}{\sin\tfrac{\pi aL}{2}(5+\cos(\pi aL))}e_1,&
		\alpha^\ast
		&:=
		\frac{3}{9-8\cos(\pi aL)-\cos(2\pi aL)}.
	\end{align*}
	It is straightforward to show $N(D_1\sF(\phi^\ast,\alpha^\ast)-I_{C(\Omega)})=\spann\set{e_1}$. Hence, $(B_1$--$B_2)$ hold, and we define $\xi^\ast_0:=e_1$. For the dual operator of the derivative, one obtains the null space
$
	N(D_1\sF(\phi^\ast,\alpha^\ast)'-I_{C(\Omega)})=\spann\set{\phi^\ast \, e_1}=\spann\set{e_1^2}.
$
Setting $\eta^\ast_0:=e_1^2$, we can verify the conditions of \tref{thmfold}. Thanks to $aL\in(0,\tfrac{1}{2}]$, one obtains
\begin{align*}
	g_{01}
	&=
	\int_{-\tfrac{L}{2}}^{\tfrac{L}{2}}\eta_0^\ast(x)[D_2\sF(\phi^\ast,\alpha^\ast)](x)\d x
	=
	\pi a \omega_{10}\omega_{30}>0,\\
	g_{20}
	&=
	\int_{-\tfrac{L}{2}}^{\tfrac{L}{2}} \eta_0^\ast(x)[D_1^2\sF(\phi^\ast,\alpha^\ast)(\xi^\ast_0)^2](x)\d x
	=
	\pi a\omega_{30}^2>0. 
\end{align*}
Therefore, \tref{thmfold}(a) implies that \eqref{deq} with the right-hand side \eqref{exfold1} undergoes a subcritical fold bifurcation at $(\phi^\ast,\alpha^\ast)$ (see \fref{figfold1} (left)). Finally, one shows that $N(D_1\sF(\phi^\ast,\alpha^\ast)-\lambda I_{C(\Omega)})=\set{0}$ for all $\lambda\in\C$ with $|\lambda|>1$, and so \cref{bifurmorse} yields that the bifurcating branch $\Gamma^-$ growing in the direction of $-e_1$ is exponentially stable for all $\alpha<\alpha^\ast$ near $\alpha^\ast$, while the bifurcating branch $\Gamma^+$ growing in the direction of $e_1$ is unstable (with Morse index $1$) for all $\alpha<\alpha^\ast$ near $\alpha^\ast$. The graph of the total population in \fref{figfold1} (right) reflects the fold. 
	\begin{figure}
		\includegraphics[width=52mm]{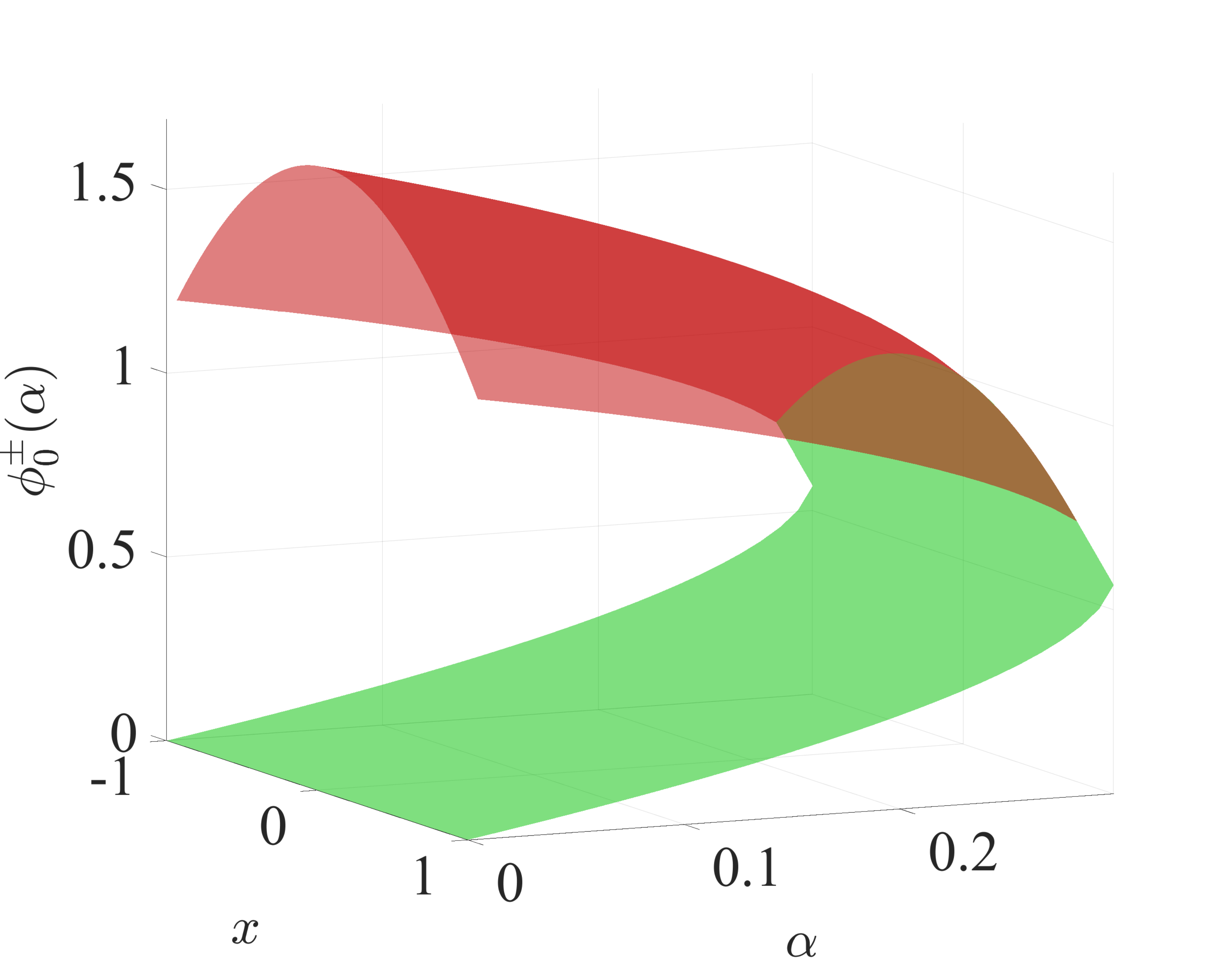}
		\includegraphics[width=52mm]{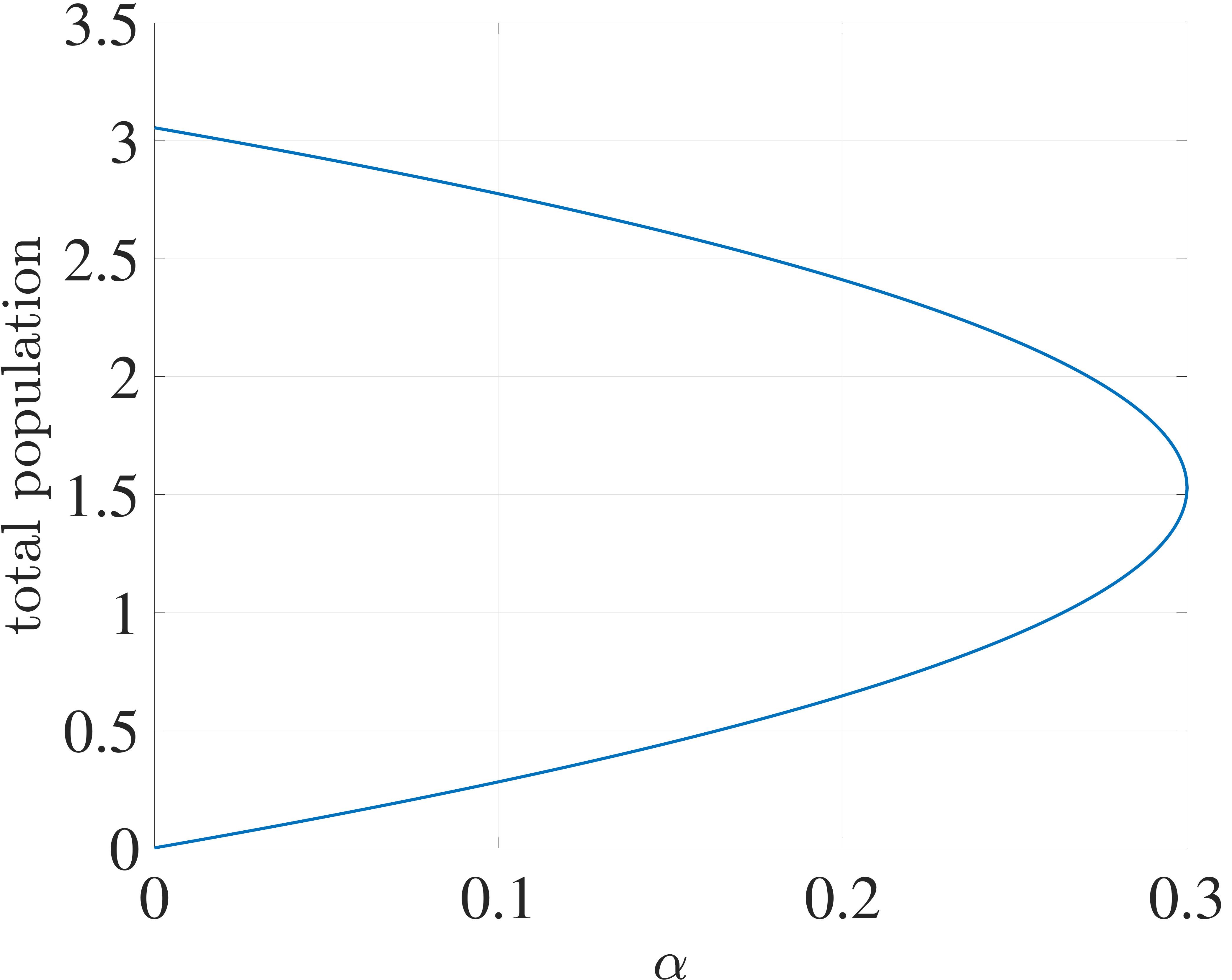}
		\caption{Branch of the subcritical fold bifurcation for \eqref{deq} with right-hand side \eqref{exfold1} and kernel \eqref{kerfinrad} (left). 
		Total population over $\alpha\in[0.0,0.3]$ with $a=\tfrac{1}{4}$, $L=2$ along the branch (right)}
		\label{figfold1}
	\end{figure}

	We point out that along the branches $\Gamma^+$ and $\Gamma^-$ further bifurcations of pitchfork and period doubling type can be observed for parameters $\alpha<\alpha^\ast$. Without providing an explicit analysis, we refer to \fref{figschem1} (left) for a schematic illustration. 
\end{example}
\subsubsection{Bifurcation of $2$-periodic solutions}
For $U_t\equiv C(\Omega)$ consider an autonomous IDE \eqref{deq} with right-hand side
\begin{equation}
	\sF(u,\alpha)
	:=
	\alpha\int_\Omega k(\cdot,y)u(y)\intoo{u(y)^2-1}\d\mu(y)
	\label{exflip1}
\end{equation}
satisfying the symmetry condition \eqref{symcond} with $\phi_t^\ast\equiv 0$ on $\Z$. A $2$-periodic solution born from a period doubling is explicitly constructed and shown to pitchfork bifurcate into solutions of the same period $2$.

For all $v\in C(\Omega)$, the derivatives of \eqref{exflip1} in the origin read as
\begin{align*}
	D_1\sF(0,\alpha)v&=-\alpha\int_\Omega k(\cdot,y)v(y)\d\mu(y),&
	D_2\sF(0,\alpha)&=0,\\
	D_1D_2\sF(0,\alpha)v&=-\int_\Omega k(\cdot,y)v(y)\d\mu(y),&
	D_1^2\sF(0,\alpha)&=0,\\
	D_1^3\sF(0,\alpha)v^3&=6\alpha\int_\Omega k(\cdot,y)v(y)^3\d\mu(y).&&
\end{align*}
\begin{example}[cosine kernel]\label{exaarset}
	Working again with the degenerate kernel \eqref{kerfinrad} studied in \eref{exfold}, the right-hand side \eqref{exflip1} becomes
	$$
		\sF(u,\alpha)
		=
		\frac{\pi a \alpha}{2}\sum_{i=1}^2\int_{-\tfrac{L}{2}}^{\tfrac{L}{2}}e_i(y)u(y)\intoo{u(y)^2-1}\d y\,e_i
	$$
	and is even. For all $v,\bar v\in C[-\tfrac{L}{2},\tfrac{L}{2}]$, the derivatives read as
	\begin{align*}
		D_1\sF(u,\alpha)v&=\frac{\pi a \alpha}{2}\sum_{i=1}^2\int_{-\tfrac{L}{2}}^{\tfrac{L}{2}}e_i(y)(3u(y)^2-1)v(y)\d y\,e_i,\\
		D_1^2\sF(u,\alpha)v\bar v&=3\pi a\alpha\sum_{i=1}^2\int_{-\tfrac{L}{2}}^{\tfrac{L}{2}}e_i(y)u(y)v(y)\bar v(y)\d y\,e_i,\\
		D_1^3\sF(u,\alpha)v^3&=3\pi a\alpha\sum_{i=1}^2\int_{-\tfrac{L}{2}}^{\tfrac{L}{2}}e_i(y)v(y)^3\d y\,e_i,\\
		D_2\sF(u,\alpha)&=\frac{\pi a}{2}\sum_{i=1}^2\int_{-\tfrac{L}{2}}^{\tfrac{L}{2}}e_i(y)u(y)(u(y)^2-1)\d y\,e_i,\\
		D_2^2\sF(u,\alpha)&=0,\\
		D_1D_2\sF(u,\alpha)v&=\frac{\pi a}{2}\sum_{i=1}^2\int_{-\tfrac{L}{2}}^{\tfrac{L}{2}}e_i(y)(3u(y)^2-1)v(y)\d y\,e_i. 
	\end{align*}
	First, we claim a period doubling at $(0,\alpha_0^0)$ with $\alpha_0^0:=\frac{2}{\pi a\omega_{20}}>0$. Indeed, one shows $N(D_1\sF(0,\alpha_0^0)+I_{C(\Omega)})=\spann\set{e_1}$ and $1\notin\sigma(D_1\sF(0,\alpha_0^0))$. As per \sref{sec43}, we set $\theta:=2$, and verify $N(\Xi_2(\alpha_0^0)-I_{C(\Omega)})=N(\Xi_2(\alpha_0^0)'-I_{C(\Omega)})=\spann\set{e_1}$, giving rise to $2$-periodic solutions $\xi^\ast = (\xi^\ast_t)_{t\in\Z}$ of \eqref{var} resp.\ $\eta^\ast = (\eta^\ast_t)_{t\in\Z} \in \ell_2$ of $(V'_{\alpha_0^0})$, satisfying $\xi^\ast_0 = \eta^\ast_0 = e_1$, $\xi^\ast_1 = \eta^\ast_1 = -e_1$. Noting that $D_1^2\sF(0,\alpha_0^0)=0$ implies that the quantity $\bar\psi$ from \eqref{wbar} is zero, we can compute the bifurcation conditions
	\begin{align*}
		g_{11}
		=&
		\sum_{t=0}^1\int_{-\tfrac{L}{2}}^{\tfrac{L}{2}}\eta^\ast_{t+1}(x)[D_1D_2\sF(0,\alpha_0^0)\xi^\ast_t](x)\d x = \pi a\omega_{20}^2 > 0,\\
		%		-\frac{\pi a}{4\omega_{20}}\int_{-\tfrac{L}{2}}^{\tfrac{L}{2}}e_1(x)\intoo{[D_1D_2\sF(0,\alpha_0^0)\xi^\ast_1](x)-[D_1D_2\sF(0,\alpha_0^0)\xi^\ast_0](x)}\d x\\
		%&=
		%\frac{\pi a}{2}\int_{-\tfrac{L}{2}}^{\tfrac{L}{2}}e_1(x)^2\d x = ,\\ %<
		&\sum_{t=0}^1\int_{-\tfrac{L}{2}}^{\tfrac{L}{2}}\eta^\ast_{t+1}(x)
		[D_2^2\sF(0,\alpha_0^0)](x)\d x = 0, \\
		&
		\sum_{t=0}^1\int_{-\tfrac{L}{2}}^{\tfrac{L}{2}}\eta^\ast_{t+1}(x)
		[D_1^2\sF(0,\alpha_0^0)(\xi^\ast_t)^2](x)\d x = 0,\\
		\bar{g}
		=&
		\sum_{t=0}^1\int_{-\tfrac{L}{2}}^{\tfrac{L}{2}}\eta^\ast_{t+1}(x)
		[D_1^3\sF(0,\alpha_0^0)(\xi^\ast_t)^3](x)\d x + 0 %\\
		%&=
		%-\frac{12\omega_{40}}{\omega_{20}}\int_{-\tfrac{L}{2}}^{\tfrac{L}{2}}e_1(x)^2\d x 
		=
		-12\omega_{40} < 0.
	\end{align*}
	Whence, \pref{proppitch}(b) guarantees a supercritical pitchfork bifurcation of $2$-periodic solutions from the trivial branch at $\alpha_0^0$, that is, a flip bifurcation. 

	Second, we verify a bifurcation along this branch of $2$-periodic solutions. This branch, parametrized by $\svector{\gamma}{\alpha}:S\to B_\eps(0)\tm\R$, $\varepsilon>0$, defined in a neighborhood $S\subseteq\R$ of $0$, is locally unique. Moreover, \pref{proppitch} suggests the ansatz $\gamma(s)_0 = se_1$, $\alpha(s) = \alpha_0^0 - \tfrac{g_{30}}{3g_{11}}s^2 + \tilde{\alpha}(s) = \alpha_0^0(1 + \tfrac{\omega_{40}}{\omega_{20}}s^2) + \tilde{\alpha}(s)$, where $\tilde{\alpha}(0)=\dot{\tilde{\alpha}}(0)=\ddot{\tilde{\alpha}}(0)=0$. Using this, one obtains $\alpha(s) = \alpha_0^0(1 + \tfrac{\omega_{40}}{\omega_{20}}s^2) + \tfrac{\omega_{40}^2 s^4}{\omega_{20}^2(\omega_{20}-s^2\omega_{40})}$ and $\gamma(s)_1=-se_1$, which can be re-parametrized as $\phi_\pm(\alpha)_t = \pm(-1)^t\sqrt{\tfrac{a\pi\omega_{20}\alpha-2}{a\pi\omega_{40}\alpha}}e_1$ for $\alpha\geq\alpha_0^0$ near $\alpha_0^0$. Concerning ourselves with the "upper" branch $\phi_+$ (but suppressing the plus sign for readability) and retaining $\theta=2$, we observe that the linear equation
\begin{align*}
	0 & = 
	\begin{bmatrix}
		D_1\sF(\phi(\alpha)_1,\alpha)(A_1e_1+B_1e_2) \\
		D_1\sF(\phi(\alpha)_0,\alpha)(A_0e_1+B_0e_2)		
	\end{bmatrix}
	-
	\begin{bmatrix}
		A_0e_1+B_0e_2\\
		A_1e_1+B_1e_2
	\end{bmatrix}\\
	&=
	\begin{bmatrix}
		((\pi a \omega_{20} \alpha - 3)A_1 - A_0)e_1 + (\frac{3 \pi a \alpha \omega_{20} \omega_{22}-\pi a \alpha \omega_{02} \omega_{40} -6 \omega_{22}}{2 \omega_{40}}B_1 - B_0)e_2\\ 
		((\pi a \omega_{20} \alpha - 3)A_0 - A_1)e_1 + (\frac{3 \pi a \alpha \omega_{20} \omega_{22}-\pi a \alpha \omega_{02} \omega_{40} -6 \omega_{22}}{2 \omega_{40}}B_0 - B_1)e_2
	\end{bmatrix}
\end{align*}
	has the unique (up to multiples of $A_0$ and $A_1$) solution $B_0=B_1=0$ and $A_0=A_1$ for $\alpha=2\alpha_0^0=:\alpha_1^0$; we write $\phi^\ast:=\phi(\alpha_1^0)=\bigl((-1)^t\sqrt{\tfrac{\omega_{20}}{2\omega_{40}}}e_1\bigr)_{t\in\Z}$. By \pref{proplin}, we have $N(\Xi_2(\alpha_1^0)-I_{C(\Omega)}) = \spann\set{e_1}$, and derive the $2$-periodic solution $\xi^\ast$ of $(V_{\alpha_1^0})$ satisfying $\xi^\ast_0=\xi^\ast_1=e_1$. The simple nature of our system allows us to determine $N(\Xi_2(\alpha_1^0)'-I_{C(\Omega)}) = \spann\set{(3{\phi^\ast_1}^2-1)e_1}$ using a similar equation, now yielding a $2$-periodic solution $\eta^\ast$ of $(V'_{\alpha_1^0})$ with $\eta^\ast_0 = (3{\phi^\ast_1}^2-1)e_1 = (3{\phi^\ast_0}^2-1)e_1 = \eta^\ast_1$. 

	% For $\alpha=2\alpha_0^0=:\alpha_1^0$, one has $\phi_t(\alpha_1^0)=(-1)^t\sqrt{\tfrac{\omega_{20}}{2\omega_{40}}}e_1$ and derives the solution 
Before proceeding, we note that $D_2\sF(\phi^\ast_t,\alpha_1^0)\neq 0$ for all $t\in\Z$; hence, \tref{thmcross} fails to apply directly. As we know $\phi$ explicitly, this is not a significant hindrance; we pass over to the equation \eqref{deqp} of perturbed motion with
$$
	\psi_t
	:=
	\dot\phi(\alpha_1^0)_t
	= 
	(-1)^t\frac{\pi a \sqrt{\omega_{20}}^3}{8 \sqrt{2\omega_{40}}}e_1
	\fall t\in\Z.
$$
It remains to determine the unique solution $\bar{\psi}$ of \eqref{wbar}. We start by noting that there are infinitely many $\tilde\psi\in\ell_2$ solving all but the last line of \eqref{wbar}; it is not hard to demonstrate that $\tilde\psi_0 = 0$, $\tilde\psi_1 = 6\sqrt{\tfrac{2\omega_{40}}{\omega_{20}}}e_1$ represents such a solution. In order to satisfy also the last line of \eqref{wbar}, we note that
	$$
		\bar{\psi}_t
		:=
		\tilde\psi_t - \frac{\sum_{s=0}^1\int_{-\tfrac{L}{2}}^{\tfrac{L}{2}}\eta^\ast_{s}(x)\tilde\psi_s(x)\d x}{\sum_{s=0}^1\int_{-\tfrac{L}{2}}^{\tfrac{L}{2}}\eta^\ast_{s}(x)\xi^\ast_s(x)\d x}\,\xi_t^\ast
		=
		3(-1)^{t+1}\sqrt{\frac{2\omega_{40}}{\omega_{20}}}e_1\fall t\in\Z
	$$
	solves \eqref{wbar}, regardless of the choice of $\tilde\psi$. This yields the bifurcation conditions
	\begin{align*}
	g_{11}
	=&\sum_{t=0}^1\int_{-\tfrac{L}{2}}^{\tfrac{L}{2}}\eta^\ast_{t+1}(x)
	\intoo{
	[D_1D_2\sF(\phi_t^\ast,\alpha_1^0)\xi^\ast_t](x) + [D_1^2\sF(\phi_t^\ast,\alpha_1^0)\psi_t\xi^\ast_t](x)
	}\d x\\
	=&
	\,\pi a\omega_{20} > 0,\\
	&\sum_{t=0}^1\int_{-\tfrac{L}{2}}^{\tfrac{L}{2}}\eta^\ast_{t+1}(x)[D_1^2\sF(\phi_t^\ast,\alpha_1^0)(\xi^\ast_t)^2](x)\d x = 0,\\
	\bar{g}
	=&
	\sum_{t=0}^1\int_{-\tfrac{L}{2}}^{\tfrac{L}{2}}\eta^\ast_{t+1}(x)\intoo{[D_1^3\sF(\phi_t^\ast,\alpha_1^0)(\xi^\ast_t)^3](x)+3[D_1^2\sF(\phi_t^\ast,\alpha_1^0)\xi^\ast_t\bar\psi_t](x)}\d x\\
	=&
	-192\omega_{40} < 0. 
\end{align*}
	Thanks to \pref{proppitch}(b) this guarantees a supercritical pitchfork bifurcation into a $2$-periodic solution from $\phi^\ast$ at $\alpha=\alpha_1^0$, which is \emph{not} a period doubling. 
\end{example}
\begin{figure}
	\includegraphics[scale=0.6]{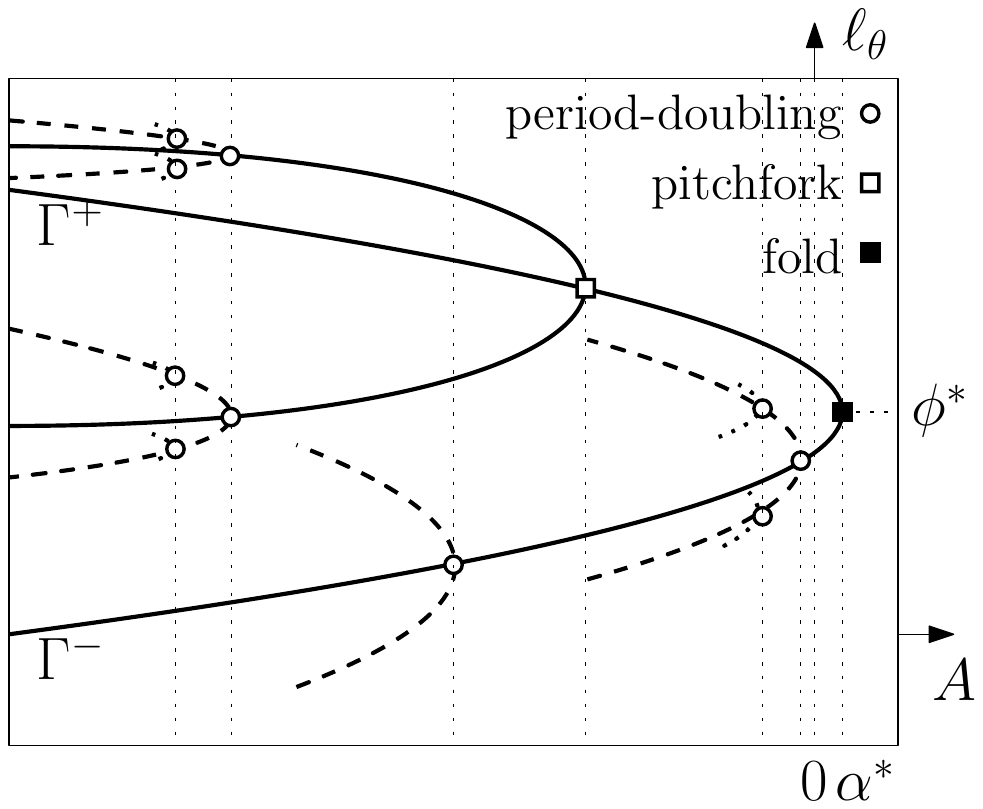}
	\includegraphics[scale=0.6]{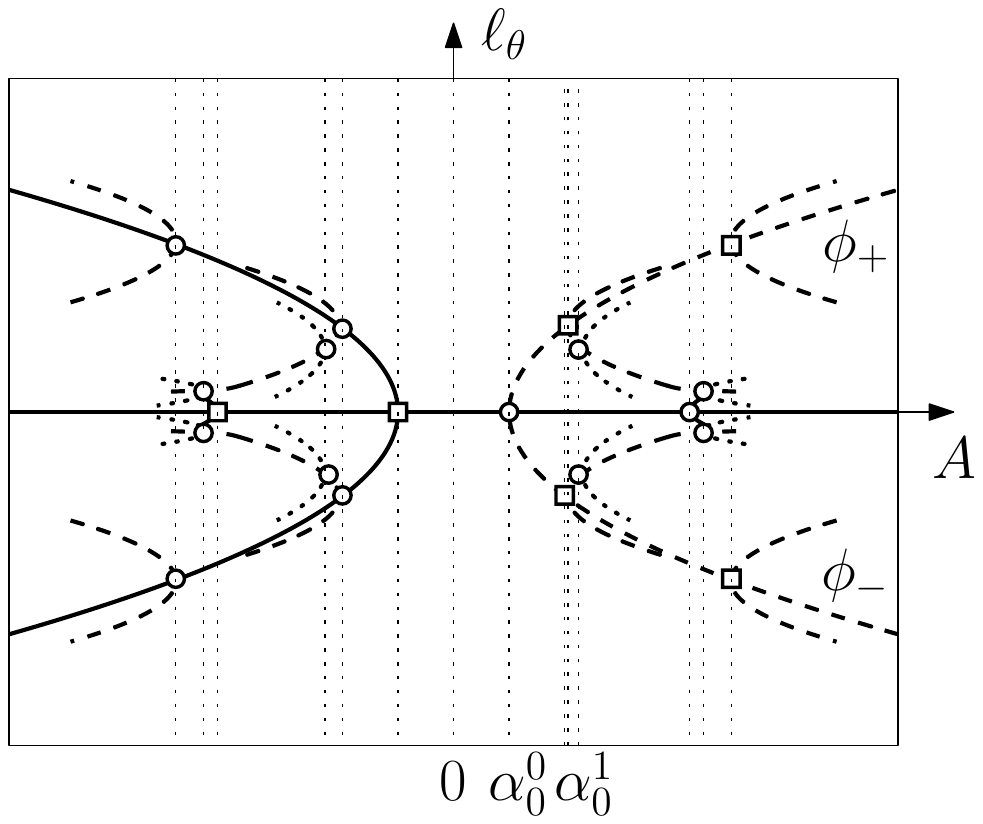}
	\caption{Schematic bifurcation diagrams for the cosine kernel \eqref{kerfinrad} with $aL\leq\tfrac{1}{2}$ illustrating branches of $\theta$-periodic solutions: 
	\eref{exfold} has a subcritical fold bifurcation of fixed points at $(\phi^\ast,\alpha^\ast)$ (left). 
	After a supercritical period doubling at $(0,\alpha_0^0)$, \eref{exaarset} has a supercritical pitchfork bifurcation of $2$-periodic solutions at $(\phi_\pm(\alpha_1^0),\alpha_1^0)$ (right). 
	Fixed point branches are solid $(\theta=1)$, branches of $2$-periodic solutions are dashed $(\theta=2)$ and $4$-periodic solutions are indicated by dotted lines $(\theta=4)$} 
	\label{figschem1}
\end{figure}
\subsection{Trivial branch analysis}
\label{sec52}
Models from ecology typically possess the zero solution. Hence, in order to determine critical parameter values it suffices to approximate eigenpairs of the linearization in $0$ numerically. This section features a flexible class of IDEs which exhibits a countable number of bifurcations from the trivial solution. The primary bifurcation is always transcritical and so is every second one. In between two transcritical bifurcations a pitchfork one occurs (at least in the autonomous case). This verifies an observation illustrated in \cite[Fig.~7]{kirk:lewis:97} for the first three branches. 

Concretely, we consider right-hand sides in \eqref{deq} of the form 
\begin{equation}
	\sF_t(u,\alpha):=G\intoo{\alpha\beta_t\int_\Omega f(\cdot,y,u(y))\d\mu(y)}
	\label{newrhs}
\end{equation}
with sufficiently smooth functions $f:\Omega\tm\Omega\tm U^1\to\R$, $G:U^2\to\R$ defined on open intervals $U^1, U^2\subseteq\R$ containing $0$ and satisfying the identities
\begin{align*}
	f(x,y,0)&\equiv 0\on\Omega\tm\Omega,&
	G(0)&=0,
\end{align*}
as well as the symmetry condition $D_3f(x,y,0)=D_3f(y,x,0)$ for all $x,y\in\Omega$. Moreover, $(\beta_t)_{t\in\Z}$ is assumed to be a $\theta$-periodic sequence in $(0,\infty)$ and $\alpha>0$ is the bifurcation parameter. Thus, \eqref{deq} has the trivial solution and we are in the set-up of \eqref{ury} with the functions
\begin{align*}
	G_t(x,z,\alpha)&=G(z),&
	f_t(x,y,z,\alpha)&=\alpha\beta_t f(x,y,z). 
\end{align*}

It is convenient to introduce the Fredholm integral operator $\sK\in L(\Omega)$, 
\begin{equation}
	\sK v:=G'(0)\int_\Omega D_3f(\cdot,y,0)v(y)\d\mu(y)\fall v\in C(\Omega),
	\label{defK}
\end{equation}
which is symmetric and therefore has the dual operator
\begin{equation}
	\sK'v=G'(0)\int_\Omega D_3f(y,\cdot,0)v(y)\d\mu(y)=\sK v\fall v\in C(\Omega).
	\label{defKs}
\end{equation}
Moreover, its eigenvalues $\lambda_i$, $i\in I$ from a countable index set $I=\set{0,1,\ldots}\subseteq\N_0$, are reals and might be ordered according to $\ldots<\lambda_2<\lambda_1<\lambda_0$, while $\xi^i\in C(\Omega)$, $i\in I$, denote the associated eigenfunctions. 
\subsubsection{Simplicity of eigenvalues}
A crucial assumption in our bifurcation analysis is the simplicity of critical eigenvalues. Here we provide a sufficient criterion guaranteeing that every eigenvalue of $\sK$ on a symmetric domain in $\R$ is simple, when $\mu$ is the Lebesgue measure on $\R$. 
\begin{prop}\label{prop_total-positivity}
	Equip $\Omega=[-\tfrac{L}{2},\tfrac{L}{2}]$ with the Lebesgue measure. If $G'(0)>0$ and
	\begin{align*}
		\det(D_3f(x_i,y_j,0))_{i,j=0}^n&\geq 0,&
		\det(D_3f(x_i,x_j,0))_{i,j=0}^n &> 0\fall n\in\N_0
	\end{align*}
	and all $-\tfrac{L}{2}<x_0<\ldots<x_n<\tfrac{L}{2}$, $\, -\tfrac{L}{2}<y_0<\ldots<y_n<\tfrac{L}{2}$ hold, then every eigenpair $(\lambda_i,\xi^i)$, $i\in I$, of $\sK$ satisfies: 
	\begin{enumerate}
		\item $\lambda_i$ is positive and simple, 
	
		\item $\xi^i\in C(\Omega)$ has exactly $i$ distinct zeros $x^i_0<\ldots<x^i_i$ and $i$ sign changes, 

		\item the zeros of $\xi^i$ and $\xi^{i+1}$ strictly interlace, that is, $x^{i+1}_0 < x^i_0 < x^{i+1}_1 < x^i_1 < \ldots < x^{i+1}_i < x^i_i < x^{i+1}_{i+1}$.
	\end{enumerate}
\end{prop}
\begin{proof}
	Apply \cite[Thm.~4.1]{pinkus:96} to the kernel $k(x,y):=G'(0)D_3f(x,y,0)$. 
\end{proof}
\begin{corollary}\label{cor_eig-alt-even-odd}
	If furthermore $D_3f(-x,-y,0)=D_3f(x,y,0)$ holds for arbitrary $x,y\in[-\tfrac{L}{2},\tfrac{L}{2}]$, then $\xi^i$ is even for even $i\in I$ and odd for odd $i\in I$.
\end{corollary}
\begin{proof}
	\cref{cor_eigv-odd-or-even} yields that every eigenfunction is odd or even. Fix $i\in I$ and assume first it is even. As odd functions cannot have an even number of zeros, it is immediate from \pref{prop_total-positivity} that $\xi^i$ is even. Assume next $i$ is odd, but that $\xi^i$ is even. As it has an odd number of zeros, necessarily $\xi^i(0) = 0$. Being even (and nowhere constant zero, as it has a finite number of zeros), it is either strictly positive or strictly negative near $0$, which is impossible, as this implies $\xi^i$ has at most $i-1$ sign changes, contradicting \pref{prop_total-positivity}. Accordingly, $\xi^i$ is odd. 
\end{proof}
\subsubsection{Bifurcations along the trivial branch}
The branchings of solutions to \eqref{deq} along zero are now determined by properties of \eqref{defK}, as we have derivatives 
\begin{align*}
	D_1\sF_t(0,\alpha)v&=\alpha\beta_t\sK v,&
	D_1D_2\sF_t(0,\alpha)v&=\beta_t\sK v,&
	D_2\sF_t(0,\alpha)=D_2^2\sF_t(0,\alpha)&=0
\end{align*}
for all $t\in\Z$ and $\alpha\in A$. We formulate a first \textbf{standing hypothesis}:
\begin{itemize}
	\item[$(Z_1)$] Suppose that $\sK\in L(C(\Omega))$ has a simple eigenpair $(\lambda,\xi_0^\ast)$ with $\lambda>0$ and set
	$$
		\alpha^\ast:=\frac{1}{\lambda\sqrt[\theta]{\beta_{\theta-1}\cdots\beta_0}}.
	$$
\end{itemize}
Let us clarify the bifurcation behavior of \eqref{deq} in $(0,\alpha^\ast)$. First, the resulting variational, resp.\ is dual difference equation
\begin{align*}
	v_{t+1}&=\alpha^\ast\beta_t\sK v_t,&
	v_t&=\alpha^\ast\beta_t\sK'v_{t+1}
	\stackrel{\eqref{defKs}}{=}
	\alpha^\ast\beta_t\sK v_{t+1}
\end{align*}
possess $\theta$-periodic solutions
\begin{align*}
	\xi_t^\ast&:=a_t\xi^\ast_0\in C(\Omega),&
	\eta_{t}^\ast&:=a_t^{-1}\xi^\ast_0\in C(\Omega),&
	a_t := (\alpha^\ast \lambda)^{t}\prod_{r=0}^{t-1}\beta_r > 0
\end{align*}
with $\eta^\ast_t\xi^\ast_t\equiv(\xi_0^\ast)^2$ on $\Z$. Thus, \eqref{deq} satisfies the bifurcation conditions $(B_1$--$B_3)$ at $(0,\alpha^\ast)$, and 
\begin{align*}
	g_{11} & := \sum_{t=0}^{\theta-1}\int_\Omega \eta^\ast_{t+1}(x)[D_1D_2\sF_t(0,\alpha^\ast)\xi^\ast_t](x)\d\mu(x) 
	= \sum_{t=0}^{\theta-1}\beta_t\int_\Omega \eta^\ast_{t+1}(x)[\sK\xi^\ast_t](x)\d\mu(x) \\
	& = \frac{1}{\alpha^\ast}\sum_{t=0}^{\theta-1}\int_\Omega \eta^\ast_{t+1}(x)\xi^\ast_{t+1}(x)\d\mu(x) 
	=
	\frac{\theta}{\alpha^\ast}\int_\Omega\xi^\ast_0(x)^2\d\mu(x) > 0.
\end{align*}
More can be said under an additional assumption tailor-made for various growth-dispersal and dispersal-growth right-hand sides, and their linear combinations.
\begin{itemize}
	\item[$(Z_2)$] There exist $c_2,d_2,c_3,d_3\in\R$ such that the representation
	\begin{equation}
		\begin{cases}
			D_1^2\sF_t(0,\alpha^\ast)v\bar v=c_2\alpha^\ast\beta_t\sK(v\bar v) + d_2\intoo{\alpha^\ast\beta_t}^2(\sK v)(\sK\bar v), \\
			D_1^3\sF_t(0,\alpha^\ast)v^3=c_3\alpha^\ast\beta_t\sK v^3 + d_3\intoo{\alpha^\ast\beta_t}^3(\sK v)^3
		\end{cases}
		\label{cdform}
	\end{equation}
	holds for all $t\in\Z$ and $v,\bar v\in C(\Omega)$. 
\end{itemize}
\begin{table}
	\begin{tabular}{r|c|cccc}
		growth function & $\hat g(z)$ & $c_2$ & $d_2$ & $c_3$ & $d_3$ \\
		\hline
		logistic & $z(1-z)$ & $-2$ & $-2$ & $0$ & $0$ \\
		Hassell & $\tfrac{z}{(1+z)^c}$ & $-2c$ & $-2c$ & $3(1+c)c$& $3(1+c)c$\\
		Ricker & $ze^{-z}$ & $-2$ & $-2$ & $3$ & $3$\\
	\end{tabular}
	\caption{Coefficients in commonly used growth functions $\hat g$, $c>0$}
	\label{tabgrowth}
\end{table}
\begin{remark}\label{remgrowth}
	One frequently encounters the situation $f(x,y,z)=k_0(x,y)g(z)$ with a continuous function $k_0:\Omega\tm\Omega\to\R$ and a $C^3$-growth function $g:U^1\to\R$ satisfying $g(0)=0$. In case $G''(0)g''(0)=0$ the coefficients in $(Z_2)$ explicitly compute as
	\begin{align*}
		c_2&=\tfrac{g''(0)}{g'(0)},&
		d_2&=\tfrac{G''(0)}{G'(0)^2},&
		c_3&=\tfrac{g'''(0)}{g'(0)},&
		d_3&=\tfrac{G'''(0)}{G'(0)^3}.
	\end{align*}
	This applies to both Hammerstein equations ($G(z)=z$, $g=\hat g$) and dispersal-growth equations ($G=\hat g$, $g(z)=z$) with the growth functions $\hat g$ from e.g.\ Tab.~\ref{tabgrowth}.
\end{remark}

Let us first provide criteria for transcritical bifurcations in \eqref{newrhs}. The assumption \eqref{cdform} leads to the following bifurcation indicators
\begin{align}
	g_{20} 
	= & 
	\sum_{t=0}^{\theta-1}\int_\Omega \eta^\ast_{t+1}(x) \intoo{c_2\alpha^\ast\beta_t[\sK(\xi^\ast_t)^2](x) + d_2\intoo{\alpha^\ast\beta_t[\sK\xi^\ast_t](x)}^2}\d\mu(x) 
	\notag\\
	\stackrel{\eqref{defK}}{=} & \sum_{t=0}^{\theta-1}c_2\int_\Omega \xi^\ast_t(y)^2 \int_\Omega \alpha^\ast\beta_tk(x,y)\eta^\ast_{t+1}(x)\d\mu(x)\d\mu(y)
	\notag\\
	& + \sum_{t=0}^{\theta-1}d_2\int_\Omega \eta^\ast_{t+1}(x)\xi^\ast_{t+1}(x)^2 \d\mu(x)
	\notag\\
	= & \sum_{t=0}^{\theta-1}\int_\Omega c_2\eta^\ast_t(x)\xi^\ast_t(x)^2 + d_2\eta^\ast_{t}(x)\xi^\ast_{t}(x)^2\d\mu(x)
	\notag\\
	= & \intoo{c_2+d_2}\intoo{\sum_{t=0}^{\theta-1}a_t}\int_\Omega\xi^\ast_0(x)^3\d\mu(x),
	\label{no1}\\
	g_{30} 
	= & 
	\sum_{t=0}^{\theta-1}\int_\Omega \eta^\ast_{t+1}(x) \intoo{c_3\alpha^\ast\beta_t[\sK(\xi^\ast_t)^3](x) + d_3\intoo{\alpha^\ast\beta_t[\sK\xi^\ast_t](x)}^3}\d\mu(x)
	\notag\\
	\stackrel{\eqref{defK}}{=} & \sum_{t=0}^{\theta-1}c_3\int_\Omega \xi^\ast_t(y)^3 \int_\Omega \alpha^\ast\beta_tk(x,y)\eta^\ast_{t+1}(x)\d\mu(x)\d\mu(y)
	\notag\\
	&+ \sum_{t=0}^{\theta-1}d_3\int_\Omega \eta^\ast_{t+1}(x)\xi^\ast_{t+1}(x)^3 \d\mu(x)
	\notag\\
	= & \sum_{t=0}^{\theta-1}\int_\Omega c_3\eta^\ast_t(x)\xi^\ast_t(x)^3 + d_3\eta^\ast_{t}(x)\xi^\ast_{t}(x)^3\d\mu(y)
	\notag\\
	= & \intoo{c_3+d_3}\intoo{\sum_{t=0}^{\theta-1}a_t^2}\int_\Omega\xi^\ast_0(x)^4\d\mu(x),
	\label{no2}
\end{align}
from which one immediately has
\begin{equation}
	g_{20}=0
	\quad\Leftrightarrow\quad
	c_2+d_2=0\text{ or }\int_\Omega\xi_0^\ast(x)^3\d\mu(x)=0.
	\label{degcase}
\end{equation}
\begin{prop}[transcritical bifurcation]\label{trivtrans}
	Suppose $(Z_1$--$Z_2)$ hold and $c_2+d_2\neq 0$. If $\int_\Omega\xi^\ast_0(x)^3\d\mu(x)\neq 0$, then there is a transcritical bifurcation at $(0,\alpha^\ast)$. 
\end{prop}
\begin{proof}
	The assumptions of \pref{proptrans} are satisfied. 
\end{proof}

For many popular IDEs, one indeed observes that a stability loss from the trivial branch occurs via a transcritical bifurcation; this is indeed a special case of the above. Denoting the \emph{spectral radius} of $\sK$ by $r(\sK)>0$, we obtain: 
\begin{prop}[primary bifurcation]\label{primary_positive}
	Suppose that $(Z_2)$ holds with $c_2+d_2\neq 0$ and
	$
		\alpha^\ast=\frac{1}{r(\sK)\sqrt[\theta]{\beta_{\theta-1}\cdots\beta_0}}. 
	$
	If 
	$$
		G'(0)D_3f(x,y,0)>0\fall x,y\in\Omega,
	$$
	then the trivial solution of \eqref{deq} is exponentially stable for $0<\alpha<\alpha^\ast$ and unstable for $\alpha>\alpha^\ast$. In particular, a transcritical bifurcation into a nontrivial branch $\phi_0^0$ of $\theta$-periodic solutions takes place at $(0,\alpha^\ast)$, where, near $\alpha^\ast$, $\phi_0^0$ is unstable for $\alpha<\alpha^\ast$ and exponentially stable for $\alpha>\alpha^\ast$.
\end{prop}
\begin{remark}[basic reproduction number]
	The critical value $\alpha^\ast$ in \pref{primary_positive} is a threshold parameter. For models studied in Sects.~\ref{sec523} and \ref{sec524} it distinguishes between extinction $(0<\alpha\leq\alpha^\ast)$ and persistence $(\alpha^\ast<\alpha)$ of a population, since the primary branch $\phi_0^0$ consists of globally attractive $\theta$-periodic solutions (w.r.t.\ solutions having positive values). The reciprocal of $\alpha^\ast$ might be interpreted as extension of the \emph{basic population turnover number} introduced in \cite{jin:thieme:16,thieme:20} to periodic problems. Since the IDEs with right-hand side \eqref{ury} studied here address populations having non-overlapping generations, this basic population turnover number coincides with the notorious \emph{basic reproduction number} $R_0$ (cf.~\cite{alzoubi:10,bacaer:09,bacaer:dats:12,cushing:ackleh,cushing:henson} in ecological and \cite[p.~353]{brauer:castillo:01} in epidemic models). Nevertheless, one can also formulate IDEs describing populations having overlapping generations, which involve e.g.\ one integral for newborns and one integral for the rest. In such a model, the basic turnover number does not coincide with the basic reproduction number and we refer to \cite[Rem.~3.14]{thieme:20} for the relation between the two. 
\end{remark}
\begin{proof}
	Let $C_+(\Omega)$ denote the solid cone of nonnegative functions in $C(\Omega)$. For every nonzero $v\in C_+(\Omega)$ there exists a $x_0\in\Omega$ and a $c>0$ such that $v(x)>c$ holds in a neighborhood $U\subseteq\Omega$ of $x_0$. Moreover, $\gamma:=G'(0)\min_{x,y\in\Omega}D_3f(x,y,0)>0$ due to the compactness of $\Omega$ and we arrive at
	$$
		[\sK v](x)
		\stackrel{\eqref{defK}}{=}
		G'(0)\int_\Omega D_3f(x,y,0)v(y)\d\mu(y)\geq c\gamma\mu(\Omega)>0\fall x\in\Omega; 
	$$
	thus, $\sK v$ is an interior point of $C_+(\Omega)$. This shows that $\sK$ is strongly positive and the Krein-Rutman theorem \cite[p.~228, Thm.~19.3]{deimling:85} applies. Hence, $r(\sK)>0$ is a simple eigenvalue of $\sK$, $|\lambda|<r(\sK)$ for all $\lambda\in\sigma(\sK)\setminus\set{r(\sK)}$, with an eigenfunction $\xi^\ast_0\in C(\Omega)$ such that $\inf_{x\in\Omega}\xi_0^\ast(x)>0$. Inserting this into the above yields $g_{20}\neq 0$, implying a transcritical bifurcation. In combination with $g_{11}>0$ as computed earlier and applying \cref{trbif}, \pref{proptrans} yields the claimed stability properties. 
\end{proof}

To describe the degenerate case $\int_\Omega\xi_0^\ast(x)^3\d\mu(x)=0$ in \eqref{degcase}, and to establish a pitchfork bifurcation, is not as straightforward and requires further preparations. For the solution $\bar\psi$ of \eqref{wbar}, a manipulation as above shows that one must solve
$$
	\begin{cases}
		\bar\psi_{t+1} = \alpha^\ast\beta_t\sK\intoo{\bar\psi_t + c_2(\xi_t^\ast)^2} + d_2(\xi_{t+1}^\ast)^2\fall 0\leq t<\theta-2, \\
		\bar\psi_0 = \alpha^\ast\beta_{\theta-1}\sK\intoo{\bar\psi_{\theta-1} + c_2(\xi^\ast_{\theta-1})^2} + d_2(\xi_0^\ast)^2, \\
		0 = \sum_{t=0}^{\theta-1}\int_\Omega \eta^\ast_{t}(x)\bar\psi_t(x)\d\mu(x).
	\end{cases}
$$
If $D_1^2\sF_t(0,\alpha^\ast)\equiv 0$ on $\Z$ holds, then $\bar\psi = 0$, and the condition $c_3+d_3\neq 0$ becomes sufficient for a pitchfork bifurcation at $(0,\alpha^\ast)$. Without this --- oppressively strong --- assumption, the following simplified scheme for computing $\bar\psi$ is advisable.
\begin{prop}\label{barwprop}
	Suppose $(Z_1$--$Z_2)$ hold with $c_2+d_2\neq 0$, and that either $\theta$ is odd or $-\lambda\not\in\sigma(\sK)$. If $g_{20}=0$, then the Fredholm integral equation of the second kind
	\begin{equation}
		\bar w = \lambda^{-\theta}\sK^\theta \bar w + (\xi_0^\ast)^2
		\label{fredeq}
	\end{equation}
	possesses a unique solution $\bar w\in R(\sK-\lambda I_{C(\Omega)})$. Moreover, one has
	\begin{align*}
		\bar g 
		:=
		g_{30} &+ 3\sum_{t=0}^{\theta-1} \int_\Omega \eta^\ast_{t+1}(x)[D_1^2\sF_t(0,\alpha^\ast)\xi^\ast_t\bar\psi_t](x)\d\mu(x) \\
		=
		g_{30} &+ 3\sum_{t=0}^{\theta-1} \int_\Omega \xi^\ast_0(x)^2\Bigl(\Bigr.(c_2+d_2)^2\lambda^{-t}\intoo{\sum_{r=0}^{t-1}a_{\theta+r-t}a_r + \sum_{r=t}^{\theta-1}a_ra_{r-t}}[\sK^t\bar w](x) \\
		& - (c_2+d_2)^2\xi^\ast_t(x)^2 + c_2d_2\intoo{\xi^\ast_{t+1}(x)^2 - \alpha^\ast\beta_t[\sK(\xi^\ast_t)^2)](x)}\Bigl.\Bigr) \d\mu(x). 
	\end{align*}
\end{prop}
\begin{proof}
	Let us first establish existence and uniqueness of the solution $\bar w$. Start by noting $\lambda^{-\theta}\sK^\theta - I_{C(\Omega)} = \lambda^{-\theta}(\sK - \lambda I_{C(\Omega)})\prod_{t=1}^{\theta-1}\intoo{\sK - \lambda e^{2\iota\pi t/\theta}I_{C(\Omega)}}$; as $\sigma(\sK)\setminus\set{0}$ consists only of real eigenvalues, $\prod_{t=1}^{\theta-1}\intoo{\sK - \lambda e^{2\iota\pi t/\theta}I_{C(\Omega)}}\in GL(C(\Omega))$ results from our assumptions. Consider for $\bar w\in R(\sK-\lambda I_{C(\Omega)})$ the equivalent system
$$
	(\xi_0^\ast)^2 = (\sK - \lambda I_{C(\Omega)})v, \qquad v = \lambda^{-\theta}\left(\prod_{t=1}^{\theta-1}\intoo{\sK - \lambda e^{\tfrac{2\pi t}{\theta}\iota}I_{C(\Omega)}}\right)\bar w.
$$
As the assumption $g_{20} = 0$ yields $\int_\Omega\eta_0^\ast(x)\xi_0^\ast(x)^2\d\mu(x) = 0$, one has the inclusion $(\xi_0^\ast)^2\in R(\sK-\lambda I_{C(\Omega)})$. This implies there is some unique $\bar v\in R(\sK-\lambda I_{C(\Omega)})$ such that $v = \bar v + \rho\xi_0^\ast$ solves the first part of the equation for all $\rho\in\R$. Note also that as $\sK\xi_0^\ast = \lambda\xi_0^\ast$ and $\lambda^{-\theta}\prod_{t=1}^{\theta-1}\intoo{\sK - \lambda e^{2\iota\pi t/\theta}I_{C(\Omega)}} = \sum_{t=0}^{\theta-1}\lambda^{-t-1}\sK^t$ holds, one obtains $\lambda^{-\theta}\intoo{\prod_{t=1}^{\theta-1}\intoo{\sK - \lambda e^{2\iota\pi t/\theta}I_{C(\Omega)}}}\rho\tfrac{\lambda}{\theta}\xi_0^\ast = \rho\xi_0^\ast$, which by invertibility guarantees 
	$$
		\lambda^{\theta}\Bigl(\prod_{t=1}^{\theta-1}(\sK - \lambda e^{2\iota\pi t/\theta}I_{C(\Omega)})\Bigr)^{-1}\rho\xi_0^\ast
		=
		\rho\tfrac{\lambda}{\theta}\xi_0^\ast\in N(\sK-\lambda I_{C(\Omega)}).
	$$
	Thus, the second equation has $\bar w = \lambda^{\theta}(\prod_{t=1}^{\theta-1}(\sK - \lambda e^{2\iota\pi t/\theta}I_{C(\Omega)}))^{-1}\bar v + \rho\tfrac{\lambda}{\theta}\xi_0^\ast$ as solutions, which all share the same unique projection into $R(\sK-\lambda I_{C(\Omega)})$, independent of $\rho$. Due to
	$$
	(\sK - \lambda I_{C(\Omega)})\prod_{t=1}^{\theta-1}(\sK - \lambda e^{\tfrac{2\pi t}{\theta}\iota}I_{C(\Omega)})
	=
	\intoo{\prod_{t=1}^{\theta-1}(\sK - \lambda e^{\tfrac{2\pi t}{\theta}\iota}I_{C(\Omega)})}(\sK - \lambda I_{C(\Omega)}),
	$$
	this projection is also a solution to the original equation, as required. Next, consider the system
	\begin{align*}
		w_0 &= \alpha\beta_{\theta-1}\sK w_{\theta-1} + (\xi_0^\ast)^2,&
		w_t &= \alpha\beta_{t-1}\sK w_{t-1} + (\xi_t^\ast)^2\fall 1\leq t<\theta
	\end{align*}
	of second kind integral equations. Using the variation of constants formula, 
	$$
		w_t := a_t\lambda^{-t}\intoo{\sum_{r=0}^{\theta-t-1}a_{\theta-r}\lambda^{-r}\sK^{t+r}\bar w + \sum_{r=\theta-t}^{\theta-1}a_{\theta-r}\lambda^{\theta-r}\sK^{t+r-\theta}\bar w}
		\fall 0\leq t< \theta
	$$
	provides a solution. Now define $\bar\psi_t := (c_2+d_2)w_t - c_2(\xi_t^\ast)^2$, $0\leq t < \theta$; it is immediate that it satisfies \eqref{wbar}. Finally, we compute
\begin{align*}
	& \sum_{t=0}^{\theta-1} \int_\Omega \eta^\ast_{t+1}(x)[D_1^2\sF_t(0,\alpha^\ast)\xi^\ast_t\bar\psi_t](x)\d\mu(x) \\
	\stackrel{\eqref{cdform}}{=} & 
	\sum_{t=0}^{\theta-1} \int_\Omega \eta^\ast_{t+1}(x)\intoo{c_2\alpha^\ast\beta_t[\sK\intoo{\xi^\ast_t\bar\psi_t}](x) + d_2(\alpha^\ast\beta_t)^2[\sK\xi^\ast_t](x)[\sK\bar\psi_t](x)}\d\mu(x) \\
	= & \sum_{t=0}^{\theta-1} \int_\Omega c_2\alpha^\ast\beta_t[\sK\eta^\ast_{t+1}](x)\xi^\ast_t(x)\intoo{(c_2+d_2)w_t(x) - c_2(\xi_t^\ast)^2(x)} \\
		& \quad+ \, d_2\eta^\ast_{t+1}(x)\xi^\ast_{t+1}(x)\alpha^\ast\beta_t[\sK\intoo{(c_2+d_2)w_t - c_2(\xi_t^\ast)^2}](x)\d\mu(x) \\
	= & \sum_{t=0}^{\theta-1} \int_\Omega \xi^\ast_0(x)^2\bigl((c_2+d_2)^2(w_t(x) - \xi^\ast_t(x)^2) \\
	&
	\quad+ c_2d_2\bigl(\xi^\ast_{t}(x) - \alpha^\ast\beta_t[\sK(\xi_t^\ast)^2](x)\bigr)\bigr) \d\mu(x) \\
	= & \sum_{t=0}^{\theta-1} \int_\Omega \xi^\ast_0(x)^2\Bigl(\Bigr.(c_2+d_2)^2\lambda^{-t}\intoo{\sum_{r=0}^{t-1}a_{\theta+r-t}a_r + \sum_{r=t}^{\theta-1}a_ra_{r-t}}[\sK^t\bar w](x) \\
		& \quad- (c_2+d_2)^2\xi^\ast_t(x)^2 + c_2d_2\intoo{\xi^\ast_{t+1}(x)^2 - \alpha^\ast\beta_t[\sK(\xi_t^\ast)^2](x)}\Bigl.\Bigr) \d\mu(x),
\end{align*}
where the final line follows from a tedious-yet-elementary index manipulation that is left to the interested reader.
\end{proof}
\begin{remark}\label{remtrans}
	The above computations also apply to right-hand sides 
	$$
		\sF_t(u,\alpha)=\alpha\beta_tG\intoo{\int_\Omega f(\cdot,y,u(y))\d\mu(y)}
	$$
	instead of \eqref{newrhs} with the only exception that the expressions $c_2+d_2$, $c_3+d_3$ in \eqref{no1}--\eqref{degcase} have to be replaced by $c_2+\lambda d_2$, $c_3+\lambda^2d_3$. 
\end{remark}

Let us next compare bifurcations from the trivial branch for such IDEs having the same growth function and the same kernel, first where growth precedes dispersal and second the other way around. Restricted to the cone of nonnegative continuous functions $C_+(\Omega)$, both classes have very simple dynamics. After the primary and transcritical bifurcation, all nontrivial solutions converge to a nonzero periodic solution in $C_+(\Omega)$ \cite[Thm.~5.1]{hardin:takac:webb:88}. Along this branch of globally attractive and positive solutions no further bifurcations occur (cf.\ Figs.~\ref{figlaplace1} and \ref{figlaplace1s}). From the trivial branch, however, countably many branches bifurcate, but all of them consist of biologically meaningless functions with negative values (cf.~Figs.~\ref{figlaplace2}, \ref{figlaplace3} and \ref{figlaplace2s}, \ref{figlaplace3s}).
\subsubsection{Growth-dispersal Beverton-Holt equation}
\label{sec523}
Consider a $\theta$-periodic scalar IDE \eqref{deq} with Hammerstein right-hand side
\begin{equation}
	\sF_t(u,\alpha):=\beta_t\int_\Omega k(\cdot,y)\frac{\alpha u(y)}{1+u(y)}\d\mu(y),
	\label{rhsbh}
\end{equation}
defined on the open and convex domain $U_t\equiv\set{u\in C(\Omega):\,\inf_{x\in\Omega}u(x)>-1}$. It is of the form \eqref{newrhs} with $G(z)=z$, $f(x,y,z)=k(x,y)\frac{z}{1+z}$ and thus
\begin{align*}
	G'(z)&=1,&
	D_3f(x,y,z)&=\frac{k(x,y)}{(1+z)^2}.
\end{align*}
Given $u\in U$, $\alpha>0$, we obtain the partial derivatives
\begin{align*}
	D_1\sF_t(u,\alpha)v&=\beta_t\int_\Omega k(\cdot,y)\frac{\alpha v(y)}{(1+u(y))^2}\d\mu(y),&&\\
	D_1D_2\sF_t(u,\alpha)v&=\beta_t\int_\Omega k(\cdot,y)\frac{v(y)}{(1+u(y))^2}\d\mu(y),&	D_2^2\sF_t(u,\alpha)&=0,\\
	D_1^2\sF_t(u,\alpha)v^2&=-2\beta_t\int_\Omega k(\cdot,y)\frac{\alpha v(y)^2}{(1+u(y))^3}\d\mu(y),&&\\
	D_1^3\sF_t(u,\alpha)v^3&=6\beta_t\int_\Omega k(\cdot,y)\frac{\alpha v(y)^3}{(1+u(y))^4}\d\mu(y).&&
\end{align*}
Because \eqref{deq} possesses the trivial solution, they become
\begin{align*}
	D_1\sF_t(0,\alpha)v&=\alpha\beta_t\sK v,&
	D_1^2\sF_t(0,\alpha)v^2&=-2\alpha\beta_t\sK v^2,\\
	D_1D_2\sF_t(0,\alpha)v&=\beta_t\sK v,&
	D_1^3\sF_t(0,\alpha)v^3&=6\alpha\beta_t\sK v^3
\end{align*}
and $D_2^2\sF_t(0,\alpha)=0$ for all $t\in\Z$, $v\in C(\Omega)$. 

If $(\lambda_i,\xi^i)$, $i\in I$, is a simple eigenpair of $\sK$ with $\int_\Omega\xi^i(x)^2\d\mu(x)=1$, then $(Z_1)$ holds with the critical parameter $\alpha^\ast$ given by
$$
	\alpha_i^0:=\frac{1}{\lambda_i\sqrt[\theta]{\beta_{\theta-1}\cdots\beta_0}}
$$
and so does the representation \eqref{cdform} with $c_2 = -2$, $c_3 = 6$, $d_2=d_3=0$ (for this, see \rref{remgrowth}). We have
\begin{align}
	g_{11}(i)
	=&
	\frac{\theta}{\alpha_i^0}\int_\Omega\xi^i(x)^2\d\mu(x) > 0,&
%	\label{condtrans}\\
	g_{20}(i)
	=&
	-2\intoo{\sum_{t=0}^{\theta-1}(\alpha_i^0 \lambda)^{t}\prod_{r=0}^t\beta_r}\int_\Omega\xi^i(x)^3\d\mu(x),
	%-2\int_\Omega\xi^i(y)^3\d\mu(y)\sum_{t=0}^{\theta-1}(\alpha_i^0 \lambda)^{t}\prod_{r=0}^t\beta_r,
	%-\frac{2}{\lambda}
	%\int_\Omega \xi^i(x)\int_\Omega k(x,y)\xi^i(y)^2\d\mu(y)\d\mu(x)\
	%\sum_{t=0}^{\theta-1}(\alpha\lambda)^{t}\prod_{r=0}^{t-1}\beta_r,
	\notag\\
	&&
	g_{30}(i)
	=&
	\, 6\intoo{\sum_{t=0}^{\theta-1}(\alpha_i^0 \lambda)^{2t}\prod_{r=0}^t\beta_r^2}\int_\Omega\xi^i(x)^4\d\mu(x)
	%\frac{6}{\lambda}
	%\int_\Omega \xi^i(x)\int_\Omega k(x,y)\xi^i(y)^3\d\mu(y)\d\mu(x)\
	%\sum_{t=0}^{\theta-1}\intoo{(\alpha\lambda)^{t}\prod_{r=0}^{t-1}\beta_r}^2
%	\fall i\in I,
	\notag
\end{align}
for all $i\in I$ and $\bar g(i)$ can be calculated using the above and \pref{barwprop}.

A further analysis needs an explicit knowledge of the eigenpairs of $\sK$. For general kernels, this requires numerical methods from App.~\ref{appB}. An exception is the convenient, yet realistic Laplace kernel (cf.~\cite{lutscher:petrovskii:08,reimer:bonsall:maini:16,kirk:lewis:97}), whose eigenvalues are determined by a transcendental equation.
\begin{figure}
	\begin{minipage}{20mm}
		\begin{tabular}{@{}rr@{}}
			\hline
			$i$ & $\alpha_i^0$\\
			\hline
			$0$ & $1.74$\\
			$1$ & $5.12$\\
			$2$ & $12.73$\\
			$3$ & $25.13$\\
			$4$ & $42.42$\\
			\hline
		\end{tabular}
	\end{minipage}
	\begin{minipage}{105mm}
		\includegraphics[width=52mm]{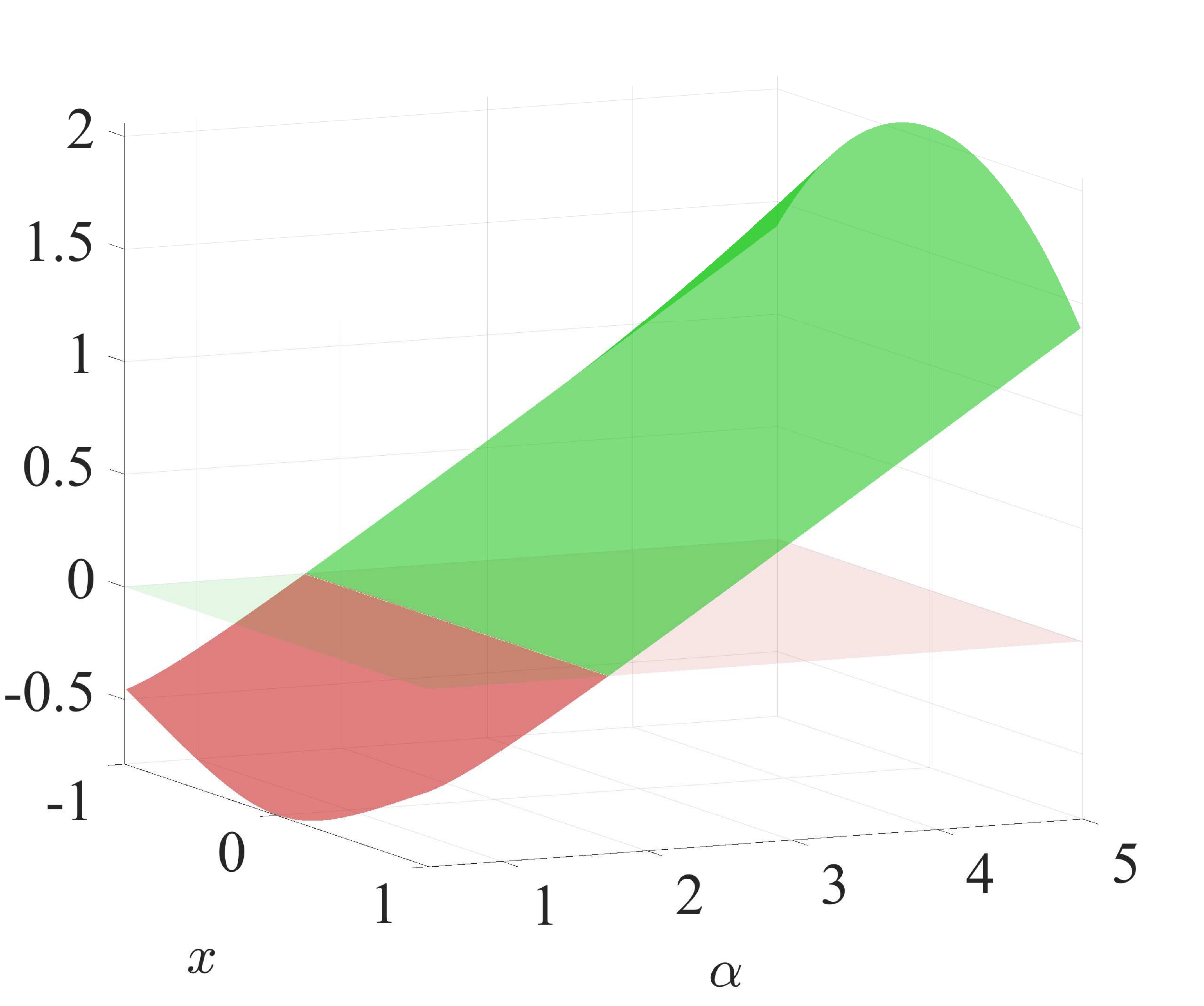}
		\includegraphics[width=52mm]{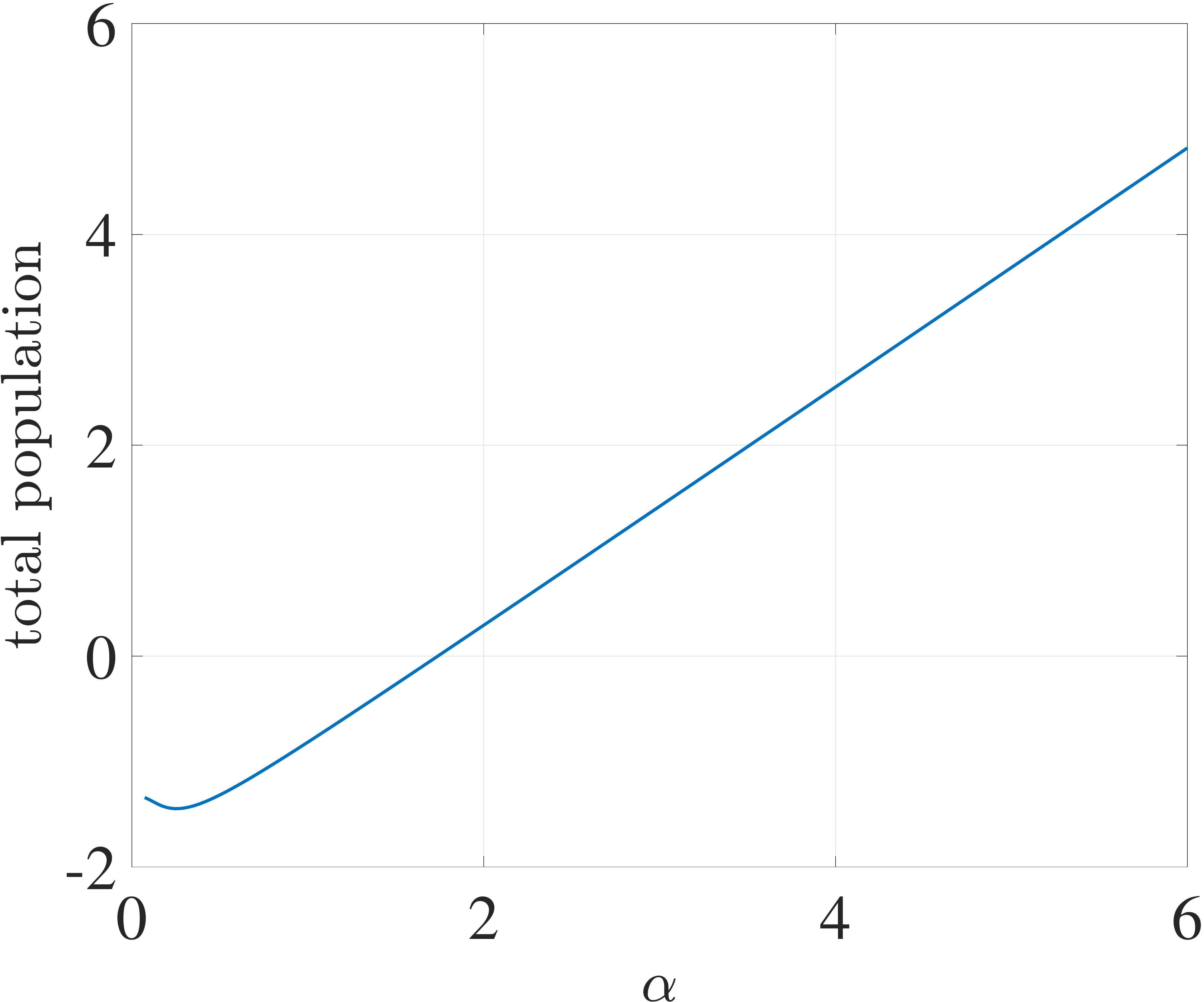}
	\end{minipage}
	\caption{First critical values $\alpha_i^0$ for bifurcations from the trivial branch (left). 
	Nontrivial branch $\phi_0^0$ of the primary transcritical bifurcation at $\alpha_0^0\approx 1.74$ for the Beverton-Holt growth-dispersal IDE with right-hand side \eqref{rhsbh}, Laplace kernel \eqref{kerlap} and $a=1$, $L=2$ (center). 
	Total population over $\alpha\in[0,6]$ along $\phi_0^0$ (right)}
	\label{figlaplace1}
\end{figure}
\begin{figure}
	\includegraphics[width=52mm]{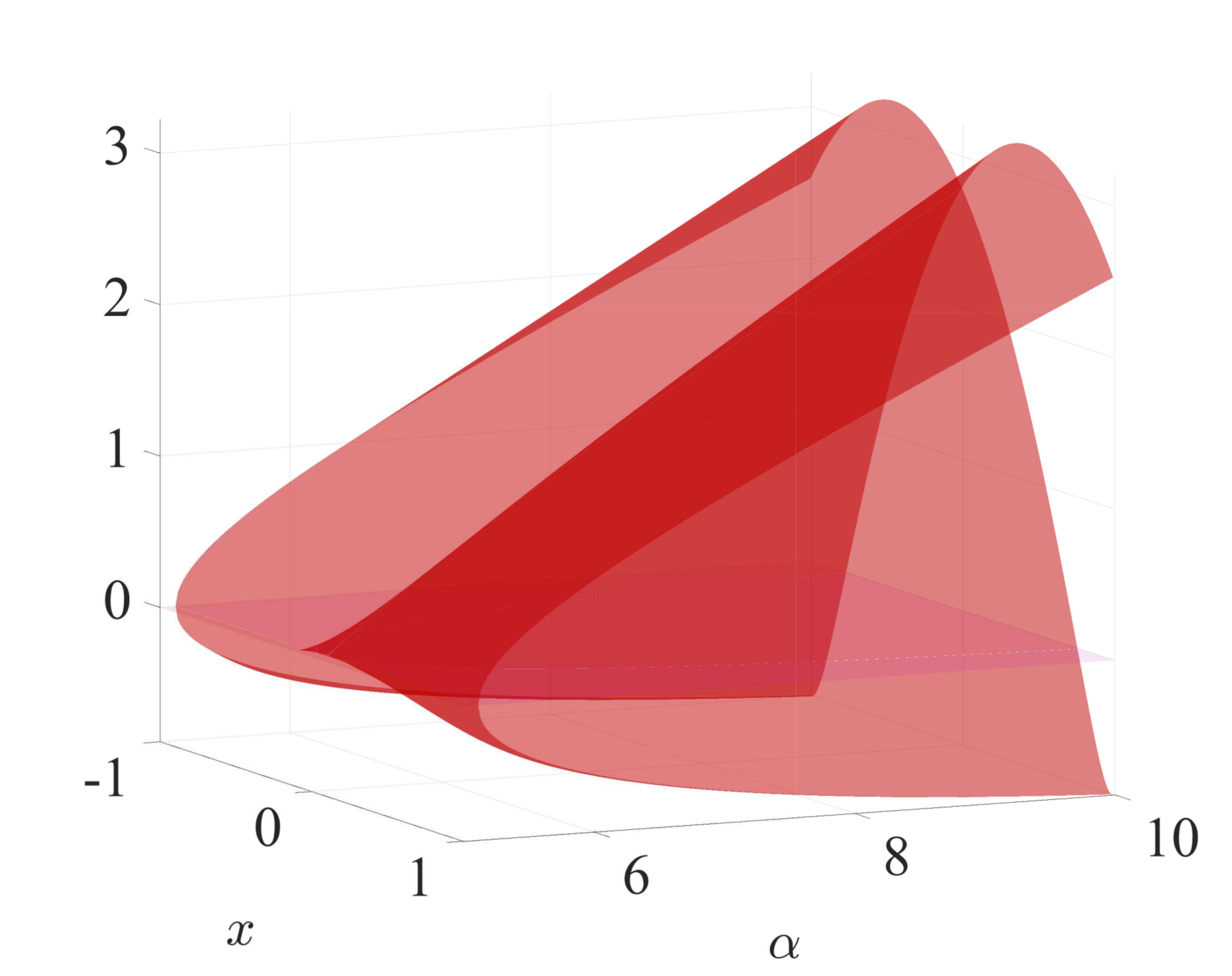}
	\includegraphics[width=52mm]{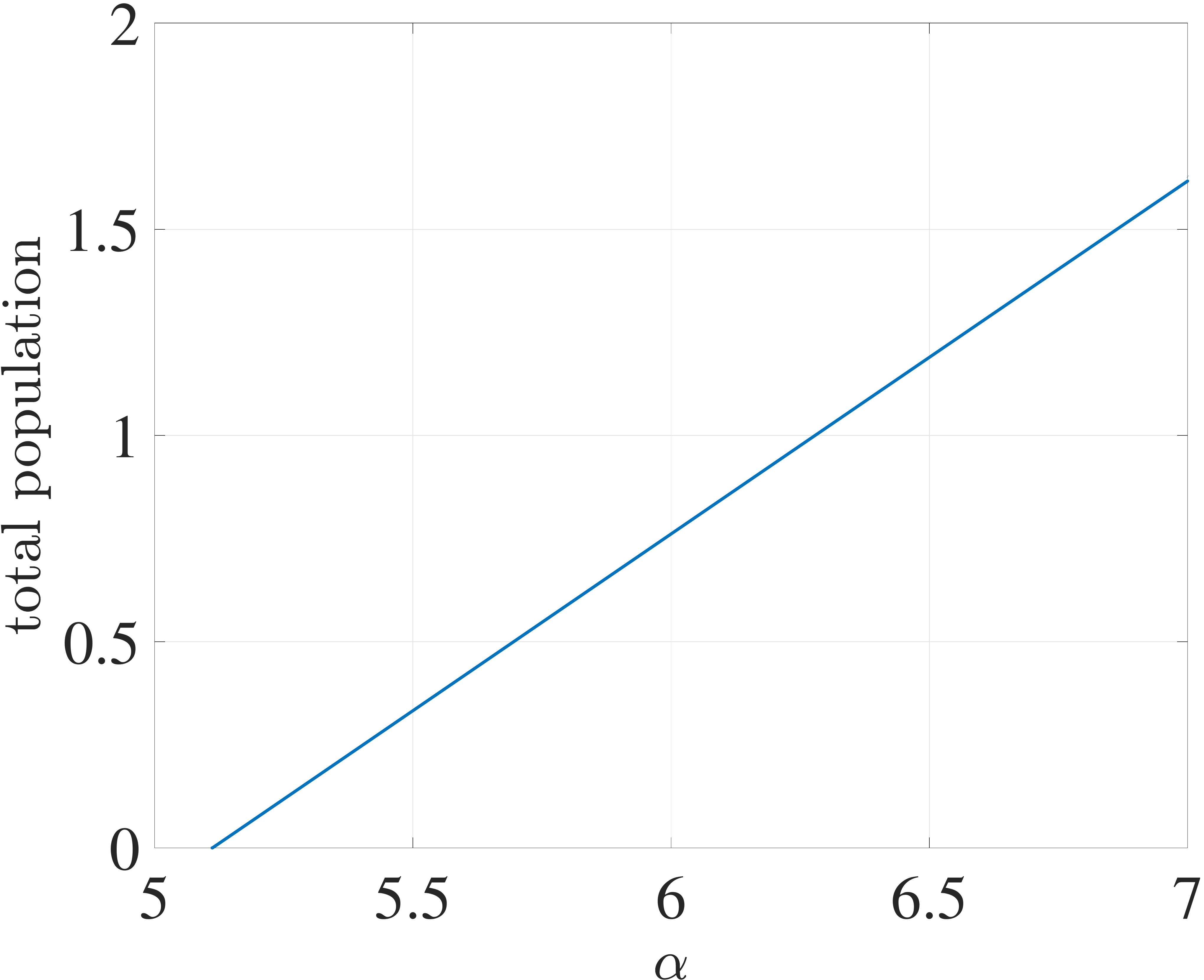}
	\caption{Branch $\phi_1^0$ of the supercritical pitchfork bifurcation at $\alpha_1^0\approx 5.12$ for the Beverton-Holt growth-dispersal IDE with right-hand side \eqref{rhsbh}, Laplace kernel \eqref{kerlap} and $a=1$, $L=2$ (left). 
	Total population over $\alpha\in[\alpha_1^0,7]$ along $\phi_1^0$ (right)}
	\label{figlaplace2}
\end{figure}
\begin{example}[Laplace kernel]\label{ex_lap-growth-dispersal}
	For $a,L>0$, $\Omega=[-\tfrac{L}{2},\tfrac{L}{2}]$ and $\mu$ being the Lebesgue measure, consider again the \emph{Laplace kernel} \eqref{kerlap}, which is symmetric and positive. It thus yields an operator $\sK\in L(C[-\tfrac{L}{2},\tfrac{L}{2}])$ as desired with index set $I=\N_0$ and furthermore an even IDE \eqref{deq}. Rather than working with \pref{prop_total-positivity}, the eigenpairs can be explicitly computed (cf.~\cite[Appendix~2]{reimer:bonsall:maini:16}): If
	\begin{itemize}
		\item $\tan(\tfrac{aL}{2}\nu)=\tfrac{1}{\nu}$ has the positive solutions $\nu_0<\nu_2<\ldots$, 

		\item $\cot(\tfrac{aL}{2}\nu)=-\tfrac{1}{\nu}$ has the positive solutions $\nu_1<\nu_3<\ldots$, 
	\end{itemize}
	then $\sigma(\sK)\setminus\set{0}=\set{\lambda_i\in\R:\,i\in\N_0}\subset(0,1)$ with the strictly decreasing sequence of simple eigenvalues $\lambda_i=\tfrac{1}{1+\nu_i^2}$ and the associate eigenfunctions 
	\begin{align*}
		\xi^{i}:[-\tfrac{L}{2},\tfrac{L}{2}]&\to\R,&
		\xi^{i}(x)
		&:=
		\begin{cases}
			\cos(a\nu_ix),&i\text{ is even},\\
			\sin(a\nu_ix),&i\text{ is odd}
		\end{cases}
		\fall i\in\N_0. 
	\end{align*}
	Also here the eigenvalues $\lambda_i$ depend only on the product $aL$, and not on $a,L$ individually. The associate eigenfunctions $\xi^{i}$ are even for even indices $i\in\N_0$ and odd for odd $i$. Finally we explicitly compute $\int_{-\tfrac{L}{2}}^{\tfrac{L}{2}} \xi^i(x)^3\d x\neq 0$ for all even $i$.
	
	\textbf{Case $i$ is even}: \pref{trivtrans} and the preceding computations thus yield
	\begin{align*}
		g_{20}(i)
		&
		\begin{cases}
			\neq 0\text{ for even }i\in\N_0,\\
			=0\text{ for odd }i\in\N_0,
		\end{cases}&
		g_{30}(i)
		&>
		0\fall i\in\N_0.
	\end{align*}
	Therefore, the trivial solution of \eqref{deq} with right-hand side \eqref{rhsbh} bifurcates transcritically into a $\theta$-periodic branch $\phi_i^0$ at the countably many critical parameters $\alpha_i^0$ for even $i$ (cf.\ \pref{proptrans}).

	\textbf{Case $i$ is odd}: One would suspect that for odd indices $i\in\N$, pitchfork bifurcations from the trivial branch take place at each $\alpha_i^0$. However, as $D_1^2\sF_t(0,\alpha_i^0)\neq 0$ for all $t\in\Z$, our earlier comment about the difficulty of verifying a pitchfork bifurcation applies. On the other hand, the structure of the Laplace kernel \eqref{kerlap} allows us to verify \pref{barwprop} explicitly. To do so, we begin with a useful remark: For any odd $i\in\N$, consider the corresponding eigenpair $\bigl(\tfrac{1}{1+\nu_i^2},\sin(a\nu_i \cdot)\bigr)$ of $\sK$. Suppressing the index $i$ for readability, we observe that as $\nu$ solves $\cot(\tfrac{aL}{2}\nu) = -\tfrac{1}{\nu}$, one has 
$$
	\frac{1}{\nu^2} = \cot(\tfrac{aL}{2}\nu)^2 = \frac{\cos(\tfrac{aL}{2}\nu)^2}{\sin(\tfrac{aL}{2}\nu)^2} = \frac{1 - \sin(\tfrac{aL}{2}\nu)^2}{\sin(\tfrac{aL}{2}\nu)^2},
$$
yielding $\sin(\tfrac{aL}{2}\nu) = \pm\tfrac{\nu}{\sqrt{1+\nu^2}}$ and $\cos(\tfrac{aL}{2}\nu) = \mp\tfrac{1}{\sqrt{1+\nu^2}}$, which allows for significant simplification of the upcoming calculations.

In order to find $\bar w$ from \pref{barwprop}, we define for each $\zeta\in\C$ the differential respectively boundary operators 
	\begin{align*}
		\mathfrak{d}_\zeta:C^2[-\tfrac{L}{2},\tfrac{L}{2}]&\to C[-\tfrac{L}{2},\tfrac{L}{2}],&
		[\mathfrak{d}_\zeta v](x)
		&:=
		v''(x) - a^2\bigl(1 - \tfrac{\zeta}{\lambda}\bigr)v(x), \\
		\mathfrak{b}:C^2[-\tfrac{L}{2},\tfrac{L}{2}]&\to\R^2,&
		[\mathfrak{b} v](x)
		&:=
		\begin{bmatrix}
			v'(-\tfrac{L}{2}) - a v(-\tfrac{L}{2})\\
			v'(\tfrac{L}{2}) + a v(\tfrac{L}{2})
		\end{bmatrix}.
	\end{align*}
	The key to solving the periodic eigenvalue problem for the Laplace kernel \eqref{kerlap} is the relation
	\begin{align*}
		\mathfrak{d}_0\sK v &= - a^2 v,&
		\mathfrak{b}\sK v &= \svector{0}{0}
		\fall v\in C[-\tfrac{L}{2},\tfrac{L}{2}]
	\end{align*}
	(cf.~\cite{jacobsen:mcadam:14} for a similar approach). Applying this to \eqref{fredeq} yields the linearly inhomogeneous $2\theta$-th order differential equation
	\begin{equation}\label{lapdeq}
		\mathfrak{d}_0^\theta \bar w = (-\tfrac{a^2}{\lambda})^\theta \bar w + \mathfrak{d}_0^\theta(\xi_0^\ast)^2.
	\end{equation}
	Additionally, $\mathfrak{d}_0 1 = - a^2$ and $\mathfrak{d}_0 \cos(2 a\nu \cdot) = - a^2\tilde{\lambda}^{-1}\cos(2 a\nu\cdot)$, where $\tilde{\lambda}:=\tfrac{1}{1+4\nu^2}$. From this, one may observe that $\tilde w(x) := B_0 + B_1\cos(2a\nu x)$ is a particular solution of \eqref{lapdeq}, where	
	\begin{align*}
		B_0 &:= \frac{\lambda^\theta}{2\lambda^\theta - 2},&
		B_1 &:= \frac{\lambda^\theta}{2\tilde\lambda^\theta - 2\lambda^\theta}. 
	\end{align*}
	Let $\zeta_t:=e^{2\pi\iota \frac{t}{\theta}}\in\C$ denote the $\theta$-th roots of unity. Regarding the linearly homogeneous differential equation $\mathfrak{d}_0^\theta \bar w = (-\tfrac{ a^2}{\lambda})^\theta \bar w$, its characteristic polynomial has exactly the roots $\zeta_0,\ldots,\zeta_{\theta-1}$, and so if $Z_t:=\sqrt{\tfrac{\zeta_t}{\lambda}-1}$, then the general solutions are of the form 
	$$
		x\mapsto\sum_{t=0}^{\theta-1}\intoo{C_t\cos(aZ_t x) + S_t\sin(aZ_t x)}
	$$
	with complex coefficients $C_t,S_t$ (cf.~\cite[p.~190, (14.5)~Thm.]{amann:90}). Next, we note that 
	\begin{align*}
		\mathfrak{d}_{\zeta_t} \cos(aZ_r \cdot) = & - a^2\tfrac{\zeta_r-\zeta_t}{\lambda}\cos(aZ_r \cdot), &
		\mathfrak{d}_{\zeta_t} \sin(aZ_r \cdot) = & - a^2\tfrac{\zeta_r-\zeta_t}{\lambda}\sin(aZ_r \cdot),\\
		\mathfrak{d}_{\zeta_t} 1 = & - a^2(1 - \tfrac{\zeta_t}{\lambda}), &
		\mathfrak{d}_{\zeta_t} \cos(2 a\nu \cdot) = & - a^2(\tilde{\lambda}^{-1} - \tfrac{\zeta_t}{\lambda})\cos(2 a\nu \cdot)
	\end{align*}
	and in particular $\mathfrak{d}_{\zeta_t}\cos(aZ_t \cdot) = \mathfrak{d}_{\zeta_t}\sin(aZ_t \cdot) = 0$. Making use of the identities
	\begin{align*}
		\prod_{t=1}^{\theta-1}(1 - \zeta_t)
		&=
		\theta,&
		\prod_{t=0}^{\theta-1}(p - \zeta_t q)
		&=
		p^\theta - q^\theta\fall p,q\in\C
	\end{align*}
	suggests the boundary conditions
	\begin{align*}
		0 & = \mathfrak{b}\mathfrak{d}_{\zeta_{\theta-1}}\cdots \widehat{\mathfrak{d}_{\zeta_t}}\cdots \mathfrak{d}_{\zeta_0} (\bar w - (\xi^\ast_0)^2) \\
		& = - a(- a^2)^{\theta-1}
		\begin{bmatrix}
			\zeta_t^{\theta-1}\theta\intoo{C_t\tilde C_t - S_t\tilde S_t} - \tilde B_t\\
			\zeta_t^{\theta-1}\theta\intoo{-C_t\tilde C_t - S_t\tilde S_t} + \tilde B_t
		\end{bmatrix}
		\fall 0\leq t<\theta, 
	\end{align*}
	where $\tilde B_t := \frac{(2+\zeta_t-\lambda)(\lambda-1)}{(\zeta_t-\lambda)(\zeta_t - \lambda\tilde{\lambda}^{-1})}$, 
	\begin{align*}
		\tilde C_t &:= \cos(\tfrac{aL}{2}Z_t) - Z_t\sin(\tfrac{aL}{2}Z_t),&
		\tilde S_t &:= Z_t\cos(\tfrac{aL}{2}Z_t) + \sin(\tfrac{aL}{2}Z_t)
	\end{align*}
	and the hat $\widehat{\cdot}$ denotes excluding the operator. We immediately obtain $S_t\tilde S_t = 0$ for all $t\in\Z$. Moreover, $\tilde B_t\neq 0$ for all $t$, by e.g.\ writing out its real and imaginary parts, implying $C_t\tilde C_t\neq 0$ for all $t$. Thus follows $C_t = \frac{\zeta_t\tilde B_t}{\theta \tilde C_t}$, leaving the $S_t$ unspecified. We can therefore define the function
	$$
		\bar w(x) := B_0 + B_1\cos(2a\nu x) + \sum_{t=0}^{\theta-1}C_t\cos(aZ_t x).
	$$
	Since $\bar w$ is even, $z'(\bar w) = \int_{-L/2}^{L/2}\xi_0^\ast(x)\bar w(x)\d x = 0$, showing that it is the requested unique $\bar w\in R(\sK-\lambda I_{C(\Omega)})$ from \pref{barwprop}. As $\mathfrak{d}_0 \cos(aZ_t\cdot) = - a^2\tfrac{\zeta_t}{\lambda}\cos(aZ_t\cdot)$ holds for all $t\in\Z$, returning to \eqref{fredeq} leads to
	$$
	\sK^t\bar w = B_0 + \tilde{\lambda}^tB_1\cos(2a\nu\cdot) + \lambda^t\sum_{r=0}^{\theta-1}\zeta_r^{-t}C_r\cos(aZ_r\cdot).
	$$
	Finally, we combine this with \pref{barwprop} to obtain the expression
	\begin{align}
		\bar g
		=&
		\intoo{c_3+d_3 - 3(c_2^2+c_2d_2+d_2^2)}\intoo{\tfrac{3 L}{8} + \tfrac{\lambda(5-2\lambda)}{4a}}
		\sum_{t=0}^{\theta-1}a_t^2
		\notag\\
		&+
		3\tfrac{(c_2+d_2)^2}{4a}
		\sum_{t=0}^{\theta-1}
		\intoo{\sum_{r=0}^{t-1}a_{\theta+r-t}a_r + \sum_{r=t}^{\theta-1}a_ra_{r-t}}\Biggr(\Biggr.
		2B_0 \tfrac{a L+2 \lambda}{\lambda^t}
		\notag\\
		&
		\qquad
		-B_1 (a L+2 \lambda (3-2 \lambda)) (\tfrac{\tilde\lambda}{\lambda})^{t}-\tfrac{8(3-\lambda)\lambda}{\theta}
		\notag\\
		&
		+8(1-\lambda)\sum_{r=1}^{\theta -1}
		C_r\zeta_r^{-t}\tfrac{(\lambda Z_r^2-2)\sin \left(\tfrac{aL}{2} Z_r\right) - 2 \lambda Z_r \cos\left(\tfrac{aL}{2} Z_r\right)}{Z_r \left(\lambda(Z_r^2+4)-4\right)} 
		\Biggl.\Biggr)
		\notag\\
		&
		-3c_2d_2\tilde{\lambda}^2\tfrac{
		36\lambda^4 - 78\lambda^3 - 16\lambda^2 + 96\lambda - 32 + aL(8-5\lambda)\tfrac{\lambda}{\tilde{\lambda}}	+ \tfrac{8(\lambda^2-3\lambda+2)^2}{e^{aL}}
		}{8a\lambda^3}
		\sum_{t=0}^{\theta-1}
		a_ta_{t+1}
		\label{gbarterm}
	\end{align}
	determining whether a super- or subcritical pitchfork bifurcation occurs. 
	
	This situation simplifies significantly in an autonomous setting, where $\theta=1$, $\beta_0=1$ and we additionally impose $c_2d_2 = 0$. Here, an elementary, albeit tedious, consequence of our assumptions and \eqref{gbarterm} yields
	\begin{equation}
		\bar g(i) 
		= 
		(c_3+d_3)\intoo{\tfrac{3 L}{8}+\tfrac{3+5\nu_i^2}{4a\left(1+\nu_i^2\right)^2}}
		-
		(c_2+d_2)^2\tfrac{\left(15 a L\left(1+\nu_i^2\right)^3+30+80\nu_i^2+66\nu_i^4\right)}{24 a\intoo{\nu_i+\nu_i^3}^2}
		\label{no3}
	\end{equation}
	and the following hold using \tref{thmcross} with Props~\ref{proptrans}, \ref{proppitch}: 
	\begin{itemize}
		\item $c_3+d_3\leq\frac{5(c_2+d_2)^2}{3}\Leftrightarrow\bar g(i)<0$, that is, a supercritical pitchfork bifurcation from the trivial branch takes place at each $\alpha_i^0$.\\
		One obtains this by observing that $\bar g(i) < 0$ if and only if 
		\begin{equation}\label{ineq_subpitchcond}
			c_3+d_3
			< 
			\tfrac{(c_2+d_2)^2}{3}\left(5 + \tfrac{5}{\nu ^2} + \tfrac{16 \nu ^2}{3 a L \left(1+\nu ^2\right)^2+6+10 \nu ^2}\right). 
		\end{equation}
		As $\tfrac{\d}{\d\nu}\bigl(\cot(\tfrac{aL}{2}\nu) + \tfrac{1}{\nu}\bigr) = -\tfrac{aL}{2}\csc(\tfrac{a L}{2}\nu) - \tfrac{1}{\nu^2} < 0$ wherever defined and 
		\begin{align*}
			\lim_{\nu\nearrow\tfrac{2\pi i}{a L}}
			\bigl(\cot(\tfrac{aL}{2}\nu) + \tfrac{1}{\nu}\bigr) &= -\infty,&
			\lim_{\nu\searrow\tfrac{2\pi i}{a L}}
			\bigl(\cot(\tfrac{aL}{2}\nu) + \tfrac{1}{\nu}\bigr) &= \infty,
		\end{align*}
		it follows that $\cot(\tfrac{aL}{2}\nu)=-\tfrac{1}{\nu}$ possesses exactly one solution in each interval $\bigl(\tfrac{2\pi i}{a L},\tfrac{2\pi (i+1)}{a L}\bigr)$. 
		Hence, as $i\to\infty$, one has $\nu_i\to\infty$, and, as is easily verified, the right-hand side of \eqref{ineq_subpitchcond} strictly monotonously decreases towards $\tfrac{5(c_2+d_2)^2}{3}$; thus, it is the largest number satisfying
		$$
			\tfrac{(c_2+d_2)^2}{3}\left(5 + \tfrac{5}{\nu ^2} + \tfrac{16 \nu ^2}{3 a L \left(1+\nu ^2\right)^2+6+10 \nu ^2}\right)
			>
			\tfrac{5(c_2+d_2)^2}{3}
			\quad\text{for all odd }i\in\N_0.
		$$

		\item Conversely, if $c_3+d_3>\frac{5(c_2+d_2)^2}{3}$, then there is some odd $i_0\in\N$ so that a supercritical pitchfork bifurcation from the trivial branch takes place at $\alpha_i^0$ if $i<i_0$, while a subcritical pitchfork bifurcation takes place at $\alpha_i^0$ if $i>i_0$.\\
		This readily results from the monotonicity of the right-hand side of \eqref{ineq_subpitchcond}.
	\end{itemize}

	\textbf{Conclusion} of \eref{ex_lap-growth-dispersal}: Since $c_2=-2$, $c_3=6$, $d_2=d_3=0$ (cf.\ \rref{remgrowth}), one has $c_3+d_3 = 6 \leq \tfrac{20}{3} = \tfrac{5(c_2+d_2)^2}{3}$ and we can summarize:
\begin{itemize}
	\item For every even $i\in\N_0$, a transcritical bifurcation of $\theta$-periodic solutions takes place at $(0,\alpha_i^0)$. \pref{prop_branches-odd-even} yields that until another bifurcation occurs along these nontrivial branches $\phi_i^0$, they each consist of only even functions. Due to \pref{primary_positive}, the trivial branch is exponentially stable for $\alpha < \alpha_0^0$, and unstable for $\alpha > \alpha_0^0$, with exponential stability being transferred to the nontrivial branch $\phi_0^0$ of $\theta$-periodic solutions (see \fref{figlaplace1} (center)). For $\alpha>\alpha_0^0$, this only biologically relevant branch consists of the nonnegative functions guaranteed by \cite[Thm.~5.1]{hardin:takac:webb:88}; for $\alpha<\alpha_0^0$ the functions $\phi_0^0(\alpha)$ have negative values.
	
	\item For every odd $i\in\N$ and $\theta=1$, a supercritical pitchfork bifurcation of fixed points takes place at $(0,\alpha_i^0)$. \pref{prop_branches-odd-even} yields that until another bifurcation occurs along these nontrivial branches $\phi_i^0$, they consist of "symmetric" pairs of branches yielding the same total population (see \fref{figlaplace2}). This explains that the curve capturing the total population in \fref{figlaplace2} (right) appears to be a single line. 
\end{itemize}
	Globally, the branches $\phi_i^0$, $i\geq 2$ even (from transcritical bifurcations), appear to be part of supercritical folds (see \fref{figlaplace3} (right)), occurring at parameters $\alpha<\alpha_i^0$. 
\end{example}
\begin{figure}
	\includegraphics[width=52mm]{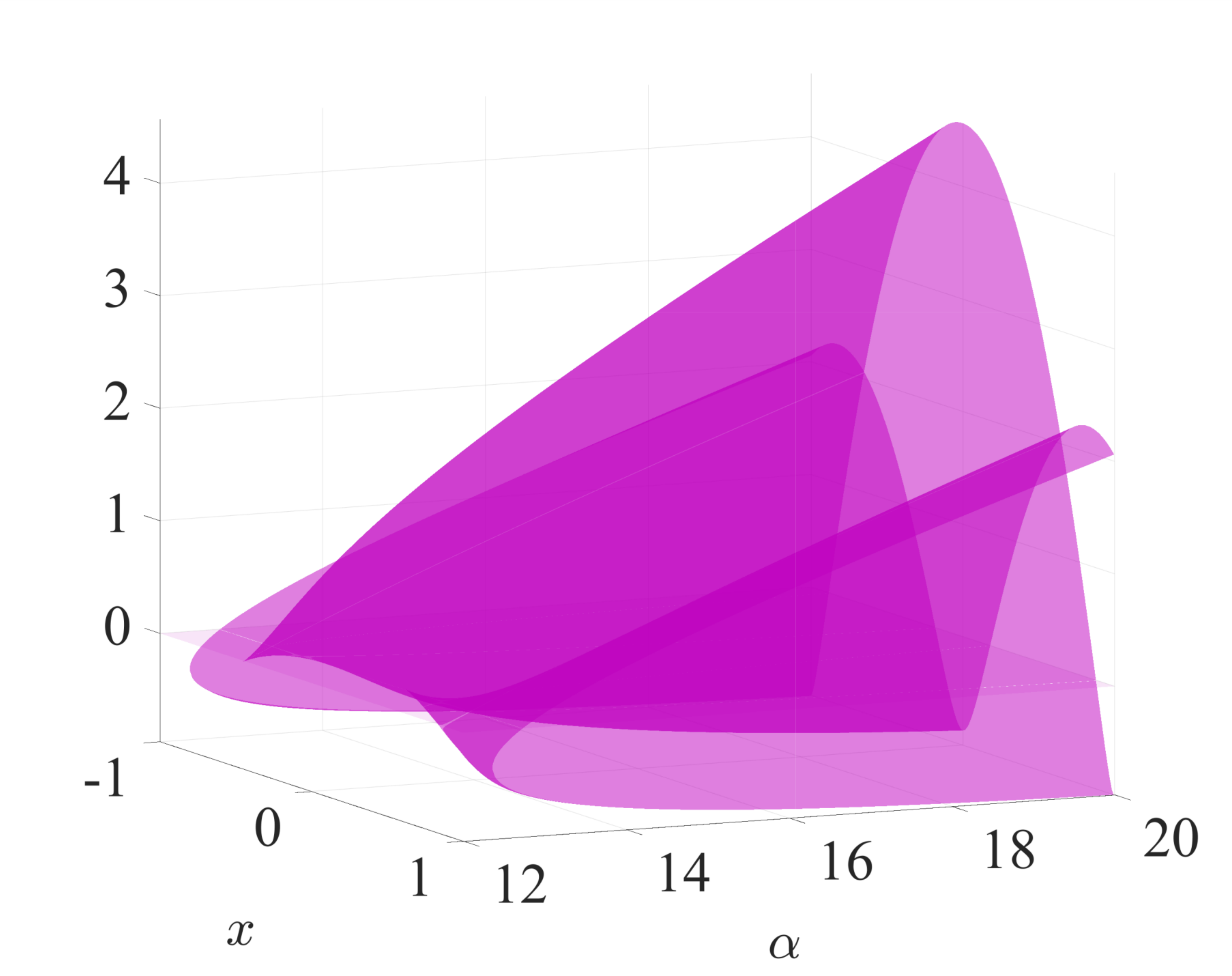}
	\includegraphics[width=52mm]{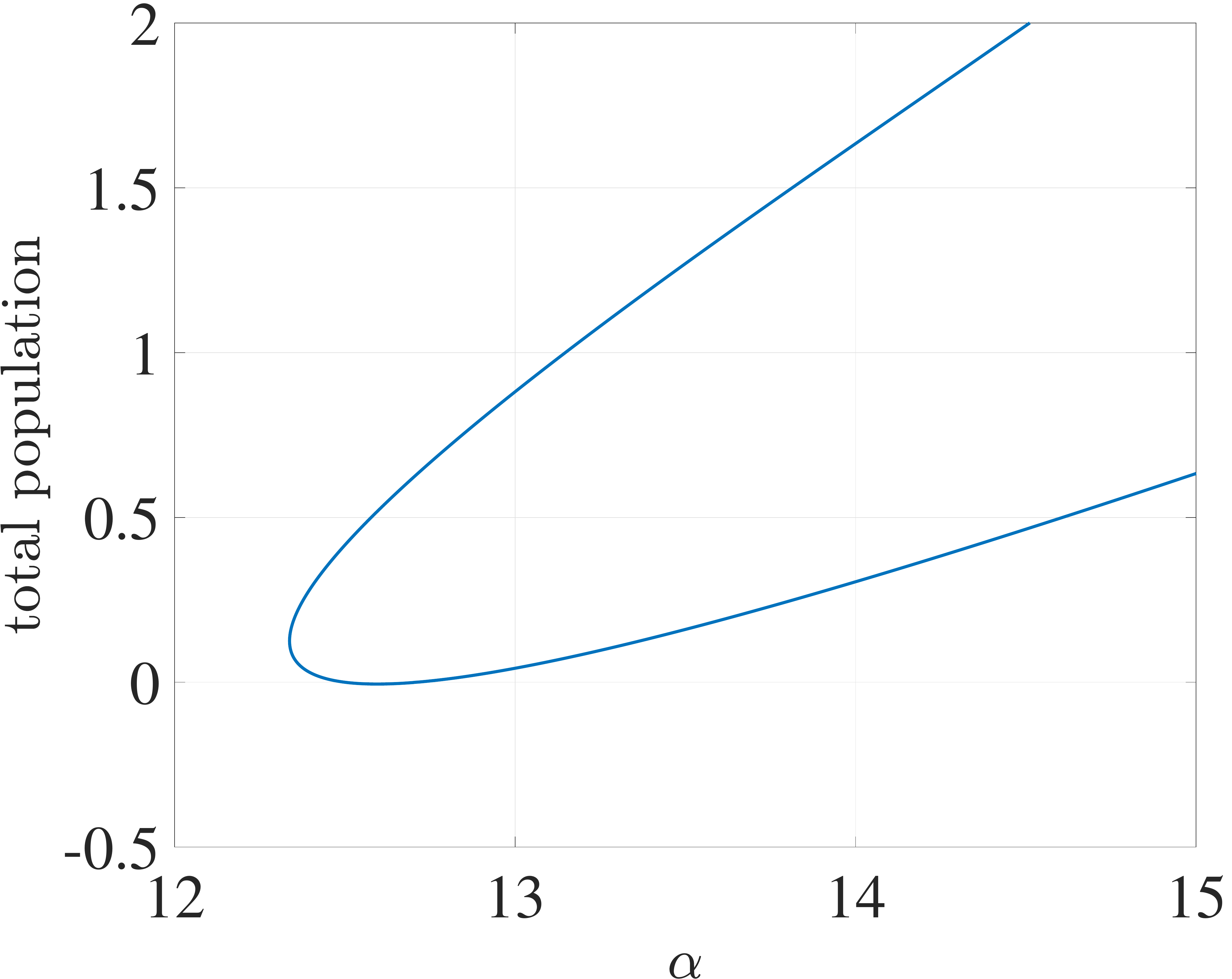}
	\caption{Branch $\phi_2^0$ of the transcritical bifurcation at $\alpha_2^0\approx 12.73$ for the Beverton-Holt growth-dispersal IDE with right-hand side \eqref{rhsbh}, Laplace kernel \eqref{kerlap} and $a=1$, $L=2$ as part of a fold (left). Total population over parameters $\alpha\in[12.4, 15]$ reflecting a fold in $\phi_2^0$ (right)}
	\label{figlaplace3}
\end{figure}
\subsubsection{Dispersal-growth Beverton-Holt equation}
\label{sec524}
Let us now focus on the $\theta$-periodic dispersal-growth IDE \eqref{deq} with right-hand side
\begin{equation}
	\sF_t(u,\alpha)
	:=
	\frac{\alpha\beta_{t+1}\int_\Omega k(\cdot,y)u(y)\d\mu(y)}{1+\int_\Omega k(\cdot,y)u(y)\d\mu(y)}
	\label{nodisgro}
\end{equation}
defined on $\set{u\in C(\Omega):\,\inf_{x\in\Omega}\int_\Omega k(x,y)u(y)\d\mu(y)>-1}\tm(0,\infty)$. It is not of the form stipulated by \eqref{newrhs}, \eqref{cdform}, but this problem can be circumvented as follows: 

Rather than relying on \rref{remtrans}, we apply the transformation $\bar u_t=\tfrac{1}{\alpha\beta_t}u_t$ leading to the modified IDE
\begin{equation}
	\tag{$\bar\Delta_{\alpha}$}
	\boxed{u_{t+1} = \bar\sF_{t}(u_t,\alpha),}
	\qquad
	\bar\sF_t(u,\alpha)
	:=
	\frac{\alpha\beta_t\int_\Omega k(\cdot,y)u(y)\d\mu(y)}{1+\alpha\beta_t\int_\Omega k(\cdot,y)u(y)\d\mu(y)},
	\label{nodisgromod}
\end{equation}
whose right-hand side is defined on the sets
\begin{align*}
	U_t&\equiv\set{u\in C(\Omega):\,\inf_{x\in\Omega}\int_\Omega k(x,y)u(y)\d\mu(y)>-\frac{1}{\bar\alpha\beta_t}},&
	A&=(0,\bar\alpha)
\end{align*}
depending on a fixed real $\bar\alpha>0$. It is immediate that $\phi\in\ell_\theta$ solves $(\Delta_\alpha)$ if and only if $\bar\phi\in\ell_\theta$ solves $(\bar\Delta_\alpha)$ whenever both are defined. Now the right-hand side of \eqref{nodisgromod} can be written as \eqref{newrhs} with functions $G(z)=\tfrac{z}{1+z}$, $f(x,y,z)=k(x,y)z$ and hence
\begin{align*}
	G'(z)&=1,&
	D_3f(x,y,z)&=k(x,y).
\end{align*}
Given $u\in U_t$, $\alpha\in(0,\bar\alpha)$, one has derivatives
\begin{align*}
	D_1\bar\sF_t(u,\alpha)v
	&=
	\frac{\alpha\beta_t}{\Bigl(1+\alpha\beta_t\int_\Omega k(\cdot,y)u(y)\d\mu(y)\Bigr)^2}\int_\Omega k(\cdot,y)v(y)\d\mu(y),&&\\
	D_1D_2\bar\sF_t(u,\alpha)v
	&=
	\frac{\beta_t-\alpha\beta_t^2\int_\Omega k(\cdot,y)u(y)\d\mu(y)}{\Bigl(1+\alpha\beta_t\int_\Omega k(\cdot,y)u(y)\d\mu(y)\Bigr)^3}\int_\Omega k(\cdot,y)v(y)\d\mu(y),&&&\\
	%D_2^2\bar\sF_t(u,\alpha)&=0,\\
	D_1^2\bar\sF_t(u,\alpha)v^2
	&=
	-\frac{2(\alpha\beta_t)^2}{\Bigl(1+\alpha\beta_t\int_\Omega k(\cdot,y)u(y)\d\mu(y)\Bigr)^3}\intoo{\int_\Omega k(\cdot,y)v(y)\d\mu(y)}^2,&&\\
	D_1^3\bar\sF_t(u,\alpha)v^3
	&=
	\frac{6(\alpha\beta_t)^3}{\Bigl(1+\alpha\beta_t\int_\Omega k(\cdot,y)u(y)\d\mu(y)\Bigr)^4}\intoo{\int_\Omega k(\cdot,y)v(y)\d\mu(y)}^3&&
\end{align*}
and along the trivial solution $\phi_0^0(\alpha)\equiv 0$ of $(\bar\Delta_\alpha)$ they simplify to
\begin{align*}
	D_1\bar\sF_t(0,\alpha)v&=\alpha\beta_t\sK v,&
	D_1^2\bar\sF_t(0,\alpha)v^2&=-2\intoo{\alpha\beta_t\sK v}^2,\\
	D_1D_2\bar\sF_t(0,\alpha)v&=\beta_t\sK v,&
	D_1^3\bar\sF_t(0,\alpha)v^3&=6\intoo{\alpha\beta_t\sK v}^3\fall t\in\Z,\,v\in C(\Omega).
\end{align*}
The earlier computations now hold with $c_2=c_3=0$, $d_2=-2$, $d_3=6$ (cf.\ \rref{remgrowth}). Indeed, for any eigenpair $(\lambda_i,\xi^i)$ of $\sK$, choose $\bar\alpha$ such that $\tfrac{1}{\bar\alpha^\theta}<\lambda_i^\theta\beta_0\cdots\beta_{\theta-1}$, and observe that our previous computations of $g_{11}(i)$, $g_{20}(i)$ and $\bar g(i)$ in \eqref{no1}, \eqref{no2} and \eqref{gbarterm} respectively \eqref{no3} remain unchanged. 
\begin{figure}
	\begin{minipage}{20mm}
		\begin{tabular}{@{}rr@{}}
			\hline
			$i$ & $\alpha_i^0$\\
			\hline
			$0$ & $1.74$\\
			$1$ & $5.12$\\
			$2$ & $12.73$\\
			$3$ & $25.13$\\
			$4$ & $42.42$\\
			\hline
		\end{tabular}
	\end{minipage}
	\begin{minipage}{105mm}
		\includegraphics[width=52mm]{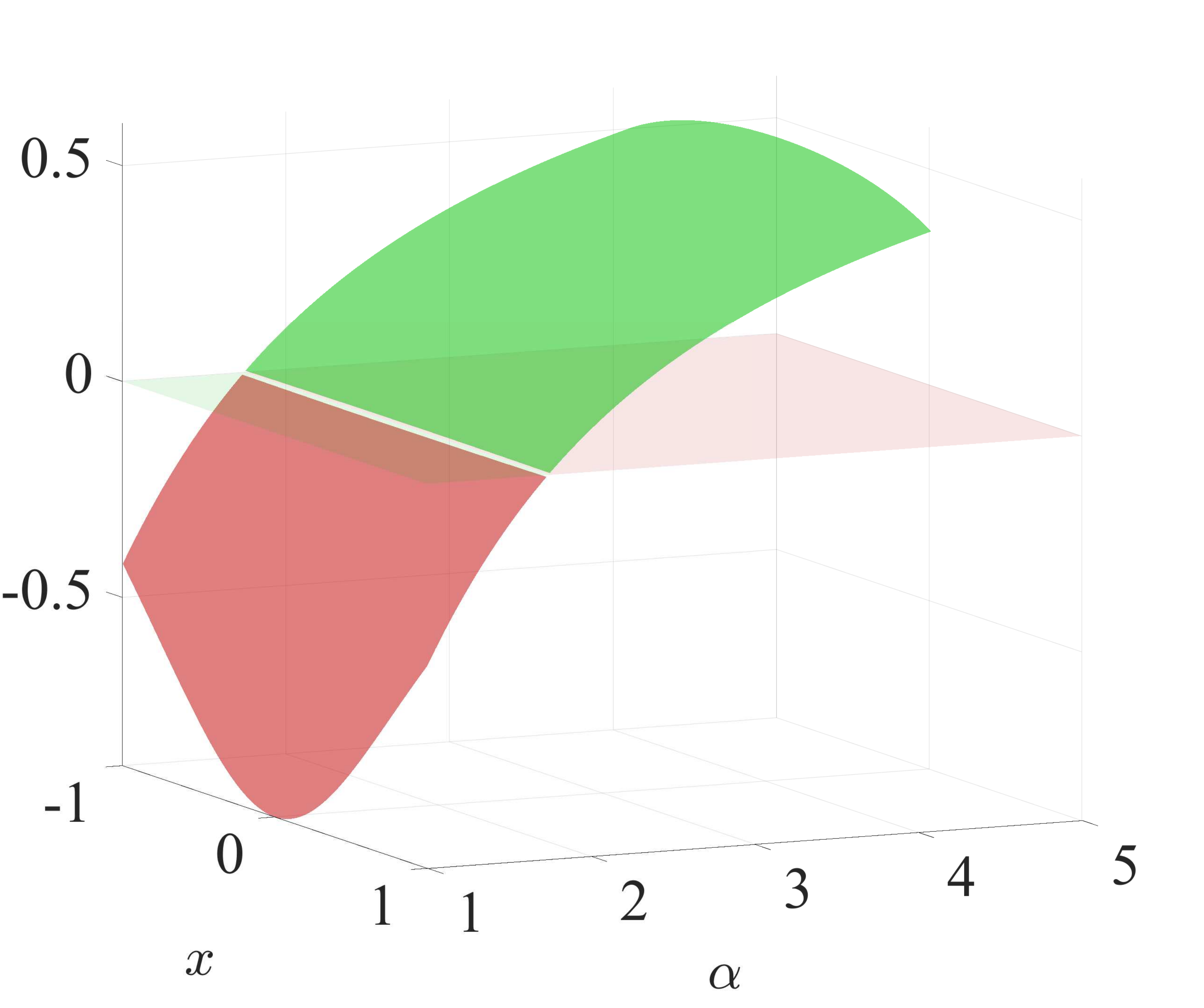}
		\includegraphics[width=52mm]{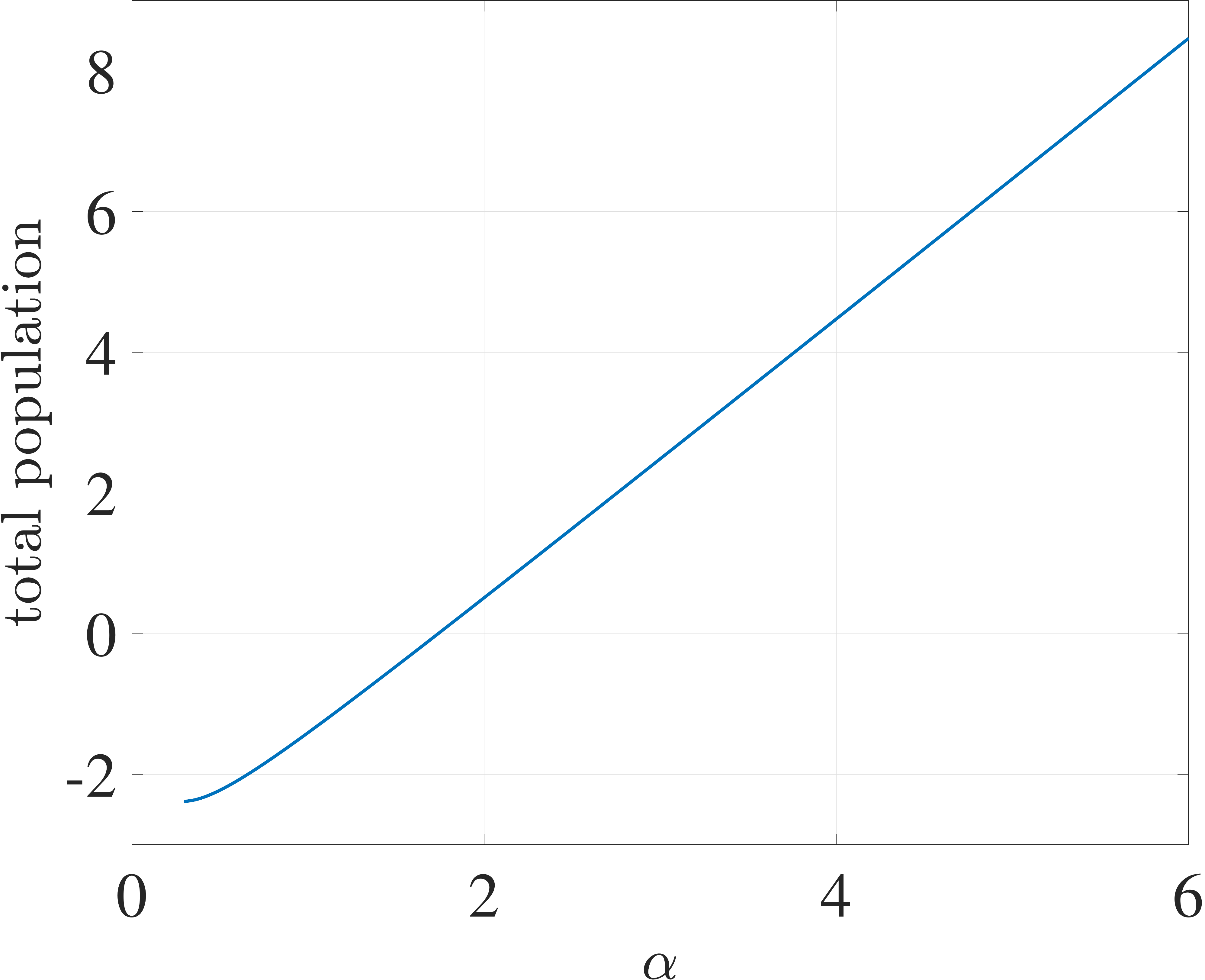}
	\end{minipage}
	\caption{First critical values $\alpha_i^0$ for bifurcations from the trivial branch (left). 
	Nontrivial branch $\phi_0^0$ of the primary transcritical bifurcation at $\alpha_0^0\approx 1.74$ for the Beverton-Holt dispersal-growth IDE with right-hand side \eqref{nodisgro}, Laplace kernel \eqref{kerlap} and $a=1$, $L=1$ (center). 
	Total population over $\alpha\in[0,6]$ along $\phi_0^0$ (right)}
	\label{figlaplace1s}
\end{figure}
\begin{figure}
	\includegraphics[width=52mm]{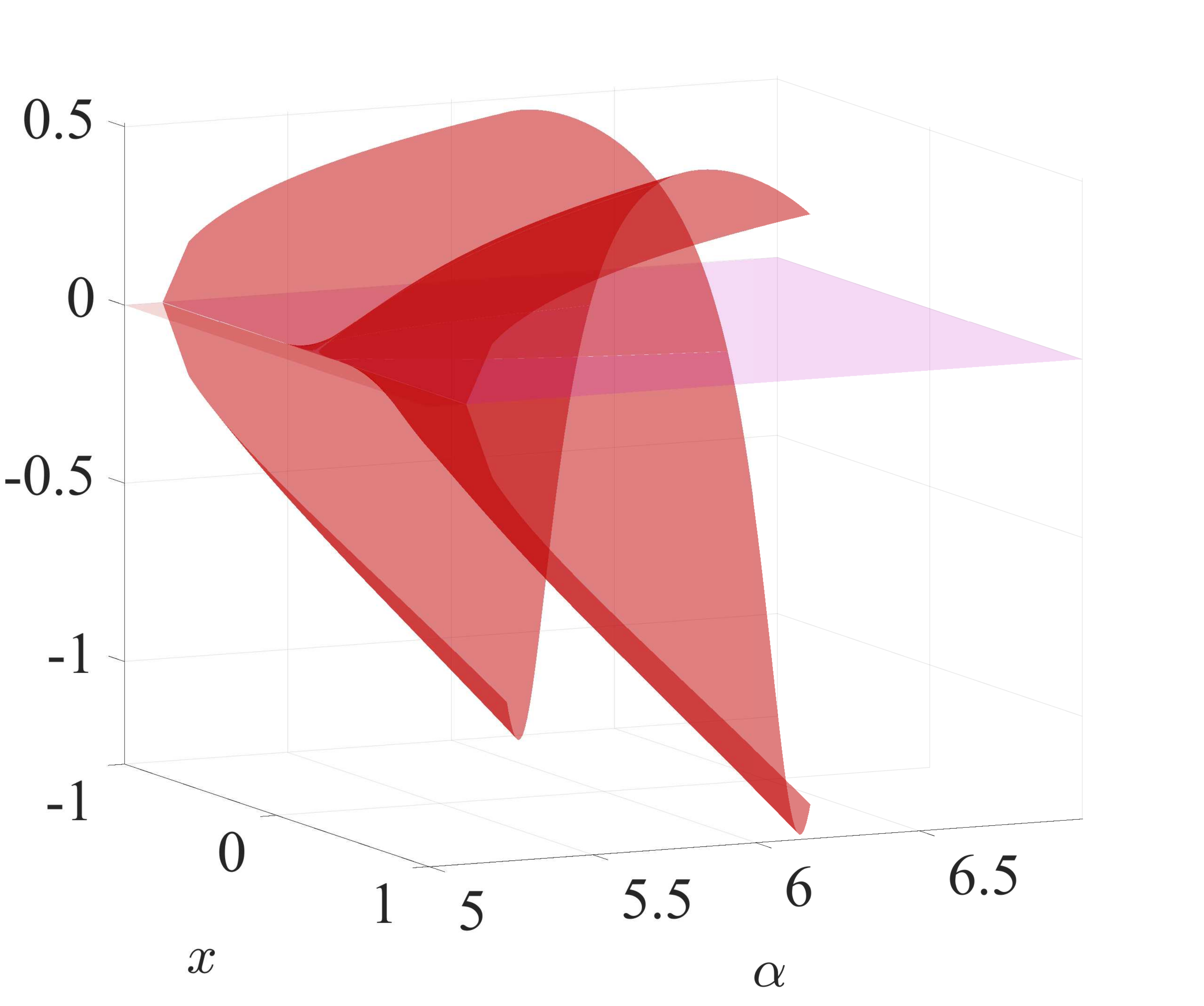}
	\includegraphics[width=52mm]{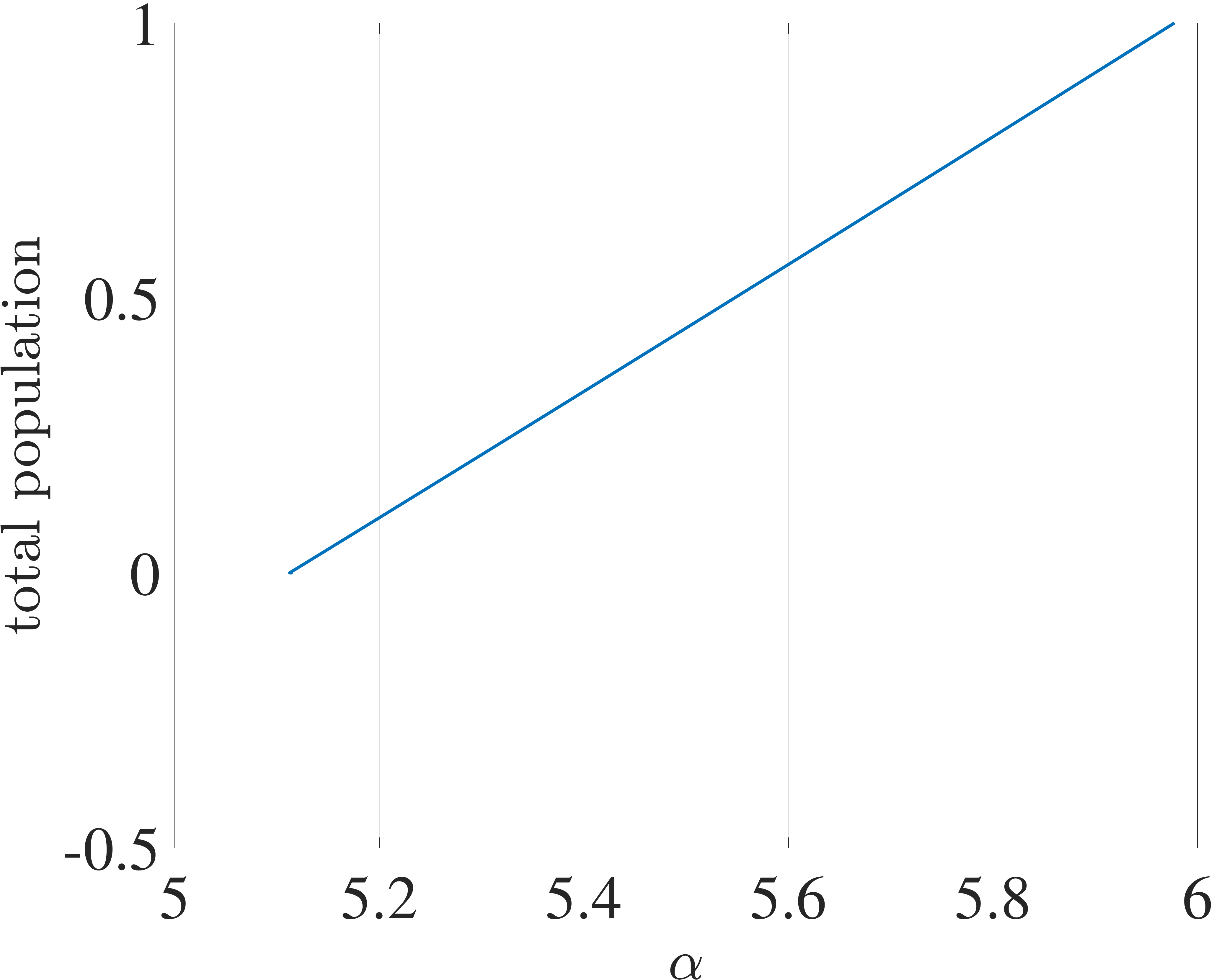}
	\caption{Branch $\phi_1^0$ of the supercritical pitchfork bifurcation at $\alpha_1^0\approx 5.12$ for the Beverton-Holt dispersal-growth IDE with right-hand side \eqref{nodisgro}, Laplace kernel \eqref{kerlap} and $a=1$, $L=2$ (left). Total population over $\alpha\in[5,6]$ along $\phi_1^0$ (right)}
	\label{figlaplace2s}
\end{figure}
\begin{example}[Laplace kernel]
	If we return to the Laplace kernel \eqref{kerlap}, then also \eqref{nodisgromod} becomes an even IDE. One sees by choosing $\bar\alpha>\alpha_i^0$ that the bifurcation behavior of \eqref{nodisgromod} at $(0,\alpha_i^0)$ is identical to that exhibited by the growth-dispersal IDE from \sref{sec523}. Moreover, it is clear that this behavior carries over to \eqref{deq}.
	In particular, we again observe countably infinitely many bifurcations along the trivial branch, alternating between transcritical and supercritical pitchfork bifurcations, with all the same remarks as for \eref{ex_lap-growth-dispersal}, as Figs.~\ref{figlaplace1s}--\ref{figlaplace3s} serve to illustrate.
\end{example}
\begin{figure}
	\includegraphics[width=52mm]{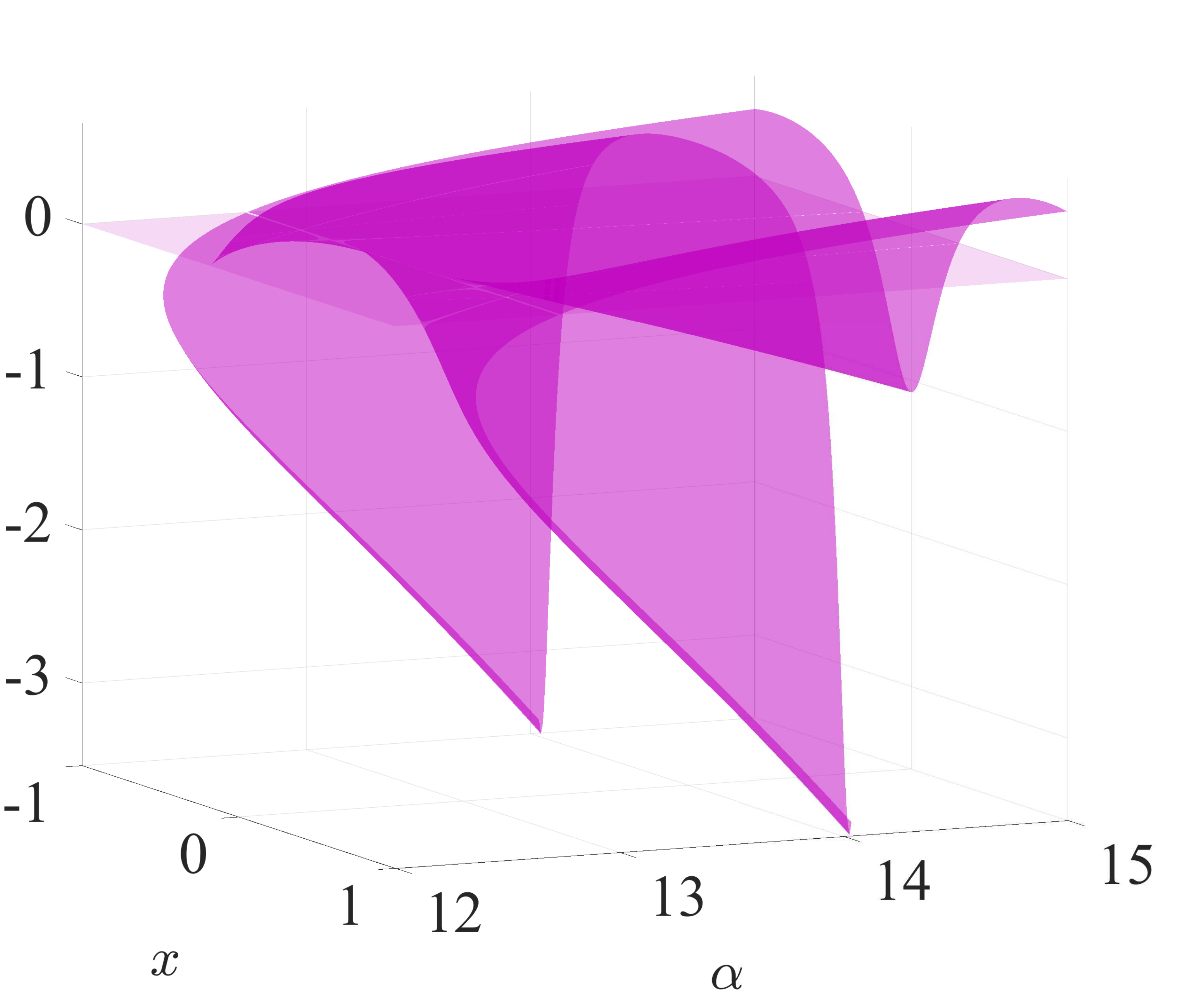}
	\includegraphics[width=52mm]{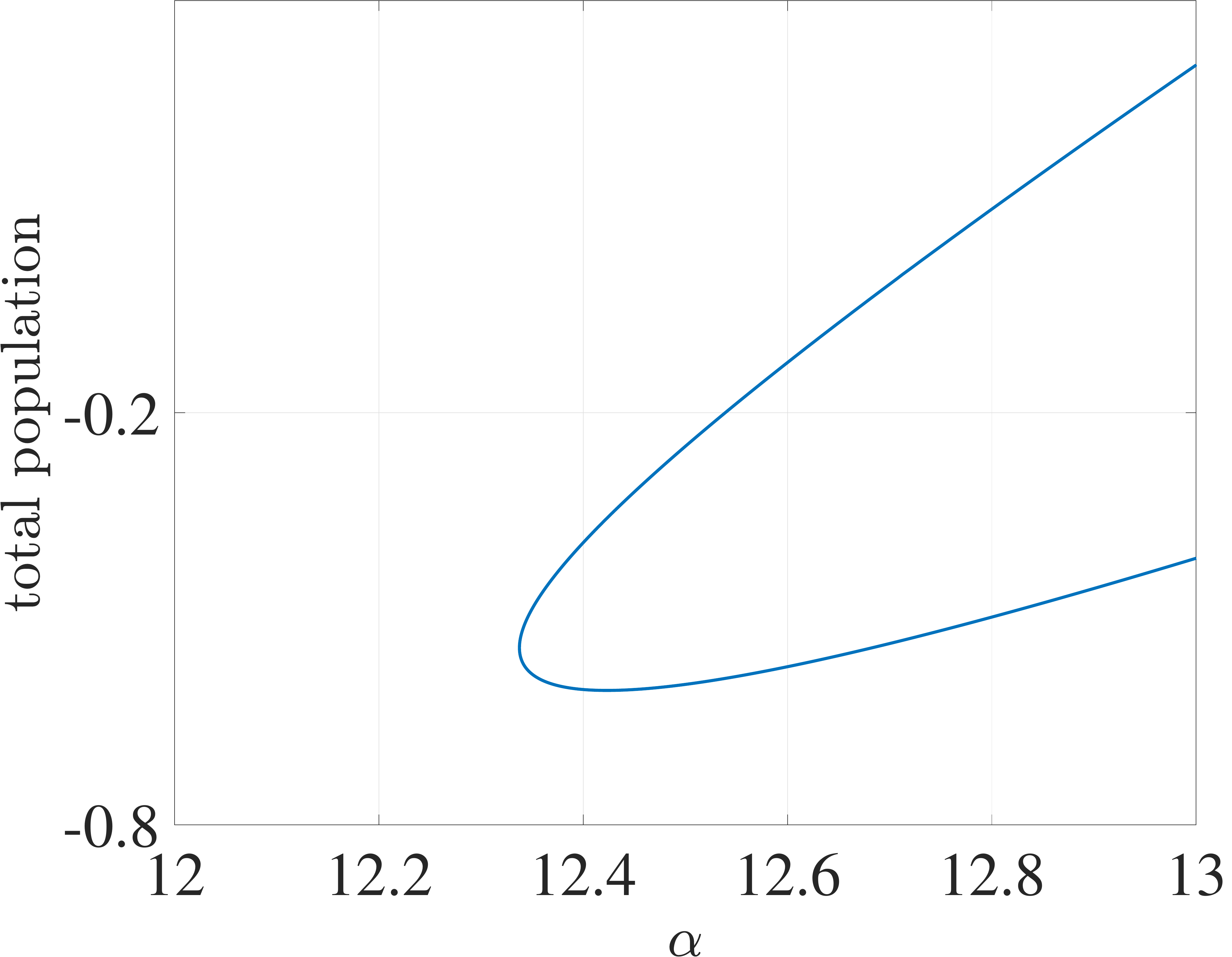}
	\caption{Branch $\phi_2^0$ of the transcritical bifurcation at $\alpha_2^0\approx 12.73$ for the Beverton-Holt dispersal-growth IDE with right-hand side \eqref{nodisgro}, Laplace kernel \eqref{kerlap} and $a=1$, $L=2$ as part of a fold (left). 
	Total population over parameters $\alpha\in[12,13]$ reflecting a fold in $\phi_2^0$ (right)}
	\label{figlaplace3s}
\end{figure}
\begin{figure}
	\includegraphics[scale=0.6]{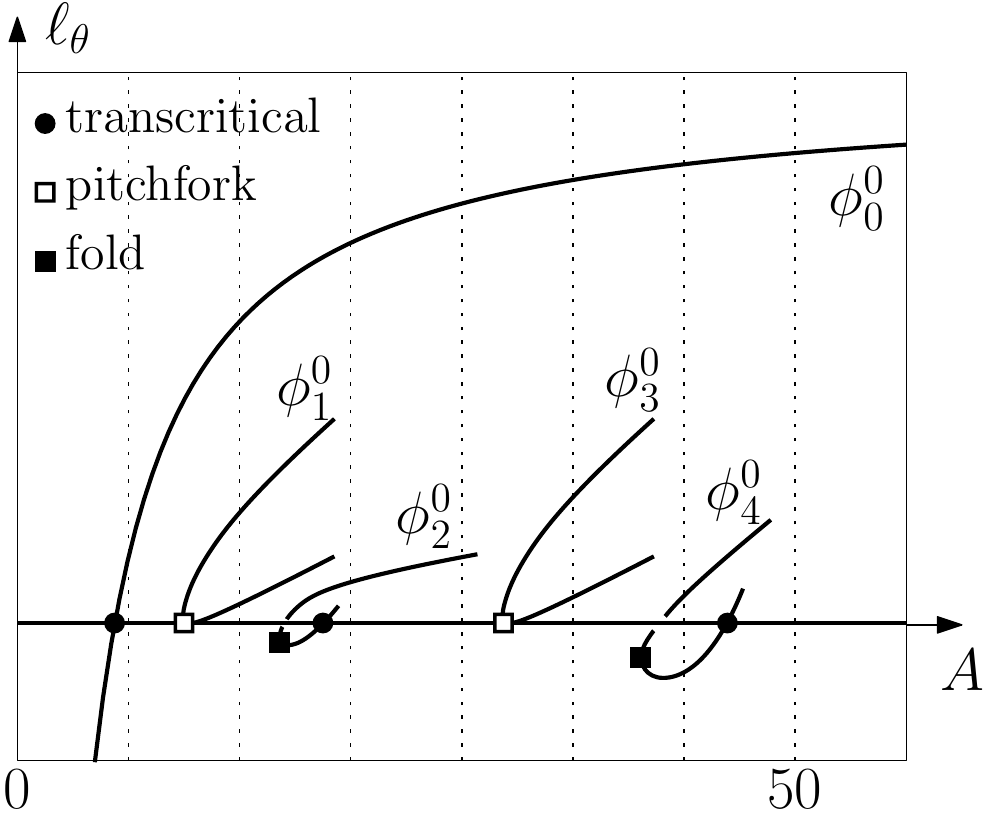}
	\includegraphics[scale=0.6]{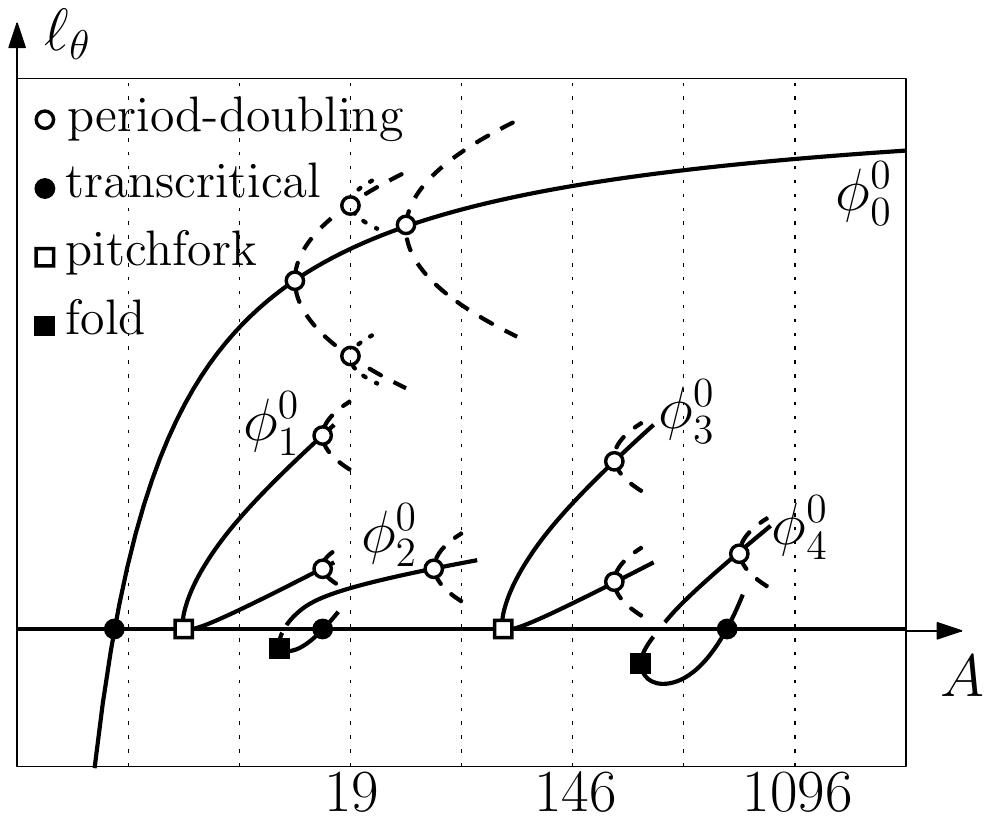}
	\caption{Schematic bifurcation diagrams: 
	Branches of $\theta$-pe\-riodic solutions in the Beverton-Holt IDE having the right-hand sides \eqref{rhsbh}/\eqref{nodisgro}, $\alpha\in[0,50]$ and the Laplace kernel \eqref{kergauss} with $a=1$, $L=2$ (left), as well as IDE with right-hand side \eqref{noricker} and $\alpha\in[0,3000]$ (logarithmic axis) for the Gau{\ss} kernel \eqref{kergauss} with $a=1$, $L=2$ (right). 
	Fixed point branches are solid $(\theta=1)$, branches of $2$-periodic solutions are dashed $(\theta=2)$ and $4$-periodic solutions are indicated by dotted lines $(\theta=4)$} 
	\label{figschem2}
\end{figure}
\subsection{Period doubling cascade}
\label{sec53}
Our final example is the numerically most challenging one, because solution branches are not known beforehand and must be computed using path-following methods. In addition, even within the biologically relevant cone, it exhibits the most complex dynamical behavior so far. 

We turn to autonomous scalar Ricker IDEs \eqref{deq} with right-hand side
\begin{equation}
	\sF(u,\alpha):=\alpha\int_\Omega k(\cdot,y)u(y)e^{-u(y)}\d\mu(y)
	\label{noricker}
\end{equation}
on $U_t\equiv C(\Omega)$, $A=(0,\infty)$. Solutions with different even periods were observed in \cite[Sect.~5]{kot:schaeffer:86} experimentally and in \cite[Fig.~10]{kirk:lewis:97} based on path-following, while \cite{andersen:91} discusses a dispersal-growth version of \eqref{noricker}. 

As Hammerstein operator, \eqref{noricker} fits in the framework of \eqref{newrhs} with $G(z)=z$, $f(x,y,z)=k(x,y)ze^{-z}$ and $\beta_t\equiv 1$ on $\Z$. This yields the derivatives
\begin{align*}
	G'(z)&=1,&
	D_3f(x,y,z)&=k(x,y)(1-z)e^{-z}. 
\end{align*}
Along the zero branch, we expect a behavior similar to what was observed above in \sref{sec523}. If $(\lambda,\xi^\ast)$ is a simple eigenpair of $\sK$ and $\alpha^\ast=\tfrac{1}{\lambda}>0$, then \eqref{cdform} is satisfied and \rref{remgrowth} yields $c_2=-2$, $c_3=3$, $d_2=d_3=0$. By \pref{trivtrans} the condition $\int_\Omega \xi^\ast(x)^3\d\mu(x)\neq 0$ is sufficient for a transcritical bifurcation at $\bigl(0,\tfrac{1}{\lambda}\bigr)$ --- in particular, \pref{primary_positive} yields a primary transcritical bifurcation with exchange of stability from the trivial branch at $\bigl(0,\tfrac{1}{r(\sK)}\bigr)$ --- while verifying a pitchfork bifurcation in case $\int_\Omega \xi^\ast(x)^3\d\mu(x) = 0$ requires computation of $\bar g$, as outlined in \pref{barwprop}. 

Concerning bifurcations along the nontrivial branches, let us suppose that $\phi^\ast$ denotes a $\theta_1$-periodic solution of $(\Delta_{\alpha^\ast})$ with right-hand side \eqref{noricker} and $\sigma_{\theta_1}(\alpha^\ast)\cap\S^1 = \set{-1}$ for some $\alpha^\ast>0$. Therefore, \tref{poinc} shows that $\phi^\ast$ can be continued to a smooth branch $\phi:B_\rho(\alpha^\ast)\to C(\Omega)$ of $\theta_1$-periodic solutions $\phi(\alpha)$ to the autonomous IDE \eqref{deq}, where $\psi=\dot\phi(\alpha^\ast)$ is uniquely determined by \eqref{poinc2}. In order to detect bifurcations along this global branch $\phi$, we make use of the corresponding equation of perturbed motion \eqref{deqp} and \rref{remshift}, which has the $\theta_1$-periodic right-hand sides 
$$
	\tilde\sF_t(u,\alpha)
	=
	\alpha\!\int_\Omega k(\cdot,y)
	\intcc{(u(y)+\phi(\alpha)_t(y))e^{-u(y)-\phi(\alpha)_t(y)}
	-
	\phi(\alpha)_t(y)e^{-\phi(\alpha)_t(y)}}\d\mu(y)
$$
and therefore the derivatives
\begin{align*}
	D_1\tilde\sF_t(0,\alpha^\ast)v&=-\alpha^\ast\int_\Omega k(\cdot,y)(\phi_t^\ast(y)-1)e^{-\phi_t^\ast(y)}v(y)\d\mu(y),\\
	D_1D_2\tilde\sF_t(0,\alpha^\ast)v&=\int_\Omega k(\cdot,y)\intcc{(1-\phi_t^\ast(y))\psi_t(y)+\phi_t^\ast(y)}
	e^{-\phi_t^\ast(y)}v(y)\d\mu(y),\\
	D_1^2\tilde\sF_t(0,\alpha^\ast)v^2&=\alpha^\ast\int_\Omega k(\cdot,y)(\phi_t^\ast(y)-2)e^{-\phi_t^\ast(y)}v(y)^2\d\mu(y),\\
	D_1^3\tilde\sF_t(0,\alpha^\ast)v^3&=-\alpha^\ast\int_\Omega k(\cdot,y)(\phi_t^\ast(y)-3)e^{-\phi_t^\ast(y)}v(y)^3\d\mu(y)
\end{align*}
and $D_2^2\tilde\sF_t(0,\alpha^\ast)=0$ for $t\in\Z$, $v\in C(\Omega)$. In order to verify bifurcation conditions for period doublings, we choose $\theta=2\theta_1$ and arrive at $1\in\sigma_\theta(\alpha^\ast)$ (cf.~\sref{sec43}). 

A more detailed analysis is possible for particular kernels. The subsequent one is popular in applications, but requires numerical methods throughout. For this purpose, our simulations are based on Nystr\"om discretizations (see App.~\ref{appB}) of the integral operators involved. To validate them, the midpoint rule ($N_n=100$), the trapezoidal rule ($N_n=101$), as well as the Chebyshev rule ($N_n=100$) are applied in order to compare the corresponding results. In any case, the simulations capture the behavior of the IDEs \eqref{deq} w.r.t.\ the measures $\mu_n$ given in App.~\ref{appB}. 
\begin{figure}
	\begin{minipage}{20mm}
	\begin{tabular}{@{}rr@{}}
		\hline
		$i$ & $\alpha_i^0$\\
		\hline
		$0$ & $1.36$\\
		$1$ & $3.31$\\
		$2$ & $13.23$\\
		$3$ & $78.08$\\
		$4$ & $617.01$\\
		\hline
	\end{tabular}
	\end{minipage}
	\begin{minipage}{105mm}
		\includegraphics[width=52mm]{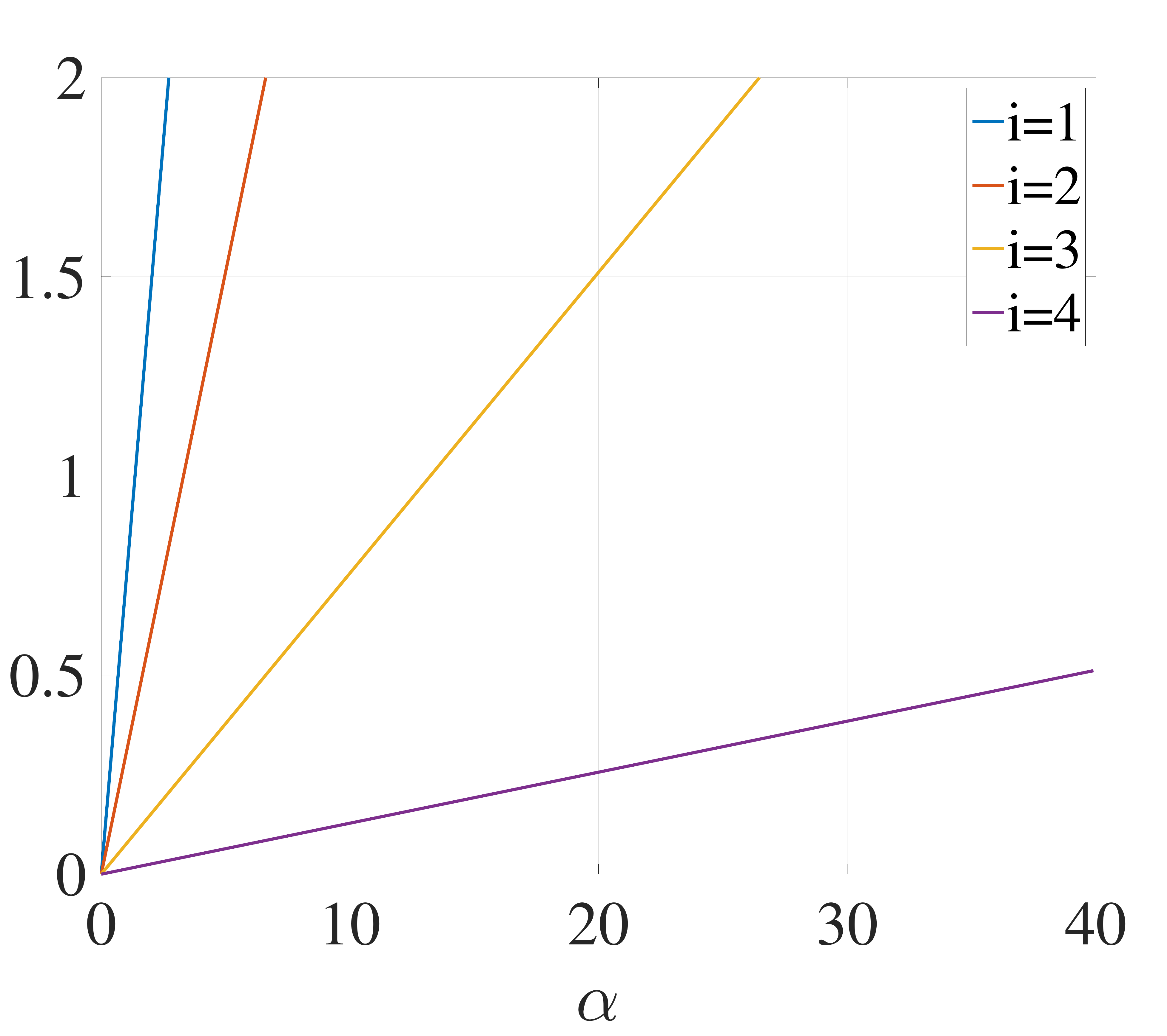}%
		\includegraphics[width=52mm]{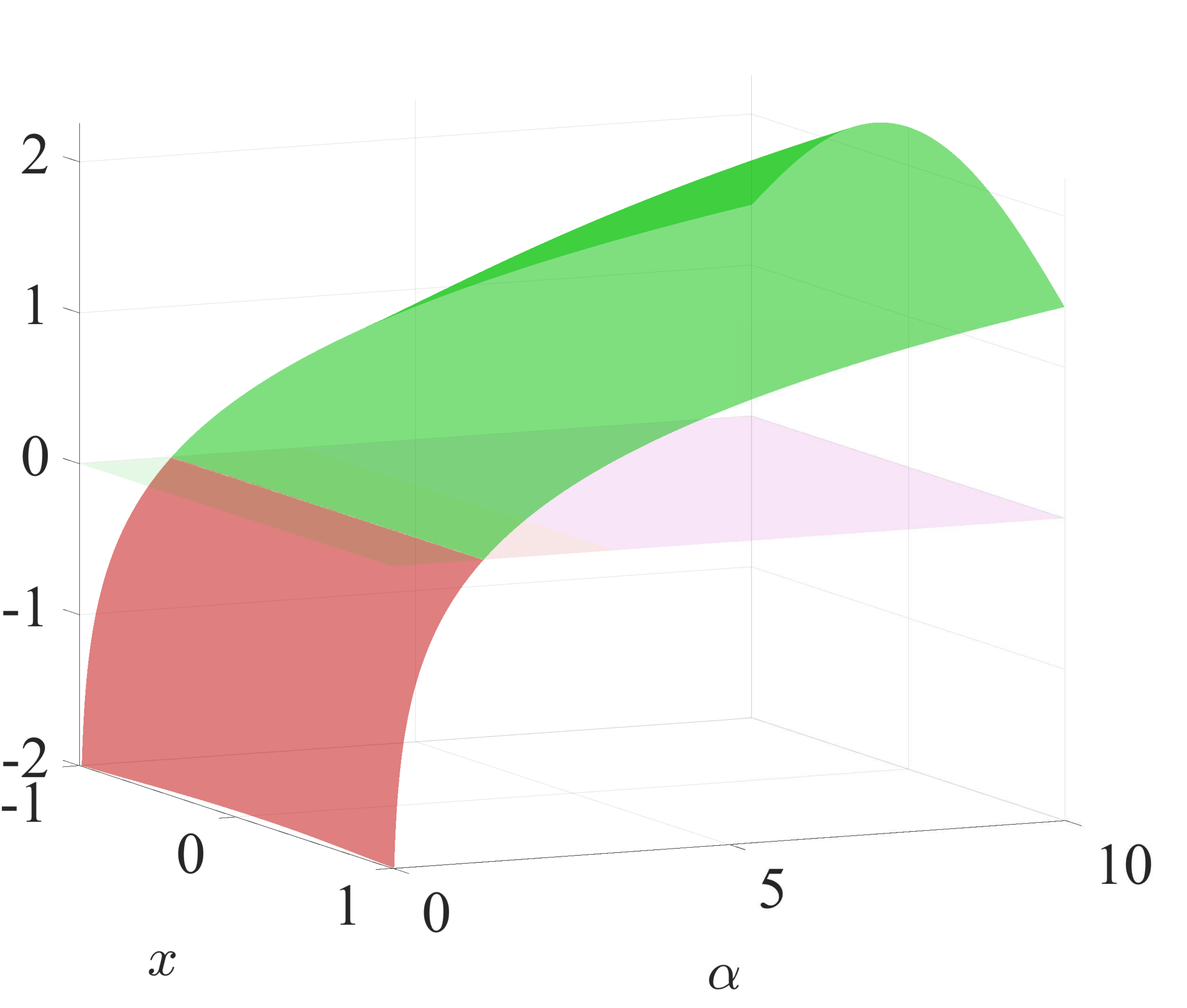}%
	\end{minipage}
		\caption{First critical values $\alpha_i^0$ for bifurcations from the trivial branch (left). 
		The $\alpha$-dependence of the corresponding eigenvalues $\lambda_{i-1}^0(\alpha)$ (center). 
		Primary transcritical bifurcation of the branch $\phi_0^0$ at $\alpha^0_0\approx 1.36$ for the Gau{\ss} kernel \eqref{kergauss} with $a=1$, $L=2$ (right)}
		\label{figgauss1}
\end{figure}

\begin{example}[Gau{\ss} kernel]
	Let $a,L>0$ and $\Omega=[-\tfrac{L}{2},\tfrac{L}{2}]$ be endowed with the Lebesgue measure. For the \emph{Gau{\ss} kernel}
	\begin{equation}
		k(x,y):=\tfrac{a}{\sqrt{\pi}}e^{-a^2(x-y)^2},
		\label{kergauss}
	\end{equation}
	we conclude total positivity from \cite[7.I(b)]{ando:87}. Hence, by \pref{prop_total-positivity} its eigenvalues
	$$
		0<\ldots<\lambda_i<\ldots<\lambda_1<\lambda_0
	$$
	form a strictly decreasing sequence and are all simple, while \cref{cor_eig-alt-even-odd} grants us that the eigenfunction $\xi^i$ corresponding to $\lambda_i$, $i\in\N_0$, alternates between being even or odd. Moreover, the resulting IDE \eqref{deq} becomes even. The above \pref{primary_positive} yields that the trivial branch loses stability to a transcritical bifurcation at the parameter $\alpha^0_0=\tfrac{1}{r(\sK)}>0$. Concerning the spectral radius, \cite[Thm.~6.1]{anselone:lee:74} yields the estimate
	$$
		\tfrac{1}{2}\erf(aL)\leq r(\sK)\leq 2\erf(\tfrac{aL}{2})
	$$
	with the \emph{error function} $\erf(x):=\tfrac{2}{\sqrt{\pi}}\int_0^xe^{-t^2}\d t$; this inclusion improves for small values of $aL>0$. At the countably many critical parameters $\alpha_i^0:=\tfrac{1}{\lambda_i}$, $i\in\N_0$, one has alternating transcritical and supercritical pitchfork bifurcations from the trivial solution (cf.~\fref{figgauss1}). The first such values $\alpha_i^0$ are listed in \fref{figgauss1} (left). However, as the eigenfunctions $\xi^i$ assume both negative and positive values for $i>0$, the bifurcations at these $\alpha_i^0$ do not lead to biologically relevant solutions, at least close to $(0,\alpha_i^0)$. Along the primary bifurcation branch $\phi_0^0$, which transcritically bifurcates from the trivial solution (see \fref{figgauss1} (right)), one detects a strictly increasing sequence $\alpha_i^1>\alpha_0^0$ of critical parameter values yielding supercritical period doubling bifurcations (see \fref{figgauss2} for the values (left) and the eigenvalue dependence on $\alpha$ (center)). These additional period doublings into unstable $2$-periodic solutions distinguishes the spatial Ricker model \eqref{noricker} from its scalar counterpart $u_{t+1}=\alpha u_te^{-u_t}$. In fact, there happens to be a period doubling cascade of positive functions as illustrated in \fref{figgauss2} (right). A schematic diagram containing several bifurcations for \eqref{deq} with right-hand side \eqref{noricker} is given in \fref{figschem2} (right). In particular, period doublings can be observed along any branch $\phi_i^0$, $i\in\N$.
	\begin{figure}
		\begin{minipage}{20mm}
		\begin{tabular}{@{}rr@{}}
			\hline
			$i$ & $\alpha_i^1$\\
			\hline
			$0$ & $10.32$\\
			$1$ & $31.67$\\
			$2$ & $6749.47$\\
			\hline
		\end{tabular}
		\end{minipage}
		\begin{minipage}{105mm}
			\includegraphics[width=52mm]{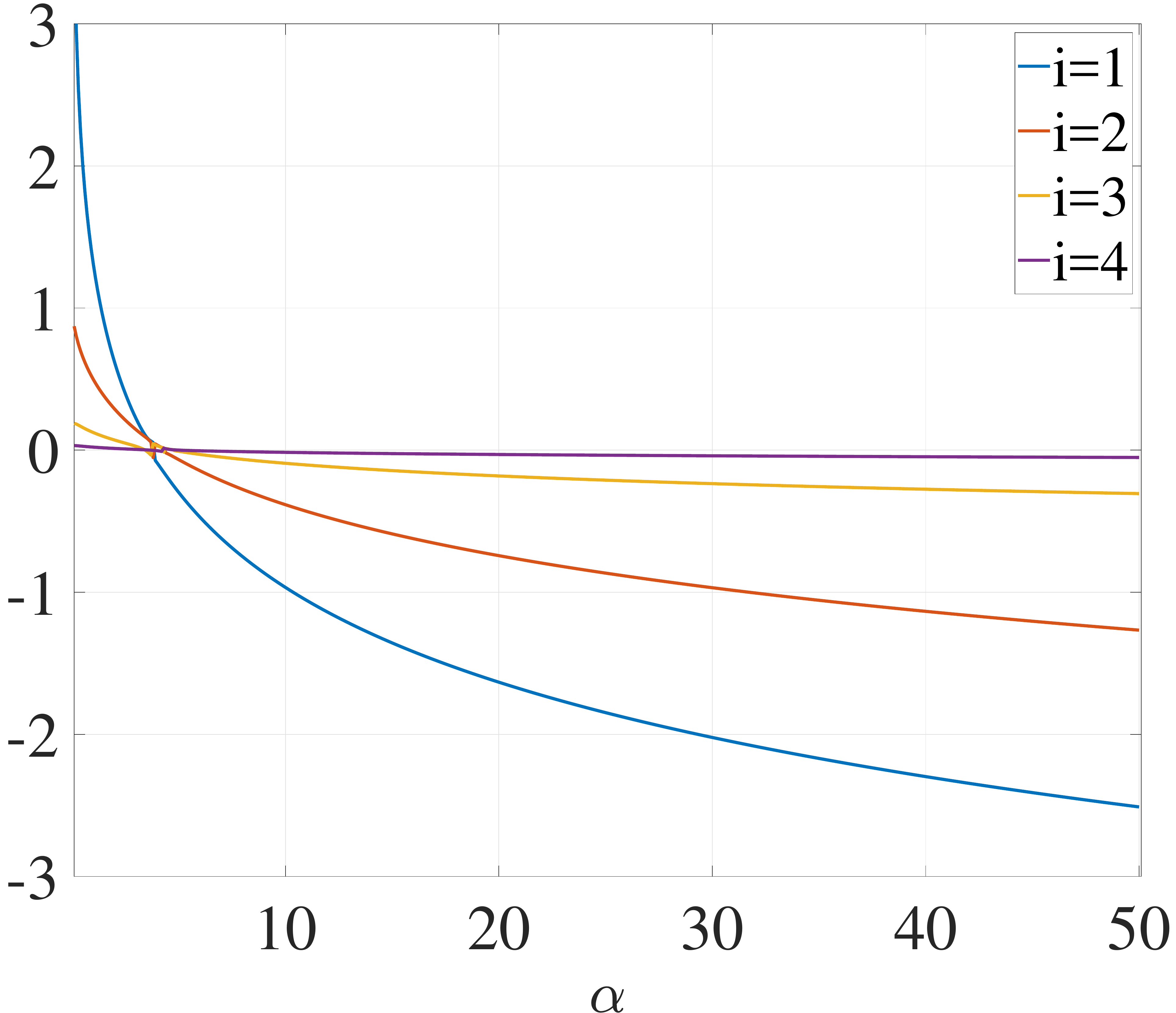}%
			\includegraphics[width=52mm]{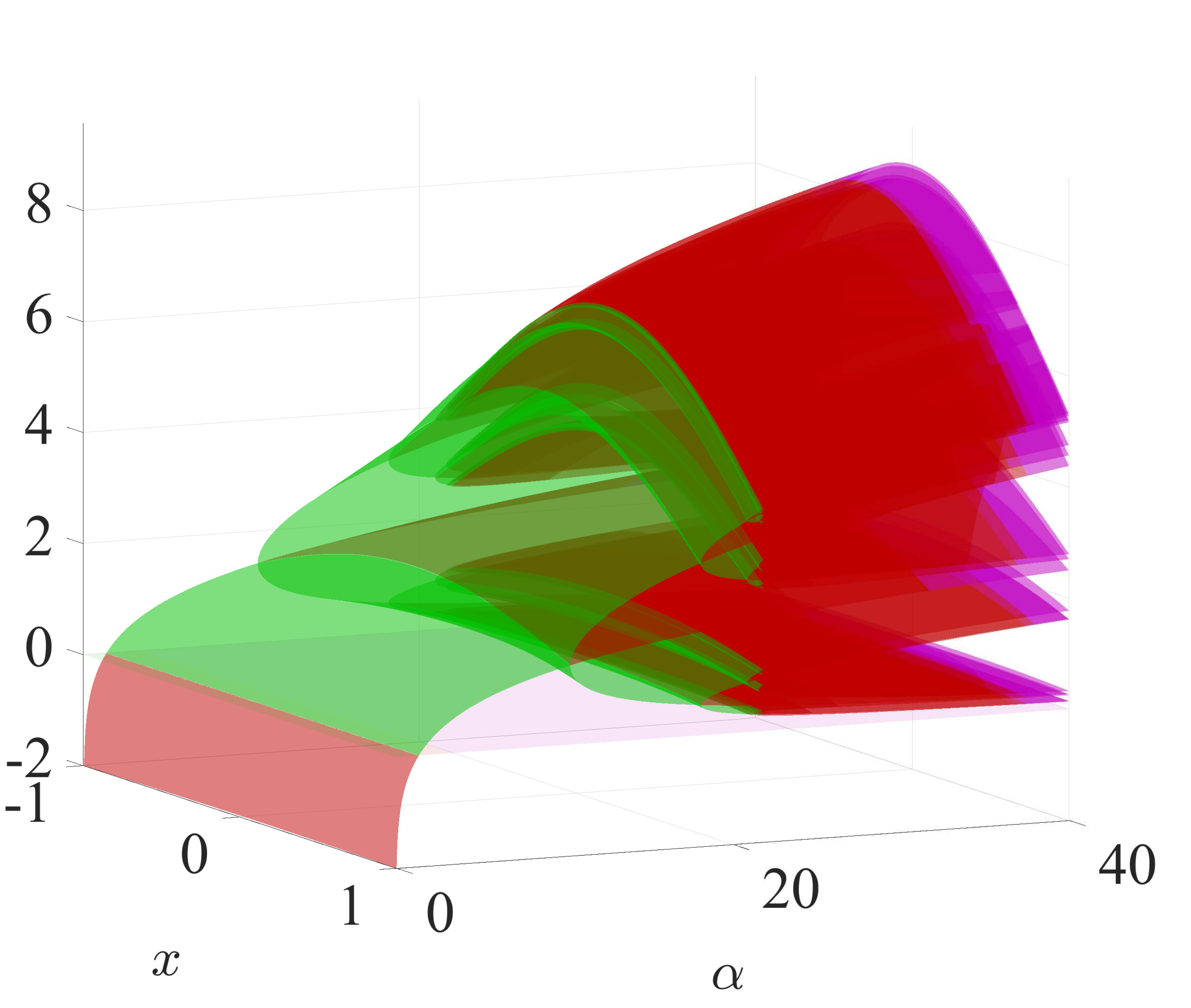}%flipricker0}%
		\end{minipage}
		\caption{First critical values $\alpha_i^1$ for bifurcations along the primary branch $\phi_0^0$ (left). 
		The $\alpha$-dependence of the corresponding eigenvalues $\lambda^1_{i-1}(\alpha)$ for $\alpha\geq 0.1$ (center).
		Period doubling cascade for the Ricker IDE with right-hand side \eqref{noricker} and $\alpha\in[0,40]$ for the Gau{\ss} kernel \eqref{kergauss} with $a=1$, $L=2$ (right)}
			\label{figgauss2}
	\end{figure}
\end{example}
\section{Concluding remarks and perspectives}
First, we did not tackle stability properties of periodic solutions $\phi^\ast$ at critical parameter values $\alpha^\ast$, i.e.\ at bifurcation points. Such results require reduction to a --- in our case --- $1$-dimensional $\theta$-periodic center manifold of $(\Delta_{\alpha^\ast})$, and we refer to the manuscript \cite{poetzsche:russ:19} having a particular focus on these issues, where the time-dependence can even be aperiodic. 

Second, this paper addressed simple Floquet multipliers crossing the stability boundary $\S^1$ at $\pm 1$ under parameter variation. The complementary situation of complex-conjugated pairs, leading to Neimark-Sacker bifurcations, is featured in the companion paper \cite{aarset:poetzsche:19}, where we provide conditions for bifurcations of discrete tori. 

Our results allow some immediate generalizations and transfers, which were neglected here for the sake of an accessible and compact presentation: 
\begin{itemize}
	\item Both the habitats $\Omega$ and the measures $\mu$ in \eqref{ury} can vary $\theta_0$-periodically in time, as long as they stay compact resp.\ finite. For general difference equations acting on the space of measures, and their applications in population dynamics, we refer to \cite{thieme:20}.

	\item Although the examples in \sref{sec5} were scalar $(d=1)$, our abstract results do apply to systems of IDEs like for instance predator-prey models \cite{kot:schaeffer:86,aarset:poetzsche:19}. Another field of applications is the analysis of models capturing an age- or size-structure of a population \cite{bacaer:09,bacaer:dats:12,cushing:ackleh,cushing:henson}, which are additionally equipped with a spatial component (cf.\ \cite{alzoubi:10,lutscher:lewis:04,robertson:cushing:12}). The specific algebraic form of such structured models might simplify our assumptions and their concrete verification. 

	\item Our tools are not restricted to IDEs in the space $C_d$ of $\R^d$-valued continuous functions on $\Omega$, but apply also to IDEs between $L^p(\Omega,\R^d)$ spaces, as
	$$
		\iprod{u,v}:=\int_\Omega\sprod{u(x),v(x)}\d\mu(x)\fall u\in L^{p'}(\Omega,\R^d),\,v\in L^p(\Omega,\R^d)
	$$
	induces a duality pairing $\iprod{L^{p'}(\Omega,\R^d),L^p(\Omega,\R^d)}$ with $p>1$, $\tfrac{1}{p}+\tfrac{1}{p'}=1$. Hence, only well-definedness and smoothness properties of the superposition and integral operators involved in \eqref{ury} remain to be verified. 
\end{itemize}
We point out that our assumptions require to evaluate integrals and to solve (systems of) integral equations. In case $\mu$ is the Lebesgue measure this necessitates numerical methods (e.g.\ \cite{atkinson:92,atkinson:97}, \cite[pp.~219ff]{kress:14}) to approximate integrals. The situation of a measure $\mu$ as in \eref{exmeasure}(2) allows a direct treatment and is therefore relevant in numerical discretizations, since integrals become weighted sums and the resulting linear equations can be solved exactly. 
\begin{appendix}
\renewcommand{\theequation}{\Alph{section}.\arabic{equation}}
\section{Abstract bifurcation results with one-dimensional kernel}
\label{appA}
We review the necessary machinery from local bifurcation theory for index $0$ Fredholm operators with one-dimensional kernel dating back to \cite{crandall:rabinowitz:71,crandall:rabinowitz:73}. While a fold bifurcation result can be found in standard references (cf.\ e.g.\ \cite{kielhoefer:12}), the subsequent \tref{thmcrosbif} due to \cite{liu:shi:wang:07} is of more recent origin. It does not rely on the knowledge of a given solution branch, but includes the classical transcritical and pitchfork bifurcations as special cases. Moreover, the behavior of critical eigenvalues along bifurcation branches is studied. 

Suppose throughout that $X,Z$ are real Banach spaces and $U\subseteq X$, $A\subseteq\R$ denote nonempty open neighborhoods of $u^\ast\in X$, $\alpha^\ast\in\R$ in the respective spaces. We deal with $C^m$-mappings $G:U\tm A\to Z$, $m\in\N$, vanishing at $(u^\ast,\alpha^\ast)$, i.e.\
\begin{equation}
	\boxed{G(u^\ast,\alpha^\ast)=0.}
	\label{zero}
\end{equation}
The pair $(u^\ast,\alpha^\ast)$ is called a \emph{bifurcation point} of the abstract equation 
\begin{equation}
	\tag{$O_\alpha$}
	\boxed{G(u,\alpha)=0,}
	\label{full}
\end{equation}
whenever there exists a parameter sequence $(\alpha_n)_{n\in\N}$ in $A$ having the limit $\alpha^\ast$ and distinct solutions $u_n^1,u_n^2\in U$, $n\in\N$, to $(O_{\alpha_n})$ with
$
	\lim_{n\to\infty}u_n^1
	=
	\lim_{n\to\infty}u_n^2
	=
	u^\ast.
$
Assume henceforth that $D_1G(u^\ast,\alpha^\ast)\in L(X,Z)$ is Fredholm of index $0$ and
\begin{equation}
	N(D_1G(u^\ast,\alpha^\ast))=\spann\set{\xi^\ast}
	\label{fred}
\end{equation}
holds for some $\xi^\ast\in X\setminus\set{0}$. There exists, for instance by the Hahn-Banach theorem, a continuous functional $z'\in Z'$ and a (fixed) closed subspace $X_1\subseteq X$ with 
\begin{align}
	N(z')&=R(D_1G(u^\ast,\alpha^\ast)),&
	X&=N(D_1G(u^\ast,\alpha^\ast))\oplus X_1.
	\label{nozero}
\end{align}

\begin{lemma}[{cf.~\cite[p.~38, Prop.~3.6.1]{buffoni:toland:12}}]\label{lemma:eigder}
	Suppose that $X\subseteq Z$ with the embedding operator $J$ satisfying
	\begin{equation}
		J\in L(X,Z). 
		\label{embed}
	\end{equation}
	If there is an open neighborhood $S\subseteq\R$ of $0$ and a $C^{m-1}$-mapping $\svector{\gamma}{\alpha}:S\to U\tm A$ satisfying $G(\gamma(s),\alpha(s))\equiv 0$ on $S$ and $	\svector{\gamma}{\alpha}(0)=\svector{u^\ast}{\alpha^\ast}$, then there exist a $\rho>0$ with $(-\rho,\rho)\subseteq S$ and a $C^{m-1}$-curve $\svector{\xi}{\lambda}:(-\rho,\rho)\to X\tm\C$ so that $\svector{\xi}{\lambda}(0)=\svector{\xi^\ast}{0}$ and 
	\begin{align}
		\tag{$E_\Gamma$}
		D_1G(\gamma(s),\alpha(s))\xi(s)&\equiv\lambda(s)J\xi(s),&
		z'(J\xi(s))&\equiv 1\on(-\rho,\rho)
		\label{eigen}
	\end{align}
	hold. Moreover, each $\lambda(s)\in\C$ is a simple eigenvalue of $D_1G(\gamma(s),\alpha(s))\in L(X,Z)$ and uniquely determined in a neighborhood of $0$. 
\end{lemma}

It is handy to abbreviate $G_{ij}:=D_1^iD_2^jG(u^\ast,\alpha^\ast)$ and $g_{ij}:=z'(G_{ij}(\xi^\ast)^i)$ for every $i,j\in\N_0$, $i+j\leq m$. We now formulate the basic bifurcation results:
\begin{theorem}[abstract fold bifurcation, {\cite[p.~12, Thm.~I.4.1 and p.~29, (I.7.30)]{kielhoefer:12}}]\label{asnbif}
	If \eqref{zero}, \eqref{fred} and $g_{01}\neq 0$ hold, then there exist $\eps>0$, open convex neighborhoods $S\subseteq\R$ of $0$, $A_0\subseteq A$ of $\alpha^\ast$ so that % and a $C^m$-function $\gamma=(\gamma_1,\gamma_2):S\to B_\eps(u^\ast)\tm A_0$ such that
	$
%		\gamma(S)=\set{(u,\alpha)\in B_\eps(u^\ast)\tm A_0:\,G(u,\alpha)=0},
		\set{(u,\alpha)\in B_\eps(u^\ast)\tm A_0:\,G(u,\alpha)=0} = \Gamma,
	$
%	where $\gamma$ satisfies $\gamma(0)=(u^\ast,\alpha^\ast)$ and $\dot\gamma(0)=(w,0)$. Moreover, in case $m\geq 2$ and
where $\Gamma=\svector{\gamma}{\alpha}(S)$ and $\svector{\gamma}{\alpha}:S\to B_\eps(u^\ast)\tm A_0$ is a $C^m$-curve satisfying 
	\begin{align*}
		\gamma(0)&=u^\ast,&
		\alpha(0)&=\alpha^\ast,&
		\dot\gamma(0)&=\xi^\ast,&
		\dot\alpha(0)&=0.
	\end{align*}
	Furthermore, in case $m\geq 2$ and $g_{20}\neq 0$, the pair $(u^\ast,\alpha^\ast)$ is a bifurcation point of \eqref{full}, one has $\ddot\alpha(0)=-\tfrac{g_{20}}{g_{01}}$ and
	\begin{enumerate}
		\item if $\frac{g_{20}}{g_{01}}<0$, then 
		$
			\#\set{u\in B_\eps(u^\ast):\,G(u,\alpha)=0}
			=
			\begin{cases}
				0,\quad\alpha<\alpha^\ast,\\
				1,\quad\alpha=\alpha^\ast,\\
				2,\quad\alpha>\alpha^\ast, 
			\end{cases}
		$

		\item if $\frac{g_{20}}{g_{01}}>0$, then 
		$
			\#\set{u\in B_\eps(u^\ast):\,G(u,\alpha)=0}
			=
			\begin{cases}
				0,\quad\alpha>\alpha^\ast,\\
				1,\quad\alpha=\alpha^\ast,\\
				2,\quad\alpha<\alpha^\ast,
			\end{cases}
		$

		\item if additionally \eqref{embed} holds, then the solution of \eqref{eigen} satisfies $\dot\lambda(0)=g_{20}$. 
	\end{enumerate}
\end{theorem}

The classical Crandall-Rabinowitz result, e.g.\ \cite[Thm.~1]{crandall:rabinowitz:71}, \cite[p.~18, Thm.~I.5.1]{kielhoefer:12}, is prototypical for further bifurcation phenomena. However, it requires a constant solution branch $G(u^\ast,\alpha)\equiv 0$ on $A$ or the a priori knowledge of a smooth solution branch. Here, we weaken this to the local assumption
\begin{equation}
	D_2G(u^\ast,\alpha^\ast)=0
	\label{a6cond}
\end{equation}
supplemented by a transversality condition, which guarantees that $G^{-1}(0)$ consists of two intersecting curves, one of them tangential to the $\alpha$-axis in $(u^\ast,\alpha^\ast)$: 
\begin{theorem}[abstract crossing curve bifurcation]\label{thmcrosbif}
	Let $m\geq 2$. If beyond \eqref{zero}, \eqref{fred} and \eqref{a6cond} also the \emph{transversality conditions}
	\begin{align*}
		g_{11}&\neq 0,&
		g_{02}&=0
%		\label{a7cond}
	\end{align*}
	hold, then there exist $\eps>0$, open convex neighborhoods $S\subseteq\R$ of $0$ and $A_0\subseteq A$ of $\alpha^\ast$ such that 
	$
		\set{(u,\alpha)\in B_\eps(u^\ast)\tm A_0:\,G(u,\alpha)=0}
		=
		\Gamma_1\cup\Gamma_2, 
	$
	where the branches $\Gamma_1=\svector{\gamma_1}{\alpha_1}(S)$, $\Gamma_2=\svector{\gamma_2}{\alpha_2}(S)$ have the following properties: 
	\begin{enumerate}
		\item $\svector{\gamma_1}{\alpha_1}:S\to B_\eps(u^\ast)\tm A_0$ is a $C^{m-1}$-curve satisfying 
		\begin{align}
			\gamma_1(0)&=u^\ast,&
			\alpha_1(s)&=\alpha^\ast+s,&
			\dot\gamma_1(0)&=0,
			\label{thmcrosbif3}
		\end{align}
		
		\item $\svector{\gamma_2}{\alpha_2}:S\to B_\eps(u^\ast)\tm A_0$ is a $C^{m-1}$-curve satisfying 
		\begin{align}
			\gamma_2(s)&=u^\ast+s\xi^\ast+\tilde\gamma_2(s),&
			\alpha_2(s)&=\alpha^\ast-\frac{g_{20}}{2g_{11}}s+\tilde\alpha_2(s)
			\label{thmcrosbif4}
		\end{align}
		and $\tilde\gamma_2:S\to X_1$, $\tilde\alpha_2:S\to\R$ are $C^{m-2}$-functions fulfilling
		\begin{align*}
			\tilde\gamma_2(0)&=\dot{\tilde\gamma}_2(0)=0,&
			\tilde\alpha_2(0)&=\dot{\tilde\alpha}_2(0)=0, 
		\end{align*}

		\item if additionally \eqref{embed} holds, then the solution of $(E_{\Gamma_1})$ satisfies $\dot\lambda(0)=g_{11}$. 
	\end{enumerate}
\end{theorem}
\begin{proof}
	Referring to \cite[Cor.~2.3]{liu:shi:wang:07}, it only remains to establish (c). For that purpose, we obtain from \lref{lemma:eigder} that the simple eigenvalue $0$ of $D_1G(u^\ast,\alpha^\ast)$ is embedded into a $C^1$-curve $\lambda:(-\rho,\rho)\to\C$ satisfying $(E_{\Gamma_1})$. If we differentiate $(E_{\Gamma_1})$ on the interval $(-\rho,\rho)$, then $z'(J\dot\xi(s))\equiv 0$ and
	\begin{align}
		\dot\lambda(s)J\xi(s)
		\equiv\,&
		D_1^2G(\gamma_1(s),\alpha_1(s))\dot\gamma_1(s)\xi(s)+
		\dot\alpha_1(s)D_1D_2G(\gamma_1(s),\alpha_1(s))\xi(s)
		\notag\\
		&
		+D_1G(\gamma_1(s),\alpha_1(s))\dot \xi(s)-\lambda(s)J\dot \xi(s)
		\label{bifident}
	\end{align}
	follows. Setting $s=0$ now implies $\dot\lambda(0)J\xi^\ast\stackrel{\eqref{thmcrosbif3}}{=}G_{11}\xi^\ast+	G_{10}\dot \xi(0)$. Applying the functional $z'$ in combination with \eqref{nozero} yields the desired $\dot\lambda(0)=z'(G_{11}\xi^\ast)$. %- J\dot\xi(0)
\end{proof}

Further information on the branch $\Gamma_2$ is provided next: 
\begin{corollary}[abstract transcritical bifurcation]\label{cortransbif}
	In case $g_{20}\neq 0$, then
	$$
		\#\set{u\in B_\eps(u^\ast):\,G(u,\alpha)=0}
		=
		\begin{cases}
			1,&\alpha=\alpha^\ast,\\
			2,&\alpha\neq\alpha^\ast.
		\end{cases}
	$$
	If additionally \eqref{embed} holds, then the solution of $(E_{\Gamma_2})$ satisfies $\dot\lambda(0)=\tfrac{g_{20}}{2}$. 
\end{corollary}
\begin{proof}
	From $z'(G_{20}(\xi^\ast)^2)=g_{20}\neq 0$ follows $\dot\alpha_2(0)\neq 0$, and the assertion on the local solution structure of \eqref{full} results with \tref{thmcrosbif}. As in the proof of \tref{thmcrosbif}, we obtain the identity \eqref{bifident} with $\Gamma_2$ instead of $\Gamma_1$, and for $s=0$ results
	$$
		\dot\lambda(0)J\xi^\ast
		\stackrel{\eqref{thmcrosbif4}}{=}
		G_{20}(\xi^\ast)^2-\tfrac{g_{20}}{2g_{11}}G_{11}\xi^\ast+G_{10}\dot \xi(0), \quad z'(J\dot\xi(0)) = 0 % - J\dot\xi(0)
	$$
	We apply the functional $z'$, and thus $\dot\lambda(0)=g_{20}-\tfrac{g_{20}}{2g_{11}}g_{11}=\tfrac{g_{20}}{2}$ holds. 
\end{proof}

We point out that the degenerate case $g_{20}=z'(G_{20}(\xi^\ast)^2)=0$ yields the inclusion $G_{20}(\xi^\ast)^2\in R(G_{10})$. Due to $X_1=N(G_{10})^\perp$, the linear-inhomogeneous equation 
\begin{equation}
	G_{10}\bar{\psi}+G_{20}(\xi^\ast)^2=0
	\label{lineqn}
\end{equation}
possesses a unique solution $\bar{\psi}\in X_1$, which is crucial for
\begin{corollary}[abstract pitchfork bifurcation]\label{corpitch}
	In case $m\geq 3$ and 
	\begin{align*}
		g_{20}&=0,&
		g_{30}+3z'(G_{20}\xi^\ast\bar{\psi})\neq 0
	\end{align*}
	hold, then one has $\ddot\alpha_2(0)=-\tfrac{g_{30}+3z'(G_{20}\xi^\ast\bar{\psi})}{3g_{11}}$, $\ddot\gamma_2(0) = \bar{\psi}$ and: 
	\begin{enumerate}
		\item If $\ddot\alpha_2(0)>0$, then 
		$
			\#\set{u\in B_\eps(u^\ast):\,G(u,\alpha)=0}
			=
			\begin{cases}
				1,&\alpha\leq\alpha^\ast,\\
				3,&\alpha>\alpha^\ast, 
			\end{cases}
		$

		\item if $\ddot\alpha_2(0)<0$, then 
		$
			\#\set{u\in B_\eps(u^\ast):\,G(u,\alpha)=0}
			=
			\begin{cases}
				3,&\alpha<\alpha^\ast,\\
				1,&\alpha\geq\alpha^\ast,
			\end{cases}
		$
		
		\item if additionally \eqref{embed} holds, then the solution of $(E_{\Gamma_2})$ satisfies $\dot\lambda(0)=0$ and $\ddot\lambda(0)=\tfrac{2}{3}g_{30}+2z'(G_{20}\xi^\ast\bar{\psi})$. 
	\end{enumerate}
\end{corollary}
\begin{proof}
	\tref{thmcrosbif} yields the claim except the following: 

	(I) The expression for $\ddot\alpha_2(0)$ is shown in \cite[(4.6)]{shi:99}. In order to compute $\ddot\gamma_2(0)$, let us differentiate $G(\gamma_2(s),\alpha_2(s))\equiv 0$ and $z'(\tilde\gamma_2(s))\equiv 0$ on $S$ twice in $s=0$; this identity holds due to the construction of the solution branch and $\tilde\gamma_2(s)\in R(G_{10})=N(z')$ for every $s\in S$. Indeed, $z'(\ddot{\tilde\gamma}_2(0)) = z'(\ddot\gamma_2(0)) = 0$ shows that $\ddot\gamma_2(0)\in R(G_{10})$, while $G_{20}(\xi^\ast)^2+G_{10}\ddot\gamma_2(0)=0$ holds and implies $\ddot\gamma_2(0) = \bar{\psi}$, as $\bar{\psi}$ is the unique solution of \eqref{lineqn} in $R(G_{10})$. 

	(II) It results from \cref{cortransbif} that $\dot\lambda(0)=0$. Concerning the second derivative, we differentiate $(E_{\Gamma_2})$ twice in $s=0$ and apply the functional $z'$ in order to arrive at
	\begin{multline*}
		\ddot\lambda(0)
		=
		z'(G_{30}\dot\gamma_2(0)\dot\gamma_2(0)\xi^\ast) + 2\dot\alpha_2(0)z'(G_{21}\dot\gamma_2(0)\xi^\ast) + z'(G_{20}\ddot\gamma_2(0)\xi^\ast) \\
		+ 3z'(G_{20}\dot\gamma_2(0) \dot \xi(0)) + \dot\alpha_2(0)^2z'(G_{12}\xi^\ast) + 2\dot\alpha_2(0)z'(G_{11}\dot \xi(0))
		+ \ddot\alpha_2(0)z'(G_{11}\xi^\ast),
	\end{multline*}
	making use of $z'(J\dot\xi(0)) = z'(J\ddot\xi(0)) = 0$. First, as in \cite[p.~28, (I.7.27)]{kielhoefer:12}, one shows that $\dot \xi(0)=\ddot\gamma_2(0)=\bar{\psi}$, and second, using the above expressions for $\ddot\gamma_2$, $\ddot\alpha_2$ and \eqref{thmcrosbif4} together with $g_{20}=0$ implying $\dot\alpha_2(0) = 0$ yields
	\begin{align*}
		\ddot\lambda(0)
		= 
		g_{30} + 3z'(G_{20}\xi^\ast \bar{\psi} ) 
		-\tfrac{g_{30}+3z'(G_{20}\xi^\ast\bar{\psi} )}{3g_{11}}g_{11}
		=
		\tfrac{2}{3}g_{30}+2z'(D_1^2G(u^\ast,\alpha^\ast)\xi^\ast\bar{\psi}),
	\end{align*}
	as desired. 
\end{proof}
\section{Nystr\"om methods and numerical algorithms}
\label{appB}
Our simulations in \sref{sec5} rely on \emph{Nystr\"om discretizations} of IDEs over compact sets $\Omega\subset\R^\kappa$. This classical approach (see e.g.\ \cite[pp.~100ff]{atkinson:97}, \cite[pp.~219ff]{kress:14}) approximates integrals by quadrature (in dimension $\kappa=1$) or cubature ($\kappa>1$) formulas
\begin{equation}
	\int_\Omega u(y)\d y
	=
	\sum_{j=1}^{N_n}w_ju(\eta_j)+e_n(u)
	\label{noB1}
\end{equation}
with \emph{weights} $w_j>0$, \emph{nodes} $\eta_j\in\Omega$ and an error term $e_n(u)=O(\tfrac{1}{n^r})$ as $n\to\infty$, where $r>0$ denotes the \emph{convergence rate} of the method. Higher order methods require smooth integrands $u$. Simple examples of composite quadrature formulas over $\Omega=[a,b]$ are given in Tab.~\ref{tabquad} (cf.~\cite[pp.~361ff]{engeln:uhlig:96}). Furthermore, we refer to \cite[pp.~406ff]{engeln:uhlig:96} for various cubature rules over different domains in $\R^\kappa$, $\kappa>1$. 
\begin{table}
	\begin{tabular}{r|cccc}
		rule & nodes $\eta_j$ & weights $w_j$ & $r$ & $N_n$ \\
		\hline
		midpoint & $a+h(j-\tfrac{1}{2})$ & $h$ & $2$ & $n$ \\
		\multirow{2}{*}{trapezoidal} & \multirow{2}{*}{$a+h(j-1)$} & $\tfrac{h}{2}\text{ for }j\in\set{1,n+1}$ & $2$ & $n+1$\\
		& & $h$\text{ else} & & \\
		\multirow{2}{*}{Chebyshev} & $a+(j-\tfrac{\sqrt{3}+1}{2\sqrt{3}})h$ for $j\leq n$ & \multirow{2}{*}{$\tfrac{h}{2}$} & \multirow{2}{*}{$4$} & \multirow{2}{*}{$2n$}\\
		& $a+(j-n+\tfrac{\sqrt{3}-1}{2\sqrt{3}})h$ for $n<j$ & & &
	\end{tabular}
	\caption{Quadrature rules \eqref{noB1} with $h:=\tfrac{b-a}{n}$}
	\label{tabquad}
\end{table}

The IDEs \eqref{deq} are based on abstract integrals with finite measures $\mu$ (cf.~\eqref{ury}). The Lebesgue measure $\mu=\lambda_\kappa$ over $\Omega$ is typically used in models \cite{kot:schaeffer:86,lutscher:petrovskii:08,lutscher:19,reimer:bonsall:maini:16,kirk:lewis:97}, but also the measure $\mu_n(\Omega_n):=\sum_{\eta\in\Omega_n}w_{\eta}$ on finite subsets $\Omega_n\subset\Omega$ consisting of the nodes $\eta_j$ in \eqref{noB1} fits into our framework (see \eref{exmeasure}(2)). Thus, our bifurcation theory from \sref{sec4} applies to Nystr\"om discretizations as well by choosing the measure $\mu_n$ in \eqref{ury}. 

In doing so, for implementations we represent functions $u:\Omega_n\to\R$ as vectors $\upsilon\in\R^{N_n}$ via $\upsilon(j):=u(\eta_j)$ for $1\leq j\leq N_n$. Then the bilinear form \eqref{dpair} becomes
$$
	\iprod{\upsilon,\bar\upsilon}_n=\sum_{j=1}^{N_n}w_j\upsilon(j)\bar\upsilon(j)
	\fall\upsilon,\bar\upsilon\in\R^{N_n}, 
$$
the right-hand sides \eqref{ury} turn into
\begin{align*}
	\sF_t^n:\R^{N_n}\tm A&\to\R^{N_n},&
	\sF_t^n(\upsilon,\alpha)
	&=
	\intoo{G_t\intoo{\eta_i,\sum_{j=1}^{N_n}w_jf_t(\eta_i,\eta_j,\upsilon(j),\alpha),\alpha}}_{i=1}^{N_n} 
\end{align*}
and conveniently abbreviating $H_i^n:=D_2G_t\Bigl(\eta_i,\sum_{j=1}^{N_n}w_jf_t(\eta_i,\eta_j,\phi^\ast(j),\alpha),\alpha\Bigr)$, the $N_n\tm N_n$-matrix acting as
$$
	D_1\sF_t^n(\phi^\ast,\alpha)\upsilon
	=
	\left(H_i^n
	\sum_{j=1}^{N_n}w_jD_3f_t(\eta_i,\eta_j,\phi^\ast(j),\alpha)\upsilon(j)
	\right)_{i=1}^{N_n}
	\fall\upsilon\in\R^{N_n}
$$
is the derivative \eqref{derf10} in case $\mu=\mu_n$. 

Given this, we employ the $\theta$-periodic difference equation
\begin{equation}
	\boxed{\upsilon_{t+1}=\sF_t^n(\upsilon_t,\alpha)}
	\label{deqn}
\end{equation}
in $\R^{N_n}$ to approximate solutions $\phi_t\in C(\Omega)$ of \eqref{deq} equipped with the Lebesgue measure $\mu=\lambda_\kappa$, in terms of $\upsilon_t(j)\approx\phi_t(\eta_j)$. Then, for a fixed parameter $\alpha\in A$, 
\begin{itemize}
	\item $\theta$-periodic solutions $\phi^\ast$ of \eqref{deq} are the zeros of the nonlinear operator $G(\cdot,\alpha)$ from \eqref{Gdef}, where we replace $\sF_t$ by $\sF_t^n$. This requires to solve a nonlinear equation in $\R^{\theta N_n}$, for which we use the Newton solver \texttt{nsoli} from \cite{kelley:03}. Thanks to \cite{weiss:74}, this process preserves the convergence rate of quadrature schemes, i.e.\ one obtains $\max_{j=1}^{N_n}\abs{\upsilon_t(j)-\phi_t^\ast(\eta_j)}=O(\tfrac{1}{n^r})$. 

	\item The computation of Floquet multipliers to periodic solutions is based on the derivatives $D_1G(\phi^\ast,\alpha)$ given in \eqref{propG1}. Suitable methods for the resulting cyclic block matrices in $\R^{\theta N_n\tm\theta N_n}$ are described in \cite[p.~291ff, Chap.~8]{watkins:07}. Furthermore, convergence results for the eigenpairs when replacing the Fr{\'e}chet derivatives $D_1\sF_t$ by $D_1\sF_t^n$ are due to \cite{atkinson:75}.
\end{itemize}
Our numerical bifurcation analysis itself requires to detect and to continue solution branches along the real bifurcation parameter $\alpha\in A$. We suggest to achieve this using the pseudo-code from Alg.~\ref{algcont}, where $\sprod{\cdot,\cdot}$ denotes the Euclidean inner product. 
\begin{algorithm}[h]
 \KwData{$\alpha_0$ (initial parameter), $h>0$ (step size), $k_{\max}>0$ (number of steps)}
 \KwResult{Solution branch $(\hat\phi(\alpha_k),\alpha_k)$ over parameter range}
 Solve $G(\hat\phi_0,\alpha_0)=0$ for $\hat\phi_0$ using Newton iteration\;
 Set $(\hat\phi,\alpha)\leftarrow(\hat\phi_0,\alpha_0)$\;
 Solve $DG(\bar\phi_0)\binom{z_0}{\delta_0}=0$, $\sprod{z_0,z_0}+\delta_0^2=1$ for $\binom{z_0}{\delta_0}$ using Newton iteration\;
  \If{$\delta_0<0$}{
	$z_0\leftarrow -z_0$;\, $\delta_0\leftarrow -\delta_0$
   }
 \For{$k\leftarrow 0$ \KwTo $k_{\max}$}{
  Set $(\hat\phi_{k+1}',\alpha_{k+1}')\leftarrow(\hat\phi_k,\alpha_k)+h(z_k,\delta_k)$\;
  Use initial value $(\phi_{k+1}',\alpha_{k+1}')$ to solve $G(\hat\phi_{k+1},\alpha_{k+1})=0$, $\sprod{z_k,\hat\phi_{k+1}-\hat\phi_{k+1}'}+\delta_k(\alpha_{k+1}-\alpha_{k+1}')=0$ for $(\hat\phi_{k+1},\alpha_{k+1})$ using Newton iteration\;
  Solve $DG(\hat\phi_{k+1},\alpha_{k+1})\binom{z_{k+1}}{\delta_{k+1}}=0$, $\sprod{z_k,z_{k+1}}+\delta_k\delta_{k+1}=1$ for $(z_{k+1},\delta_{k+1})$\;
\If{$\delta_k\delta_{k+1}<0$}{
	fold detected\\
	determine critical pair $(\hat\phi^\ast,\alpha^\ast)$
   }
\If{ eigenvalue of $D_1G(\hat\phi_{k+1},\alpha_{k+1})$ crosses stability boundary}{
	determine critical pair $(\hat\phi^\ast,\alpha^\ast)$\\
	check and classify bifurcation using \sref{sec4}
   } }
 \caption{Pseudo-arc length continuation to compute a branch $(\hat\phi(\alpha_k),\alpha_k)$, $0\leq k\leq k_{\max}$, of $\theta$-periodic solutions to \eqref{deqn} with initially increasing $\alpha_k$}
 \label{algcont}
\end{algorithm}
\end{appendix}
\section*{Acknowledgments}
We would like to thank the referees for their perspective on our manuscript from different angles and the constructive criticism. Their opinion, hints and remarks greatly improved our presentation, scope and, after all, the overall quality of the paper.
\providecommand{\bysame}{\leavevmode\hbox to3em{\hrulefill}\thinspace}
\providecommand{\MR}{\relax\ifhmode\unskip\space\fi MR }
\providecommand{\MRhref}[2]{%
	\href{http://www.ams.org/mathscinet-getitem?mr=#1}{#2}
}
\providecommand{\href}[2]{#2}


\begin{thebibliography}{10}
\bibitem{aarset:poetzsche:19} C.\ Aarset, C.\ P\"otzsche, 
	\emph{Bifurcations in periodic integrodifference equations in $C(\Omega)$ II: Discrete torus bifurcations}, Commun.\ Pure Appl.\ Anal.~\textbf{19(4)} (2020), 1847--1874

\bibitem{alzoubi:10} M.Y.M.\ Alzoubi, 
	\emph{The net reproductive number and bifurcation in an integro-difference system of equations}, 
	Appl.\ Math.\ Sci.\ \textbf{4(1--4)} (2010), 191--200

\bibitem{amann:90} H.~Amann, 
	\emph{Ordinary Differential Equations: An Introduction to Nonlinear Analysis}, Studies in Mathematics~13, Walter de Gruyter, Berlin-New York, 1990

\bibitem{andersen:91} M.~Andersen, 
	\emph{Properties of some density-dependent integrodifference equation population models}, 
	Math.\ Biosci.~\textbf{104} (1991), 135--157 

\bibitem{ando:87} T.~Ando, 
	\emph{Totally positive matrices}, 
	Linear Algebra Appl.~\textbf{90} (1987), 165--219

\bibitem{anselone:lee:74} P.~Anselone, J.~Lee, 
	\emph{Spectral properties of integral operators with nonnegative kernels}, 
	Linear Algebra Appl.~\textbf{9} (1974), 67--87

\bibitem{atkinson:75} K.~Atkinson, 
	\emph{Convergence rates for approximate eigenvalues of compact integral operators}, 
	SIAM J.\ Numer.\ Anal.~\textbf{12(2)} (1975), 213--222

\bibitem{atkinson:92} K.~Atkinson, 
	\emph{A survey of numerical methods for solving nonlinear integral equations}, 
	J.\ Integr.\ Equat.\ Appl.~\textbf{4(1)} (1992), 15--46

\bibitem{atkinson:97} K.~Atkinson, 
	\emph{The Numerical Solution of Integral Equations of the Second Kind}, 
	Monographs on Applied and Comp.\ Mathematics 4, University Press, Cambridge, 1997

\bibitem{bacaer:09} N.~Baca{\"e}r, 
	\emph{Periodic matrix population models: Growth rate, basic reproduction number, and entropy}, 
	Bull.\ Math.\ Biol.~\textbf{71} (2009), 1781--1792

\bibitem{bacaer:dats:12} N.~Baca{\"e}r, E.H.\ Ait Dads, 
	\emph{On the biological interpretation of a definition for the parameter $R_0$ in periodic population models}, 
	J.\ Math.\ Biol.~\textbf{65} (2012), 601--621

\bibitem{beyn:huels:samtenschnieder:08} W.-J. Beyn, T.~H{\"u}ls, M.-C. Samtenschnieder.
	\emph{On $r$-periodic orbits of $k$-periodic maps}, 
	J.\ Difference Equ.\ Appl.~\textbf{14(8)} (2008), 865--887

\bibitem{bramburger:lutscher:18} J.\ Bramburger, F.\ Lutscher, 
	\emph{Analysis of integrodifference equations with a separable dispersal kernel}, 
	Acta Applicandae Mathematicae~\textbf{161} (2019), 127--151

\bibitem{brauer:castillo:01} F.~Brauer, C.~Castillo-Chavez, 
	\emph{Mathematical Models in Population Biology and Epidemiology}, 
	Texts in Applied Mathematics~40, Springer, Berlin etc., 2001

\bibitem{buffoni:toland:12} B.~Buffoni, J.F.\ Toland, 
	\emph{Analytic Theory of Global Bifurcation: An Introduction}, 
	University Press, Princeton NJ, 2003

\bibitem{cohn:80} D.~Cohn, 
	{\em Measure Theory}, 
	Birkh{\"a}user, Boston etc., 1980

\bibitem{cushing:ackleh} J.M.~Cushing, A.S.~Ackleh, 
	\emph{A net reproductive number for periodic matrix models}, 
	J.\ Biol.\ Dyn.~\textbf{6} (2012), 166--188

\bibitem{cushing:henson} J.M.~Cushing, S.M.~Henson, 
	\emph{Periodic matrix models for seasonal dynamics of structured populations with application to a seabird population}, 
	J.\ Math.\ Biol.~\textbf{77} (2018), 1689--1720

\bibitem{crandall:rabinowitz:71} M.~Crandall, P.~Rabinowitz, 
	\emph{Bifurcation from simple eigenvalues}, 
	J.\ Funct.\ Anal.~\textbf{8} (1971), 321--340

\bibitem{crandall:rabinowitz:73} \bysame, 
	\emph{Bifurcation, perturbation of simple eigenvalues and linearized stability}, 
	Arch.\ Ration.\ Mech.\ Anal.~\textbf{52} (1973), 161--180

\bibitem{day:junge:mischaikow:04} S.~Day, O.~Junge, K.~Mischaikow, 
	\emph{A rigerous numerical method for the global dynamics of infinite-dimensional discrete dynamical systems}, 
	SIAM J.\ Appl.\ Dyn.\ Syst.~\textbf{3(2)} (2004), 117--160

\bibitem{deimling:85} K.~Deimling, 
	\emph{Nonlinear Functional Analysis}, 
	Springer, Berlin etc., 1985

\bibitem{engeln:uhlig:96} G.~Engeln-M{\"u}llges, F.~Uhlig, 
	\emph{Numerical Algorithms with C}, 
	Springer, Berlin etc., 1996 

\bibitem{gyori:pituk:01} I.~Gy{\H{o}}ri, M.~Pituk, 
	\emph{The converse of the theorem on stability by the first approximation for difference equations}, 
	Nonlin.\ Analysis (TMA)~\textbf{47} (2001), 4635--4640

\bibitem{hardin:takac:webb:88} D.P.~Hardin, P.~Tak{\'a}{\v c}, G.F.~Webb, 
	\emph{A comparison of dispersal strategies for survival of spatially heterogeneous populations}, 
	SIAM J.\ Appl.\ Math.~\textbf{48(6)} (1988), 1396--1423

\bibitem{hardin:takac:webb:90} \bysame, 
	\emph{Dispersion population models discrete in time and continuous in space}, 
	J.\ Math.\ Biol.~\textbf{28} (1990), 1--20

\bibitem{iooss:79} G.~Iooss, 
	\emph{Bifurcation of Maps and Applications}, 
	Mathematics Studies~36, North-Holland, Amsterdam etc., 1979

\bibitem{jacobsen:mcadam:14} J.~Jacobsen, T.~McAdam,
	\emph{A boundary value problem for integrodifference population models with cyclic kernels}, 
	Discrete Contin.\ Dyn.\ Syst.\ (Series B)~\textbf{19(10)} (2014), 3191--3207

\bibitem{jin:thieme:16} W.~Jin, H.R.~Thieme, 
	\emph{An extinction/persistence threshold for sexually reproducing populations: The cone spectral radius}, 
	Discrete Contin.\ Dyn.\ Syst.\ (Series B)~\textbf{21(2)} (2016), 447--470

\bibitem{kato:80} T.~Kato, 
	{\em Perturbation Theory for Linear Operators} (corrected 2nd ed.), 
	Grund\-lehren der mathematischen Wissenschaften 132, Springer, Berlin etc., 1980

\bibitem{kelley:03} C.~Kelley, 
	\emph{Solving Nonlinear Equations with Newton's Method}, 
	Fundamentals of Algorithms~1, SIAM, Philadelphia, PA, 2003

\bibitem{kielhoefer:12} H.~Kielh{\"o}fer, 
	\emph{Bifurcation Theory: An Introduction with Applications to PDEs} ($2$nd ed.), Applied Mathematical Sciences~156, Springer, New York etc., 2012

\bibitem{kot:schaeffer:86} M.~Kot, W.M.~Schaffer, 
	\emph{Discrete-time growth-dispersal models}, 
	Math.\ Biosci.~\textbf{80} (1986), 109--136

\bibitem{krause:15} U.~Krause, 
	\emph{Positive dynamical systems in discrete time}, 
	Studies in Mathematics~62, de Gruyter, Berlin etc., 2015

\bibitem{kress:14} R.~Kress, 
	\emph{Linear Integral Equations} ($3$rd ed.), 
	Applied Mathematical Sciences~82, Springer, New York etc., 2014

\bibitem{liu:shi:wang:07} P.~Liu, J.~Shi, Y.~Wang, 
	\emph{Imperfect transcritical and pitchfork bifurcations}, 
	J.\ Funct.\ Anal.~\textbf{251} (2007), 573--600

\bibitem{luis:elaydi:oliveira:12} R.~Lu{\'\i}s, S.~Elaydi, H.~Oliveira, 
	\emph{Local bifurcation in one-dimensional nonautonomous periodic difference equations}, 
	Int.\ J.\ Bifurcation Chaos \textbf{23(3)}, 2013

\bibitem{lutscher:lewis:04} F.~Lutscher and M.A.~Lewis, 
	\emph{Spatially-explicit matrix models}, 
	J.\ Math.\ Biol.\ \textbf{48} (2004), 293--324

\bibitem{lutscher:petrovskii:08} F.~Lutscher, S.~Petrovskii, 
	\emph{The importance of census times in discrete-dime growth-dispersal models}, 
	J.\ Biol.\ Dynamics~\textbf{2(1)} (2008), 55--63

\bibitem{lutscher:19} F.~Lutscher, 
	\emph{Integrodifference equations in spatial ecology}, 
	Interdisciplinary Applied Mathematics~49, Springer, Cham, 2019

\bibitem{martin:76} R.H.~Martin, 
	\emph{Nonlinear Operators and Differential Equations in Banach Spaces}, 
	Pure and Applied Mathematics~11, John Wiley \& Sons, Chichester etc., 1976

\bibitem{pinkus:96} A.~Pinkus,
	\emph{Spectral properties of totally positive kernels and matrices},
	in \emph{Total Positivity and Its Applications} (M.~Gasca et al., eds.), 
	Mathematics and Its Applications~359, Kluwer, Dordrecht (1996), 477--511 

\bibitem{poetzsche:12} C.~P{\"o}tzsche, 
	\emph{Bifurcations in a periodic discrete-time environment}, 
	Nonlinear Analysis: Real World Applications \textbf{14} (2013), 53--82

\bibitem{poetzsche:18a} \bysame, 
	\emph{Numerical dynamics of integrodifference equations: Basics and discretization errors in a $C^0$-setting}, 
	Appl.\ Math.\ Comput.~\textbf{354} (2019), 422--443

\bibitem{poetzsche:russ:19} C.~P{\"o}tzsche, E.\ Ru{\ss}, 
	\emph{Reduction principle for nonautonomous integrodifference equations at work}, 
	Preprint (2020)

\bibitem{rana:02} I.K.~Rana,
	\emph{An Introduction to Measure and Integration} ($2$nd ed.),
	Graduate Studies in Mathematics~45, American Mathematical Society, Providence RI, 2002

\bibitem{reimer:bonsall:maini:16} J.R.~Reimer, M.B.~Bonsall, P.K.~Maini, 
	\emph{Approximating the critical domain size of integrodifference equations}, 
	Bull.\ Math.\ Biol.~\textbf{78} (2016), 72--109

\bibitem{robertson:cushing:12} S.L.~Robertson, J.M.\ Cushing, 
	\emph{A bifurcation analysis of stage-structured density dependent integrodifference equations}, 
	J.\ Math.\ Anal.\ Appl.~\textbf{388(1)} (2012), 490--499

\bibitem{shi:99} J.~Shi, 
	\emph{Persistence and bifurcation of degenerate solutions}, 
	J.\ Funct.\ Anal.~\textbf{169} (1999), 494--531

\bibitem{slatkin:73} M.~Slatkin, 
	\emph{Gene flow and selection in a cline}, 
	Genetics~\textbf{75} (1973), 733--756 

\bibitem{thieme:79} H.R.~Thieme, 
	\emph{On a class of Hammerstein integral equations}, 
	Manuscripta Math.~\textbf{29} (1979), 49--84

\bibitem{thieme:20} \bysame, 
	\emph{Discrete time population dynamics on the state space of measures}, 
	Math.\ Biosci.\ Engin.\ \textbf{17} (2020), 1168--1217

\bibitem{kirk:lewis:97} R.W.~Van Kirk, M.A.~Lewis, 
	\emph{Integrodifference models for persistence in fragmented habitats}, 
	Bull.\ Math.\ Biol.~\textbf{59(1)} (1997), 107--137

\bibitem{watkins:07} D.S.~Watkins, 
	{\em The Matrix Eigenvalue Problem --- GR and Krylov Subspace Methods}, 
	SIAM, Philadelphia, PA, 2007 
	
\bibitem{weiss:74} R.~Weiss, 
	\emph{On the approximation of fixed points of nonlinear compact operators}, 
	SIAM J.\ Numer.\ Anal.~\textbf{11(3)} (1974), 550--553

\bibitem{zeidler:95} E.~Zeidler, 
	\emph{Applied Functional Analysis: Main Principles and their Applications}, 
	Applied Mathematical Sciences~109, Springer, Heidelberg, 1995

\bibitem{zhao:17} X.-Q.~Zhao, 
	\emph{Dynamical systems in population biology} ($2$nd ed.), 
	CMS Books in Mathematics, Springer, Cham, 2017

\bibitem{zhou:fagan:17} Y.~Zhou, W.F.~Fagan, 
	\emph{A discrete-time model for population persistence in habitats with time-varying sizes}, 
	J.\ Math.\ Biol.~\textbf{75} (2017), 649--704
\end{thebibliography}
\end{document}